\pdfoutput=1 
\documentclass{article}
\usepackage{amsmath,amssymb,amsthm}
\usepackage{mathrsfs}
\usepackage[colorlinks=true]{hyperref}

\def\R{\mathbb{R}}
\def\N{\mathbb{N}}

\def\Z{\mathbb{Z}}
\def\C{\mathbb{C}} 
\def\H{\mathbb{H}} 

\newcommand{\Op}{\mathrm{Op}}
\newcommand{\OpW}{\mathrm{Op^{\mathcal{W}}}}

\renewcommand{\geq}{\geqslant}
\renewcommand{\leq}{\leqslant}

\newtheorem{theorem}{Theorem}[section]

\newtheorem{proposition}{Proposition}[section]
\newtheorem{corollary}{Corollary}[section]
\newtheorem{definition}{Definition}[section]
\newtheorem{lemma}{Lemma}[section]
\theoremstyle{definition}
\theoremstyle{definition}\newtheorem{remark}{Remark}[section]

\newtheorem{innercustomthm}{Theorem}
\newenvironment{customthm}[1]
{\renewcommand\theinnercustomthm{#1}\innercustomthm}
{\endinnercustomthm}

\title{Spectral asymptotics for sub-Riemannian Laplacians. I: \\ quantum ergodicity and quantum limits \\ in the 3D contact case.\\
\textit{\small To the memory of Louis Boutet de Monvel (1941--2014)} }

\author{Yves Colin de Verdi\`ere\footnote{Universit\'e de Grenoble-Alpes,
Institut Fourier, Unit{\'e} mixte de recherche CNRS-UJF 5582, BP 74, 38402-Saint Martin d'H\`eres Cedex, France 
(\texttt{yves.colin-de-verdiere@univ-grenoble-alpes.fr}).}
\and
Luc Hillairet \footnote{Universit\'e d'Orl\'eans, F\'ed\'eration Denis Poisson, Laboratoire MAPMO, route de Chartres, 45067 Orl\'eans Cedex 2, France (\texttt{luc.hillairet@univ-orleans.fr}). }
\and
Emmanuel Tr\'elat\footnote{Sorbonne Universit\'es, UPMC Univ Paris 06, CNRS UMR 7598, Laboratoire Jacques-Louis Lions, F-75005, Paris, France (\texttt{emmanuel.trelat@upmc.fr}).}
}

\begin{document}

\maketitle

\newpage 
\begin{abstract}
This is the first paper of a series in which we plan to study spectral asymptotics for sub-Riemannian Laplacians and to extend
 results that are classical in the Riemannian case concerning Weyl measures, quantum limits, quantum ergodicity, quasi-modes,
 trace formulae.
Even if hypoelliptic operators have been well studied from the point of view of PDEs, global geometrical and dynamical aspects have
 not been the subject of much attention. As we will see, already in the simplest case, the statements of the results in the
 sub-Riemannian setting are quite different from those in the Riemannian one.  

Let us consider a sub-Riemannian (sR) metric on a closed three-dimensional manifold with an oriented contact distribution.
 There exists a privileged choice of the contact form, with an associated Reeb vector field and a canonical volume form that
 coincides with the Popp measure. We establish a Quantum Ergodicity (QE) theorem for the eigenfunctions of any associated sR
 Laplacian under the assumption that the Reeb flow is ergodic. The limit measure is given by the normalized Popp measure.

This is the first time that such a result is established for a hypoelliptic operator, whereas the usual Shnirelman theorem yields
 QE for the Laplace-Beltrami operator on a closed Riemannian manifold with ergodic geodesic flow.

To prove our theorem, we first establish a microlocal Weyl law, which allows us to identify the limit measure and to prove the
 microlocal concentration of the eigenfunctions on the characteristic manifold of the sR Laplacian. Then, we derive a Birkhoff
 normal form along this characteristic manifold, thus showing that, in some sense, all 3D contact structures are microlocally
 equivalent. The quantum version of this normal form provides a useful microlocal factorization of the sR Laplacian. Using the
 normal form, the factorization and the ergodicity assumption, we finally establish a variance estimate, from which QE follows.


We also obtain a second result, which is valid without any ergodicity assumption: every Quantum Limit (QL) can be decomposed
 in a  sum of two mutually singular measures: the first measure is supported on the unit cotangent bundle and is invariant under
 the sR geodesic flow, and the second measure is supported on the characteristic manifold of the sR Laplacian and is invariant under
 the lift of the Reeb flow. Moreover, we prove that the first measure is zero for most QLs.
\end{abstract}

\newpage
\tableofcontents
\newpage

\section{Introduction and main results}\label{sec1}

\textit{Quantum ergodicity} (QE) theorems started with the seminal note \cite{Shn-74} by A. Shnirelman (see also \cite{Shn-93}).
His arguments were made precise in \cite{yCdV-85,Zel-87}, and then extended to the case of manifolds with boundary in \cite{GL-93,ZZ-96},
to the semi-classical regime in \cite{HMR-87} and  to the case of discontinuous metrics in \cite{JSSC-14}.
A weak version of QE is the following:
let $M$ be a compact metric space, endowed with a measure $\mu$, and let $T$ be a self-adjoint operator on $L^2(M,\mu)$,
 bounded below and having a compact resolvent (and hence a discrete spectrum). Let $(\phi_n)_{n\in\N^*}$ be a (complex-valued)
 Hilbert basis of $L^2(M,\mu)$, consisting of eigenfunctions of $T$, associated with the ordered sequence of
 eigenvalues $\lambda_1\leq\cdots\leq\lambda_n\leq\cdots$. 
We say that QE holds for the eigenbasis $(\phi_n)_{n\in\N^*}$ of  $T$ if there exist a probability measure $\nu$ on $M$
 and a density-one sequence $(n_j)_{j\in\N^*}$ of
 positive integers such that the sequence of probability measures $|\phi_{n_j}|^2 d\mu$ converges weakly to $\nu$.
The measure $\nu$ may be different from some scalar multiple of $\mu$,
 and may even be singular with respect to $\mu$ (see \cite{CHT-II}). In some cases, microlocal versions of QE hold true  and are
  stated in terms of  pseudo-differential operators.

Such a property provides some insight on the asymptotic behavior of  eigenfunctions of the operator $T$  in the limit of
 large eigenvalues. When $M$ is a compact
 Riemannian manifold and $T$ is the usual Laplace-Beltrami operator, QE is established under the assumption that
 the geodesic flow is  ergodic, and the limit measure $\nu$ is then 
the projection on $M$ of the normalized Liouville measure on the unit cotangent bundle of $M$.

To our knowledge, similar results are known in different contexts, but always for elliptic operators.
In the present paper, we establish a QE theorem for \textit{sub-Riemannian Laplacians} on a closed three-dimensional
 contact sub-Riemannian manifold without boundary.
Let us describe our main results.

\medskip

Let $M$ be a smooth connected compact three-dimensional manifold, equipped with an arbitrary non-vanishing  smooth density $\mu$
 (the associated measure is denoted by $\mu$ as well). Let $D\subset TM$ be a smooth oriented subbundle of codimension one.
Let $g$ be a smooth Riemannian metric on $D$. We assume that $D$ is a contact structure, so that there exists
 a unique contact form $\alpha_g$ defining $D$ (i.e., $D=\ker \alpha_g$) such that ${(d\alpha_g)}_{\vert D}$ coincides with the oriented
 volume form induced by $g$ on $D$.
Let $Z$ be the associated Reeb vector field, defined by $\alpha_g(Z)=1$ and $d\alpha_g(Z,\cdot)=0$.
The vector field $Z$ is transversal to the distribution $D$.
The flow generated by $Z$ on $M$ is called the \textit{Reeb flow}. It follows from Cartan's formula that the Reeb flow
preserves the measure given by the density $|\alpha_g \wedge d\alpha_g|$, also called the Popp measure in sub-Riemannian geometry,
and denoted by $dP$. Note that there is no relation between $\mu$ and the Popp measure.
 We also define the \textit{Popp probability measure} $\nu$ by
 $d\nu=\frac{1}{P(M)}dP$.

Let $\triangle_{sR}$ be the sub-Riemannian (sR) Laplacian associated with the contact sub-Riemannian structure $(M,D,g)$
 and with the measure $\mu$ (see Section \ref{sec:rappels} for some reminders and for the precise definition).
The operator $-\triangle_{sR}$ is self-adjoint on $L^2(M,\mu)$, is hypoelliptic, has a compact resolvent and thus has a discrete spectrum
$0 = \lambda_1 <\lambda_2 \leq \cdots\leq\lambda_n\leq\cdots$, with $\lambda_n\rightarrow +\infty$ as $n\rightarrow+\infty$.
The operator  $\triangle_{sR}$ commutes with complex conjugation (i.e., is ``real'') and hence admits an eigenbasis of real-valued eigenfunctions.

\medskip

Here, and throughout the paper, the notation $\langle\ ,\ \rangle$ stands for the (Hermitian) inner product in $L^2(M,\mu)$.

\newpage

Our first main result is the following.

\begin{customthm}{A}\label{thm1} 
{\it
We assume that the Reeb flow is \textit{ergodic} on $(M,\nu)$.\footnote{The Reeb flow $(\mathcal{R}_t)_{t\in\R}$
 generated on $M$ by the vector field $Z$ leaves the measure $\nu$ invariant. We say that this flow is ergodic on $(M,\nu)$ if any measurable invariant subset of $M$ is of measure $0$ or $1$. This implies, by the von Neumann ergodic theorem, that, for every continuous function $f$ on $M$, we have
$$
\frac{1}{T}\int_{0}^T f\circ \mathcal{R}_t(x) \, dt  \rightarrow  \int_M f \, d\nu ,
$$
in $L^2( M,\nu)$, as $T\rightarrow +\infty$.}

Then,  for any  orthonormal Hilbert basis $(\phi_n)_{n\in \N^*}$ of $L^2(M,\mu)$, consisting of eigenfunctions
 of $\triangle_{sR}$ associated with the eigenvalues $(\lambda_n)_{n\in\N^*}$ labelled  in increasing order, there exists a density-one sequence
 $(n_j)_{j\in\N^*}$ of positive integers such that
$$
\lim _{j\rightarrow +\infty} 
\left|\phi_{n_j}\right|^2 \mu = \nu ,
$$
where the limit is taken in the (weak) sense of duality with continuous functions.

For any \textit{real-valued} orthonormal Hilbert basis $(\phi_n)_{n\in \N^*}$ of $L^2(M,\mu)$, consisting of eigenfunctions
 of $\triangle_{sR}$ associated with the eigenvalues $(\lambda_n)_{n\in\N^*}$ labelled  in increasing order, there exists a density-one sequence
 $(n_j)_{j\in\N^*}$ of positive integers such that
\begin{equation*}
\lim_{j\rightarrow+\infty} \left\langle A\phi_{n_j}, \phi_{n_j}\right\rangle  =  \frac{1}{2}\int_M \big( a(q,\alpha_g(q))
 + a(q,-\alpha_g(q)) \big) \, d\nu (q),
\end{equation*}
where the scalar product is in $L^2 (M,\mu )$, 
for every classical pseudo-differential operator $A$ of order $0$ with homogeneous principal symbol $a$. If $a$ is even then the assumption on the basis to be real-valued can be dropped.
}
\end{customthm}

\begin{remark}
In the second assertion, we take a real-valued eigenbasis to deal with
the fact that the unit subbundle of $\Sigma $ has two connected
components. An example where the same problem occurs is the Fourier
basis on the circle $\R /2\pi \Z .$ In that case, the unit cotangent
bundle also has two connected components and it is known that 
the Shnirelman theorem is valid for the basis consisting of sines and
cosines, but not for the exponential basis.

In Section  \ref{sec:non-or}, we also show how to extend the above result to the case where $D$ is not orientable (see Theorem \ref{thm8.1}).
\end{remark}

Our result is valid for any choice of a smooth (non-vanishing) density $\mu$ on $M$. This is due to the fact that,
 up to an explicit unitary gauge transform, changing the density modifies the sR Laplacian with an additional bounded operator
 (see Remark \ref{rk:changemu1} for a precise statement) and this addition plays no role in the proofs.

\paragraph{Notations.}
We adopt the following notations.

Let $E $ be an arbitrary vector bundle over $M$. The sphere bundle $SE$ is the quotient of $E\setminus \{0\}$ by the positive homotheties, i.e., $SE = (E\setminus \{0\})/(0,+\infty)$.
Homegeneous functions of order $0$ on $E\setminus \{0\}$ are identified with functions on $SE$. The bundle $ST^\star M$, called the co-sphere bundle, is denoted by $S^\star M$.

Let $h$ be an arbitrary smooth, possibly degenerate, metric on $E$
that is associated to a nonnegative quadratic form.
 The unit bundle $U_hE $ is the subset of vectors of $E$ of $h$-semi-norm equal to one.
If $h$ is non degenerate, then $U_h E$ is canonically identified with $SE$; otherwise, the cylinder bundle $U_hE$ is
 identified with an open subset
 of $SE$.
The bundle $U_{h}T^\star M$ is more shortly denoted by $U_h M$ or even by $U^\star M$.

\medskip

The classical Shnirelman theorem is established in the Riemannian setting on a Riemannian manifold $(X,h)$
 under the assumption that the Riemannian geodesic flow is ergodic on $(S^\star X,\lambda_L)$, where the limit measure is the normalized 
Liouville measure $\lambda_L$ on $S^\star X\simeq U_{h^\star} X=\{h^\star =1\}$ (where $h^\star$ is the co-metric
 on $T^\star X$ associated with the Riemannian metric $h$).

In contrast,  the Liouville measure
 on the unit bundle $U_{g^\star} M=\{g^\star =1\}$ has infinite total mass, and hence the QE property cannot be formulated in
 terms of the sR geodesic flow: ergodicity has no sense in this case.
We recall that the co-metric $g^\star: T^\star M \rightarrow \R $ is defined as follows: 
$g^\star (q,\cdot)$ is the semi-positive quadratic form on $T_q^\star M$ 
defined by  $g^\star (q,p)= \| p_{| D_q}\|_q ^2 $ where the norm $\| .\|_q$  is the norm on $D^\star$ dual of the norm $g_q$.

Another interesting difference is that, in the Riemannian setting, QE says that most eigenfunctions equidistribute in the phase space,
 whereas here, in the 3D contact case, they concentrate on $\Sigma=D^\perp$ (where $\perp$ is in the sense of duality), 
the contact cone that is also  the characteristic manifold $(g^\star )^{-1}(0)$
 of $\triangle_{sR}$ (see the microlocal Weyl formula established in Theorem \ref{theo:weyl}).

\begin{remark}
To our knowledge, Theorem \ref{thm1} is the first QE result established for a hypoelliptic operator.
Our proof is specific to the 3D contact case, but the result can probably be extended to other sub-Riemannian Laplacians.
It is likely that the microlocal Weyl formula (Theorem \ref{theo:weyl}) can be generalized to equiregular sub-Riemannian structures.
The simplest nonregular case, the Martinet case, is already more sophisticated (see \cite{CHT-II,Mo-95}), and it is difficult to identify the adequate dynamics and the appropriate invariant measure being given by the microlocal Weyl formula.
The relationship with abnormal geodesics is, in particular, an interesting issue.
\end{remark}

\begin{remark}
We call \textit{spectral invariant (of the sub-Riemannian structure)} any
quantity that can be computed knowing only the associated spectral
sequence $(\lambda_n)_{n\in \N^*}$.
It is interesting to note that, as a consequence of the Weyl asymptotic formula (Section \ref{sec:micWeyl}), 
 the Popp volume $P(M)$ of $M$ is such a spectral
 invariant of the sub-Riemannian
 structure (and this, for any choice of $d\mu$). 

When $M=\mathbb{S}^3$, then $1/P(M)$ is the asymptotic Hopf invariant of
 the Reeb vector field $Z$ 
(with respect to the Popp probability measure $\nu$) introduced in
\cite{Ar-86}. It follows that the asymptotic Hopf invariant also is a
spectral invariant.
\end{remark}

Let $(\psi_j)_{j\in\N^*}$ be an arbitrary orthonormal family of $L^2(M,\mu)$.
Let $a:S^\star M \rightarrow \C$ be a smooth function. In Appendix \ref{app:PDO}, we recall how to \textit{quantize} $a$, i.e., how to associate
to $a$ a bounded pseudo-differential operator  $\Op(a)$.
  We set $\mu_j(a) = \langle \Op(a)\psi_j,\psi_j\rangle$,
 for every $j\in\N^*$. The measure $\mu_j$ is asymptotically positive, and any
 closure point (weak limit) of $(\mu_j)_{j\in\N^*}$ is a probability measure on $S^\star M$, called a \textit{quantum limit} (QL),
 or a \textit{semi-classical measure},  associated with the family $(\psi_j)_{j\in\N^*}$.

Theorem \ref{thm1} says that, under the ergodicity assumption of the Reeb flow, the Popp probability measure $\nu$, which is
 invariant under the Reeb flow, is the ``main" QL associated with any eigenbasis.

Our second main result hereafter provides an insight on QLs in the 3D contact case in greater generality, without any ergodicity assumption.
In order to state it, we identify $S^\star M$ with the union of $U^\star M =\{ g^\star =1 \}$ (which is a cylinder bundle) and of $S\Sigma$
 (which is a two-fold covering of $M$): each fiber is obtained by compactifying a cylinder with two points at infinity.

\begin{customthm}{B}\label{thm2}
{\it 
Let $(\phi_n)_{n\in \N^*}$ be an orthonormal Hilbert basis of $L^2(M,\mu)$, consisting of eigenfunctions of $\triangle_{sR}$ associated
 with the eigenvalues $(\lambda_n)_{n\in\N^*}$ labelled in increasing order.
\begin{enumerate}
\item Let $\beta$ be a QL associated with the subsequence family $(\phi_{n_j})_{j\in\N^*}$. Using the above identification
 $S^\star M = U^\star M \cup S\Sigma$, the probability measure $\beta$ can be written as the sum $\beta =\beta_0 + \beta_\infty $
 of two mutually singular measures such that:
\begin{itemize}
\item $\beta_0(S\Sigma )=0$  and  $\beta_0$ is invariant under the sR geodesic flow associated with the sR metric $g$,
\item $\beta_\infty$ is supported on $S\Sigma $ and is invariant under the lift to $S\Sigma$
of the Reeb flow.
\end{itemize} 
\item There exists a density-one sequence $(n_k)_{k\in\N^*}$ of positive integers such that, if $\beta$ is a QL associated with a subsequence
of $(n_k)_{k\in\N^*}$, then the support of $\beta$ is contained in $S\Sigma$, i.e., $\beta_0=0$ in the previous
 decomposition.
\end{enumerate}
}
\end{customthm}

\begin{remark}
The decomposition and the statement on $\beta_0$ are actually valid
for any sR Laplacian, not only in the 3D contact case. 
A stimulating open question is to establish invariance properties of $\beta_\infty$ for more general sR geometries.
\end{remark}

In Theorem \ref{thm2}, we do not assume that the eigenbasis is real-valued.

In the classical Riemannian case, it is well known that any QL associated with a family of eigenfunctions is invariant under
 the geodesic flow. 
Indeed, denoting by $\exp(it\sqrt{-\triangle})$ the half-wave propagator, we have
$$
\left\langle \exp(-it\sqrt{-\triangle}) \Op(a) \exp(it\sqrt{-\triangle}) \phi_n,\phi_n\right\rangle = \mu_n(a) ,
$$
for every $t\in\R$, for every $n\in\N^*$, and for every classical symbol of order $0$. By the Egorov theorem,
 $\exp(-it\sqrt{-\triangle}) \Op(a) \exp(it\sqrt{-\triangle})$ is a pseudo-differential operator of order $0$,
 with principal symbol $a\circ\exp(t\vec{H})$, where $\exp(t\vec{H})$ is the Hamiltonian geodesic flow. The invariance property follows.

In the sub-Riemannian 3D contact case, the proof is not so simple and follows from arguments used to prove Theorem \ref{thm1}. In particular,
 Lemma \ref{lemcrochetnul} in that proof serves as a substitute for the Egorov theorem (see Section \ref{main:lemma}). 
It is crucial in the proof to first decompose the QL into the sum of a part that is supported away of $S\Sigma$ and of a part that
 is supported on $S\Sigma$, before proving that the latter is invariant under the Reeb flow. 

\paragraph{A general path towards QE.}
Let us indicate some ideas behind the proof of Theorem \ref{thm1}. We follow the general path towards establishing the QE property, as clarified, e.g., in \cite{Ze-10}.
We set $N(\lambda)=\# \{n \mid \lambda_n \leq \lambda \}$.

The first step consists in establishing a \textit{local Weyl law} and a \textit{microlocal Weyl law}:
both are independent of the eigenbasis. The local Weyl law provides
some information on the average localization of eigenfunctions in $M$,
while the microlocal Weyl law provides some similar information in the
Fourier-momentum space $S^\star M$. 
More precisely, the microlocal Weyl law states that
\begin{equation}\label{localWeyllaw}
\lim _{\lambda \rightarrow \infty }\frac{1}{N(\lambda)}\sum_{\lambda_n\leq \lambda } \left\langle A
 \phi_n, \phi_n\right\rangle =\bar{a}= \int _{S^\star M } a\, dW_\triangle,
\end{equation}
for every pseudo-differential operator $A$ of order $0$ with a positively homogeneous principal symbol $a$,
 where $S ^\star M$ is the unit sphere bundle of $T^\star M$, and $ dW_\triangle $ is a probability measure that we call the
 \textit{microlocal Weyl measure} (see Definition \ref{defi:weyl-measure} in Section \ref{sec:micWeyl}).
This Ces\'aro convergence property can often be established under weak
assumptions, and without any ergodicity property.

The second step consists in proving the \textit{variance estimate}
\begin{equation}\label{varianceestimate}
\lim _{\lambda \rightarrow \infty }\frac{1}{N(\lambda)}\sum_{\lambda_n\leq \lambda }
\left\vert \left\langle (A-\bar a\, \mathrm{id}) \phi_n , \phi_n \right\rangle \right\vert^2 =0 .
\end{equation}
The variance estimate usually follows by combining the microlocal Weyl law \eqref{localWeyllaw} with ergodicity properties of some
 associated classical dynamics and with an Egorov theorem. Actually, the following infinitesimal version of the Egorov theorem suffices:
$$
\left\langle[A,\sqrt{\triangle}] \phi_n ,\phi_n  \right\rangle=0 ,
$$
or even the weaker property
$$
\lim _{\lambda \rightarrow \infty }\frac{1}{N(\lambda)}\sum_{\lambda_n\leq \lambda }
\left\vert \left\langle \langle[A,\sqrt{\triangle}] \phi_n , \phi_n  \right\rangle \right\vert^2 =0 .
$$
Then  QE follows from the two properties above. Indeed, for a fixed pseudo-differential operator $A$ of order $0$, it follows from
 \eqref{varianceestimate} and from a well known lemma\footnote{This lemma states that, given a bounded sequence $(u_n)_{n\in\N}$ of nonnegative
 real numbers, the Ces\'aro mean $\frac{1}{n}\sum_{k=0}^{n-1}u_k$ converges to $0$ if and only if there exists a subset $S\subset\N$ of density
 one such that $(u_k)_{k\in S}$ converges to $0$. We recall that $S$ is of density one if $\frac{1}{n}\#\{k\in S\mid k\leq n-1\}$ converges to $1$
 as $n$ tends to $+\infty$.} due to Koopman and Von Neumann (see, e.g., \cite[Chapter 2.6, Lemma 6.2]{Petersen}) that there exists a density-one
 sequence $(n_j)_{j\in\N^*}$ of positive integers such that
$$
\lim _{j\rightarrow +\infty } \left\langle A \phi_{n_j}, \phi_{n_j} \right\rangle = \bar{a}.
$$
Using the fact that the space of symbols of order $0$ admits a countable dense subset, QE is then established with a diagonal argument.

\paragraph{Organization of the paper.}
In Section \ref{sec:rappels}, we provide the precise and complete framework that is relevant both to the sub-Riemannian geometric setting and
 to the spectral analysis of the associated sR Laplacian. 

In Section \ref{sec_examples}, we recall the spectral theory of compact quotients of the Heisenberg group with an invariant metric. The study made in \cite{yCdV-83} serves as a guiding model for our general result. This example is used in Section \ref{sec:melrose} in order to get a quantized version of the Birkhoff normal form. We also compute some QLs in this Heisenberg flat case. Finally, we give some examples of 3D contact manifolds on which the Reeb flow is ergodic.

In Section \ref{sec:micWeyl}, we establish the microlocal Weyl law. No ergodicity assumption is required here.
The local Weyl law is established in  \cite{Me-Sj-78}. The limit measure that appears in the local Weyl law is the Popp probability measure. Therefore, this measure is the good candidate for a QE result, and the fact that the Popp measure is invariant under the Reeb flow makes it natural to expect that the Reeb vector field is relevant for that purpose. For the same reason, it also indicates that the sub-Riemannian 
geodesic flow \textit{is not} a good candidate, because it does not preserve any lift of the Popp measure in general.
The microlocal Weyl law is derived thanks to a general argument proving the average concentration of the eigenfunctions on $\Sigma$.

In Section \ref{sec:melrose}, we establish a classical Birkhoff normal form and then a quantum normal form microlocally near $\Sigma$. Such normal forms have proved to be relevant in the semi-classical literature to obtain fine spectral results. Our normal form is essentially due to \cite{Me-84} and is closely related to the one of \cite{Ra-VN-13} for large magnetic fields in dimension two (see also \cite{Bou-74}). This normal form implies that, microlocally near $\Sigma$, all contact 3D sub-Riemannian structures are equivalent, and in particular can be (almost) conjugated to the local model of the Heisenberg flat case. In the quantum version, this conjugation is performed with a Fourier Integral Operator, and we infer from that result an almost factorization of $\triangle_{sR}$ microlocally near $\Sigma$ (quantum Birkhoff normal form). Here, ``almost" means that the factorization is exact only along $\Sigma$, with remainder terms that are flat along $\Sigma$.

In Section \ref{sec_variance}, we use the quantum Birkhoff normal form to prove the variance estimate, and, using pseudo-differential
 calculus techniques (in particular, averaging and brackets) and the ergodicity assumption on the Reeb flow,
 we infer the QE property for $\triangle_{sR}$, proving Theorem \ref{thm1}. 
Note that we state and prove in Section \ref{main:lemma} a crucial lemma playing the role of an infinitesimal version of the Egorov theorem. 

Theorem \ref{thm2} is proved in Section \ref{sec_proof_thm2}.

Section \ref{sec:non-or} is devoted to showing how to extend Theorem \ref{thm1} to the case of a non-orientable subbundle $D$.

Several appendices gather some useful technical statements.

\medskip

Some of the results of this paper have been announced in \cite{CHT_semEDP}, and proved in a simpler framework.
This paper is the first of a series. In \cite{CHT-II}, we will establish general properties of Weyl measures for general equiregular sR cases (microlocal concentration, heat asymptotics) and also in some singular cases, such as the Grushin and the Martinet cases.
In \cite{CHT-III}, we will describe properties of the geodesics in the 3D contact case, the main observation being that geodesics with large initial moments are spiraling around Reeb orbits. A precise version of this behavior uses a more refined normal form due to Melrose (see \cite{Me-84}), of which we give a full proof using Nelson's trick (see \cite{Nelson}), which is a scattering argument.
We are also preparing a short note on the relationship between magnetic and sR dynamics (see \cite{yCdV-17}).


\section{Geometric and spectral preliminaries}\label{sec:rappels}
Let us start with several notations. We denote by $\omega$ the canonical symplectic form on the cotangent bundle $T^\star M$ of $M$.
In local coordinates $(q,p)$ of $T^\star M$, we have $\omega = dq\wedge dp=-d\Lambda$ with $\Lambda=p\, dq$ (Liouville one-form).
 We denote by $\{\ ,\ \}_\omega$ the Poisson bracket associated with $\omega$, dropping the index $\omega$ when the context is clear. Given a smooth Hamiltonian function $h:T^\star M \rightarrow \R $,
 we denote by $\vec h$ the corresponding Hamiltonian vector field on $T^\star M$, defined by $\iota_{\vec h}\omega = dh$,
 and we denote by $\exp(t \vec h)$ the flow at time $t$ generated by $\vec h$ on $T^\star M$.
Given any smooth vector field $X$ on $M$, we denote by $h_X$ the Hamiltonian function (momentum map) on $T^\star M$ associated with $X$
defined by 
$h_X = \Lambda(X)$, that is, in local coordinates, by
$h_X(q,p)=p(X(q))$. The hamiltonian flow of $\vec{h}_X$ projects onto the integral curves of $X$.
Note also that, given two vector fields $X$ and $Y$ on $M$, we have
$\{ h_X, h_Y \}_\omega = -h_{ [ X,Y ]} $, where the Lie bracket $[X,Y]$ is defined in terms of derivation by $[X,Y]=XY-YX$.
Throughout the paper, the notation $\mathrm{orth}_\omega$ stands for the symplectic $\omega$-orthogonal.

\subsection{Sub-Riemannian Laplacians}\label{sec_DeltaSR}
Let $(M,D,g)$ be a sub-Riemannian (sR) structure where  $M$ is a smooth connected compact three-dimensional manifold, $D$ is a smooth subbundle of $TM$ of rank two (called horizontal distribution), and $g$ is a smooth Riemannian metric on $D$.
We assume that $D$ is a contact distribution, that is, we can write $D=\ker\alpha$ locally around any point, for some one-form $\alpha$
 such that $\alpha\wedge d\alpha\neq 0$ (locally). At this step, we do not need to normalize the contact form.

In order to define a sub-Riemannian Laplacian $\triangle_{sR}$, let us choose a smooth density $\mu$ on $M$.
The choice of $\mu$ is independent of that of $g$.\footnote{As we will see, the choice of $\mu$ plays no role
 in what follows. Beyond this paper, we expect that this fact is important in the non-equiregular cases where
 there is no canonical choice of $\mu$.}
Let $L^2(M,\mu)$ be the set of complex-valued functions $u$ such that $|u|^2$ is $\mu$-integrable over $M$.
Then $-\triangle_{sR} $ is the nonnegative self-adjoint operator on $L^2(M,\mu)$ defined as the Friedrichs extension of the Dirichlet integral
$$
Q(\phi)=\int _M \Vert d\phi \Vert_{g^*}^2 \, d\mu, 
$$
where the norm of $d\phi $ is calculated with respect to the (degenerate) dual metric $g^\star $ (also called co-metric)
 on $T^\star M $ associated with $g$.
The sR Laplacian $\triangle_{sR}$ depends on the choice of $g$ and of $d\mu$.
Note that $g^\star = h_X^2+h_Y^2$ if $(X,Y)$ is a local $g$-orthonormal frame of $D$.

We consider the divergence operator $\mathrm{div}_\mu$ associated with the measure $\mu$,
 defined by $\mathcal{L}_X d\mu = \mathrm{div}_\mu(X)\, d\mu$
 for any vector field $X$ on $M$.
Besides, the horizontal gradient $\nabla_{sR}\phi$ of a smooth function $\phi$ is the unique section
 of $D$ such that $g_q(\nabla_{sR}\phi(q),v)=d\phi(q).v$, for every $v\in D_q$. The operator  $-\mathrm{div}_\mu$
is the formal adjoint of $\nabla_{sR}$ on $L^2 (M,\mu)$. 
 Hence, we have 
$$
\triangle_{sR}\phi = \mathrm{div}_\mu(\nabla_{sR}\phi),
$$
for every smooth function $\phi$ on $M$. 

Since $\Vert d\phi \Vert_{g^*}^2 = \Vert \nabla_{sR}\phi\Vert_g^2$, if $(X,Y)$ is a local $g$-orthonormal frame of $D$, then
 $\nabla_{sR}\phi=(X\phi)X+(Y\phi)Y$, and $Q(\phi) = \int_M \left( (X\phi)^2+(Y\phi)^2 \right) d\mu$. It follows that
$$
\triangle_{sR}= - X^\star X - Y^\star Y = X^2+Y^2+\mathrm{div}_\mu(X) X+\mathrm{div}_\mu(Y) Y, 
$$
where the adjoints are taken in $L^2(M,\mu)$.

\begin{remark}\label{rem5}
The co-metric $g^\star$ induces on $T^\star M$ an Hamiltonian vector field $\vec{g^\star}$.
The projections onto $M$ of the integral curves of $\vec{g^\star}$ are the (normal) geodesics of the
 sub-Riemannian metric $g$ (see \cite{Mo-02}). This defines the sR geodesic flow.
\end{remark}

\begin{remark}\label{rk:changemu1}
Let us express the difference between two sub-Riemannian Laplacians $\triangle_{\mu_1}$ and
 $\triangle_{\mu_2}$ associated with two different densities (but with the same metric $g$).
Assume that $\mu_2= h^2\mu_1$ with $h$ a positive smooth function on $M$.
It is easy to see that $\mathrm{div}_{\mu_2}(X) = \mathrm{div}_{\mu_1}(X) + \frac{X(h^2)}{h^2}$, for every vector field $X$.
 By simple computations, we then establish that
$h\triangle_{\mu_2}(\phi) = \triangle_{\mu_1}(h\phi) - \phi\triangle_{\mu_1}(h)
= \triangle_{\mu_1}(h\phi) + h^2 \phi\triangle_{\mu_2}(h^{-1})$.
To settle this identity in a more abstract way, we define the isometric bijection $J:L^2(M,\mu_2)\rightarrow L^2(M,\mu_1)$ by $J\phi=h\phi$.
Then, we have
$$
J\triangle_{\mu_2}J^{-1} = \triangle_{\mu_1} - \frac{1}{h} \triangle_{\mu_1}(h) \,\mathrm{id} = \triangle_{\mu_1} + h \triangle_{\mu_2}(h^{-1})
 \,\mathrm{id} .
$$
It follows that $\triangle_{\mu_1}$ is unitarily equivalent to $\triangle_{\mu_2} +W$, where $W$ is a bounded operator.

This remark is important because, for a given metric, it allows us to work with the sub-Riemannian Laplacian associated with any density.
Usually, this kind of fact is abstracted by using half-densities (see \cite{Du-96}) which give a canonical Hilbert space on $M$.
A way of rephrasing the previous remark is then to say that $ \triangle_{sR}$
is any self-adjoint second-order differential operator whose principal symbol is $g^\star$,
 whose sub-principal symbol vanishes, and such that $\lambda_1=0$ (first eigenvalue).
\end{remark}

\begin{remark} 
Note that the definitions and statements of this section hold for any sR Laplacian.
\end{remark}

\subsection{Microlocal aspects}
We are going to use some microlocal analysis and the symbolic calculus of pseudo-differential operators in order
 to study the operator $\triangle_{sR}$. The facts that we will use are recalled in Appendix \ref{app:PDO}, such as, in particular, the notion of \textit{principal symbol} of a pseudo-differential operator that is defined in an intrinsic way; the principal symbols of pseudo-differential operators of order $0$ are identified with functions on $S^\star M$.
The principal symbol of $-\triangle_{sR}$ is given on $T^\star M$ by
$$
\sigma_P(-\triangle_{sR}) = g^\star  ,
$$ 
i.e., it coincides with the co-metric $g^\star$, 
and thus, if $D$ is locally spanned by a $g$-orthonormal frame $(X,Y)$, then $\sigma_P(-\triangle_{sR}) = h_X^2 + h_Y^2$. 
The sub-principal symbol of $-\triangle_{sR}$ is zero because
$-\triangle_{sR}= X^\star X+Y^\star Y$ locally (see Appendix
\ref{app:PDO}).

This will be important to derive the quantum normal form in Section \ref{sec:QNF} and in particular to establish Lemma \ref{main:lemma} (which is the main lemma playing the role of an infinitesimal Egorov theorem).
 
For any other volume form $d\mu'$ on $M$, the corresponding 
sub-Riemannian Laplacian is unitarily equivalent to $\triangle_{sR}+ V\,\mathrm{id}_M$, 
where $V$ is a smooth function (see also Remark \ref{rk:changemu1}).

Given any local $g$-orthonormal frame $(X,Y)$ of $D$,
the vector fields $(X,Y,[X,Y])$ generate $TM$, and it follows from \cite{Ho-67} that the operator $-\triangle_{sR}$ is hypoelliptic and has a compact resolvent (see also \cite{Strichartz_JDG1986}).
It thus has a discrete real spectrum $(\lambda_n)_{n\in\N^*}$, with $0=\lambda _1 <\lambda_2 \leq \cdots \leq \lambda_n\leq \cdots $, with $\lambda_n\rightarrow +\infty$ as $n\rightarrow+\infty$.

Throughout the paper, we consider an orthonormal Hilbert basis $(\phi_n )_{n\in \N}$ of $L^2(M,\mu)$,
 consisting of eigenfunctions of $\triangle_{sR}$, associated with the eigenvalues $(\lambda_n)_{n\in\N^*}$, labelled in increasing order.
Using that $\triangle_{sR}$ commutes with complex conjugation, there exist eigenbases consisting of real-valued eigenfunctions.

\subsection{Popp measure and Reeb vector field}\label{sec_Popp}
In order to define the \textit{Popp measure} and the \textit{Reeb vector field}, we first choose a  global contact form defining $D$. Since $D$ is assumed to be orientable, and since every contact manifold of dimension
3 is orientable, $D$ is co-orientable
 as well and hence there
are global  contact forms.
We define $\alpha _g$ as the unique contact form defining $D$, such  that, for each $q\in M$,
 ${(d\alpha_g)}_{\vert D_q}$ coincides with the volume form induced by $g$
and the orientation  on $D_q$.

We define the density $dP= \vert\alpha_g \wedge d\alpha_g\vert$ on $M$. In general, $dP$ differs from $d\mu$.
 The corresponding measure $P$ is called the \textit{Popp measure} in the existing literature (see \cite{Mo-02},
 where it is defined in the general equiregular case). Of course, here, this measure is the canonical contact measure
 associated with the normalized contact form $\alpha_g$.
In what follows, we consider the \textit{Popp probability measure}
\begin{equation*}
\nu = \frac{P}{P(M)} .
\end{equation*}
The \textit{Reeb vector field} $Z$ of the contact form $\alpha_g$ is defined as the unique vector
 field such that $\alpha_g(Z)=1$ and $d\alpha_g (Z,\cdot)=0$. Equivalently, using the identity
$d\alpha (Z,Y)=Z.\alpha (Y)- Y.\alpha (Z) -\alpha ([Z,Y])$, one gets that $Z$ is the unique vector field such that
\begin{equation}\label{proprieteReeb}
[X,Z]\in D,\quad [Y,Z]\in D,\quad [X,Y]=-Z\mod D ,
\end{equation}
for any positive orthonormal local frame $(X,Y)$ of $D$. In particular, $Z$ is transverse to $D$.

Using the Cartan formula, we have $\mathcal{L}_Z\alpha =0$, hence $\mathcal{L}_Z\nu=0$.
 This computation shows that the Popp measure $\nu$ is invariant under the vector field $Z$, and equivalently,
 under the Reeb flow generated by $Z$. As already said, it is crucial to identify such an invariance property
 in view of deriving a QE result.

\begin{remark}
The Reeb vector field $Z$ has the following dynamical interpretation. If $(q_0,v_0)\in D$, then there
 exists a one-parameter family of geodesics associated with these Cauchy data: they are the projections of the integral
 curves of the Hamiltonian vector field $\vec{g^\star}$ (defined in Remark \ref{rem5}) whose Cauchy data are $(q_0, p_0)$
 with $( p_0) _{\vert D_{q_0}}=g(v_0,\cdot)$. For every $u\in \R$, the projections on $M$ of the integral curves of $\vec{g^\star}$
 with Cauchy data $(q_0, p_0+u\alpha_g )$ in the cotangent space have the same Cauchy data $(q_0,v_0)$ in the tangent space.
As $u\rightarrow \pm\infty $, they spiral around the integral curves of $\mp Z$.
From the point of view of spectral asymptotics, this part of the dynamics is expected to be the dominant one. This is explained in more details in \cite{CHT-III}. 
\end{remark}

\subsection{The characteristic cone and the Hamiltonian interpretation of the Reeb flow}\label{sec_reeb}
Let $\Sigma \subset T^\star M$ be the characteristic manifold of $-\triangle_{sR}$, given
by
$$
\Sigma = (g^\star)^{-1}(0) =D^\perp , 
$$
that is the annihilator of $D$.
If $D$ is locally spanned by a $g$-orthonormal frame $(X,Y)$, then $\Sigma = h_X^{-1}(0)\cap h_Y^{-1}(0)$.
Note that $\Sigma$ coincides with the cone 
(also globally defined)
\begin{equation}\label{defSigma}
\Sigma = \{ (q,s\,\alpha_g (q))\in T^\star M \ \mid\ q\in M, s\in\R \}.
\end{equation}
We define the function $\rho :\Sigma \rightarrow \R$ by $\rho(s\,\alpha_g)=s$. The function $\rho$ is also
 the restriction of $h_Z$ to $\Sigma$. We set $\Sigma_s=\rho^{-1}(s)$, for every $s\in\R$.

An important feature of the contact situation is that the characteristic cone is symplectic: the restriction $\omega_{\vert \Sigma}$
 is symplectic.
We have the following result.

\begin{lemma}\label{lemmaReeb} 
The vector $\vec{h}_Z$ is tangent to $\Sigma $ and
the Hamiltonian vector field $\vec{\rho}$ on $\Sigma$ coincides with the restriction of $\vec{h}_Z$ to $\Sigma.$
\end{lemma}

\begin{proof}
Given any point $(q,p)=(q,s\,\alpha_g (q))\in\Sigma\subset T^\star M$, by definition of $Z$ we have $s=h_Z(q,p)$,
 and thus $\Sigma = \{ (q,h_Z(q,p)\,\alpha_g(q))\ \mid\ (q,p)\in T^\star M \}$, that is, $\Sigma$ is the graph of $h_Z\, \alpha_g$ 
in $T^\star M$. Hence $\rho=(h_Z)_{\vert\Sigma}$.

To prove that $\vec{h}_Z$ is tangent to $\Sigma$, it suffices to prove that $dh_X.\vec{h}_Z=dh_Y.\vec{h}_Z=0$
 along $\Sigma$, where $(X,Y)$ is a local $g$-orthonormal frame of $D$. We have
 $dh_X.\vec{h}_Z = \omega(\vec{h}_X,\vec{h}_Z)=\{h_X,h_Z\}=-h_{[X,Z]}$, and by
 \eqref{proprieteReeb} we have $[X,Z]\in D$ and thus $h_{[X,Z]}=0$ on $\Sigma=D^\perp$. The conclusion follows.

Now, since we have by definition $\iota_{\vec{h}_Z}\omega=dh_Z$, and,
 since $\vec{h}_Z$ is tangent to $\Sigma$, we have $\iota_{(\vec{h}_Z)_{\vert\Sigma}}\omega_{\vert\Sigma}=d(h_Z)_{\vert\Sigma}=d\rho$,
 and therefore $\vec{\rho}=(\vec{h}_Z)_{\vert\Sigma}$, because $(\Sigma,\omega_{\vert\Sigma})$ is symplectic.
\end{proof}

This lemma implies that 
$$
\exp(t\vec\rho)=\exp(t\vec h_Z)_{\vert\Sigma} ,
$$
and that the Hamiltonian flow $\exp(t\vec \rho)$ on $\Sigma_s$ projects onto the Reeb flow on $M$. The Reeb dynamics play a crucial role in what follows, and Lemma \ref{lemmaReeb} will be used to infer information from the dynamics of the Hamiltonian flow generated by $\rho$ on the symplectic manifold $\Sigma$ to the Reeb dynamics on $M$. 
 
We define $\Sigma^\pm =\rho^{-1}(\{s \ \mid\ \pm s>0 \})$.
Note that $\Sigma^+$ and $\Sigma^-$ are permuted when one changes the orientation of $D$, while the definition of the sR Laplacian remains unchanged. 
Each of the submanifolds $\Sigma _1$ and $\Sigma _{-1} $ is connected and is a graph over $M$. 
Denoting by $\hat\nu_s$ the lift of the measure $\nu$ to $\Sigma_s$, the ergodicity assumption of Theorem \ref{thm1} can be rephrased by saying that the Hamiltonian flow $\exp(t\vec \rho)$ is ergodic on $(\Sigma_{\pm 1},\hat \nu_{\pm 1}).$

In the Riemannian setting, the unit tangent bundle is connected for $d\geq 2$.
The fact that we have here two connected components explains why the second statement of Theorem \ref{thm1} applies only to a real-valued eigenbasis or to observables $a$ having a  principal symbol that is symmetric with respect to $\Sigma$.


\section{Examples}\label{sec_examples}
\subsection{The Heisenberg flat case}\label{sec:heis}
The simplest example is given by an invariant metric on a compact quotient of the Heisenberg group. The spectral decomposition of
 the Heisenberg Laplacian is then explicit (see \cite{yCdV-83,GW-86}) and we present it here in order to
get a better understanding, because,
as we will show in Section \ref{sec:melrose}, this example serves as a microlocal normal form model
 for all sub-Riemannian Laplacians of contact type in dimension three.

Let $G$ be the three-dimensional Heisenberg group defined as $G=\R^3$ with the product rule
$$
(x,y,z)\star (x',y',z')=(x+x',y+y',z+z'-xy') .
$$
The contact form $\alpha_H = dz+x\,dy $ and the vector fields $X_H= \partial_x$ and
 $Y_H= \partial_y - x\partial_z $ are left-invariant on $G$.
 Let $\Gamma $ be the discrete co-compact subgroup of $G$ defined
by 
$$
\Gamma =\{ (x,y,z)\in G \ \mid\ x,y\in \sqrt{2\pi}\,\Z,\, z\in 2\pi  \Z \}  .
$$
We define the 3D compact manifold manifold $M_H = \Gamma \backslash  G$,
 and we consider the horizontal distribution $D_H=\ker \alpha_H$,
 endowed with the metric $g_H$ such that $(X_H,Y_H)$ is a $g_H$-orthonormal frame of $D_H$. With this choice, we have $(\alpha_H)_g=\alpha_H$.

The Reeb vector field is given by $Z_H=-[X_H,Y_H]=\partial_z $.
The Lebesgue volume $d\mu =| dx\, dy\, dz|$ coincides with the Popp volume $dP$, and we consider the corresponding sub-Riemannian Laplacian
$$
\triangle_H=X_H^2+Y_H^2.
$$
Note that the vector fields $X_H$ and $Y_H$ have divergence zero.

We refer to this sub-Riemannian case as the \textit{Heisenberg flat case}.
Accordingly, we set
$$
g^\star_H = \sigma_P(-\triangle_H) = h_{X_H}^2 + h_{Y_H}^2 ,
$$
and 
$$
\Sigma_H=(g^\star_H)^{-1}(0) = h_{X_H}^{-1}(0) \cap h_{Y_H}^{-1}(0).
$$
Note that, denoting by $(x,y,z,p_x,p_y,p_z)$ the coordinates in the cotangent space, we have $h_{X_H}=p_x$, $h_{Y_H}=p_y-xp_z$ and $h_{Z_H}=p_z$.

\paragraph{Spectrum of $\triangle_H$.}
The main observation to derive the next result is that the sub-Riemannian Laplacian $\triangle_H$ commutes with the vector field $-iZ_H$ whose spectrum is the set of integers $m\in\Z$.

\begin{proposition}\label{prop:spec_heisflat} 
The spectrum of $-\triangle_H$ is given by
$$
\{\lambda _{\ell,m}= (2\ell+1)|m| \ \mid\ m\in \Z\setminus \{0\},\, \ell\in\N \}
\ \cup\ \{\mu_{j,k}= {2\pi }(j^2 + k^2) \  \mid\ (j,k)\in\Z^2  \} ,
$$
where $ \lambda _{\ell,m} $ is of multiplicity $|m| $. 
The spectral counting function is $N(\lambda )= \sum _{\ell\in\N }N_\ell (\lambda )+ N_T(\lambda)$,
 with $N_T(\lambda )\sim \lambda /2$ and 
$$
N_\ell (\lambda )=  \mathop{\sum_{m\neq 0}}_{|m|\leq \lambda/(2\ell+1)} |m| .
$$
The Weyl asymptotics is 
\begin{equation*} 
N(\lambda )\sim \frac{\pi^2}{8}\lambda^2 = \frac{P(M)}{32}\lambda^2,
\end{equation*} 
as $\lambda\rightarrow+\infty$.
\end{proposition}

It is interesting to compare this result with the corresponding result for a Riemannian Laplacian
 on a 3D closed Riemannian manifold, for which we have $N(\lambda )\sim C\lambda^{3/2}$ as $\lambda\rightarrow+\infty$.

\begin{proof}[Proof of Proposition \ref{prop:spec_heisflat}.]
The proof relies on the following classical lemma.

\begin{lemma}\label{lem_spectrum}
Let $A$ be a closed operator on a Hilbert space. We assume that the self-adjoint operator $B=A^\star A$ has a compact resolvent and that  $[ A, A^\star ]=c\,\mathrm{id}$ with $c\neq 0$. Then the spectrum of $B$ consists of the eigenvalues $\ell\vert c\vert$, with $\ell\in\N$, 
all of them having the same multiplicity. 
\end{lemma}

\begin{remark}
This lemma can also be used to compute the spectrum of the harmonic oscillator.
\end{remark}

Considering a Fourier expansion with respect to the $z$ variable, and using the fact that
$Z_H$ commute with $\triangle_H$, we have $L^2(M_H)=\oplus_{m\in \Z} \mathcal{H}_m $, and accordingly, $-\triangle_H =\oplus_{m\in \Z}B_m$.
The operator $B_0$ is the Laplacian on the flat torus $\R^2 /\sqrt{2\pi }\Z^2$, and thus its eigenvalues are $\mu_{j,k} = 2\pi (j^2+ k^2)$, with $(j,k)\in\Z^2$.

Let us now consider the operators $B_m$ with $m\neq 0$.
Let $A=X_H+iY_H$ and $A^\star =-X_H+iY_H $. We have
$[A,A^\star ]=-2iZ_H$ and $-\triangle_H =A^\star A +iZ_H$.
Restricting these identities to $\mathcal{H}_m$ yields
$B_m= A_m^\star A_m -m $ and $[A_m,A_m^\star ]=2m$.
It follows from Lemma \ref{lem_spectrum} that the spectrum of $B_m$ is given by the eigenvalues $(2\ell+1)|m|$, $\ell\in\N$, all of them having the same multiplicity.
Now we use the fact that $B_m$ is an elliptic operator on a complex line bundle over the torus, of principal symbol $p_x^2+p_y^2$, and we use the Weyl asymptotics in order to show that all these multiplicities are equal to $|m|$.
We get the Weyl asymptotics as follows:
$$
2  \sum _{\ell=0}^{+\infty} \sum _{m=1}^{\lambda /2\ell+1}m \sim \lambda ^2 \sum _{\ell=0}^{+\infty} \frac{1}{(2\ell+1)^2} ,
$$
as $\lambda\rightarrow+\infty$.
The proposition is proved.
\end{proof}

\paragraph{Normal form.} We next provide a reformulation of the previous computations, which is useful in the sequel.
The sub-Riemannian Laplacian can be written as a product $-\triangle_H = R_H\Omega_H$ with two commuting operators $R_H$ and $\Omega_H$.
The operator 
$$
R_H=\sqrt{Z_H^\star Z_H}
$$ acts by multiplication by $|m|$ on the functions of $\mathcal{H}_m$.
We define the two  operators $U_H$ and $V_H$ on $\oplus_{m\neq 0} \mathcal{H}_m$
by
$$
U_H = \frac{1}{i}R_H^{-\frac{1}{2}}X_H
\qquad\textrm{and}\qquad
V_H = \frac{1}{i}R_H^{-\frac{1}{2}}Y_H ,
$$
and we set
$$
\Omega_H =U_H^2+ V_H^2 .
$$
Then we have 
$$
-\triangle_H =R_H \Omega_H =\Omega_H R_H ,
$$ 
 and $[U_H,V_H]= \pm\mathrm{id}$
 (according to the sign of $h_{Z_H}$), and moreover $\exp(2i\pi\Omega_H)=\mathrm{id}$.
The operators $R_H$, $U_H$, $V_H$, $\Omega_H$ are pseudo-differential operators on the cones
$\{p_z^2 > c(p_x^2+p_y^2)\}$ with $c>0$. 
Actually, this corresponds to quantizing the factorization of the principal symbol given by
$$
g^\star_H = \sigma(-\triangle_H) = h_{X_H}^2+h_{Y_H}^2 = \vert h_{Z_H} \vert \left( \left( \frac{h_{X_H}}{\sqrt{\vert h_{Z_H}\vert }} \right)^2 +
 \left( \frac{h_{Y_H}}{\sqrt{\vert h_{Z_H}\vert }} \right)^2 \right) .
$$

Note also the the heat kernel on the Heisenberg group has been computed in \cite{Ga-77}.

\paragraph{Periodic geodesics.}
It is interesting to observe that the length spectrum, i.e., the set of lengths of periodic geodesics, can be deduced from the previous computations.
The function $h_{Z_H}=p_z$ is a first integral of the geodesic flow. For $h_{Z_H}=0$, the Hamiltonian $p_x^2 + p_y^2$ generates the closed geodesics of the flat torus $\R^2 /\sqrt{2\pi }\Z^2$.
For $h_{Z_H}\neq 0$, we have $g^\star_H = \vert h_{Z_H} \vert I_H$ where $\{ I_H ,h_{Z_H} \}=0$.
The Hamiltonian flow of  $\vec{I}_H$ (resp., of $\vec{h}_{Z_H}$) is  $\pi$-periodic (resp., $2\pi$-periodic) and
we have
$$
\overset{\rightarrow}{g^\star_H} = \vert h_{Z_H} \vert \vec{I}_H + I_H \overset{\longrightarrow}{\vert h_{Z_H}\vert} .
$$
Considering the geodesic flow on the  unit cotangent bundle $g^\star_H =1$ and a periodic orbit of length $L$, we have $\vert h_{Z_H} \vert I_H =1$, and thus 
$$
\frac{\pi p}{\vert h_{Z_H} \vert}=\frac{2\pi q}{I_H}=\frac{L}{2} ,
$$
and therefore $L=2\pi \sqrt{2pq}$ with $(p,q)\in\N^2$.
 
\paragraph{Quantum limits.}
We use the specific properties of the Heisenberg flat case to
compute some quantum limits, thus giving some intuition for the
general case. 
In the result hereafter, we use the notations of Theorem \ref{thm2}, according to which any quantum limit $\beta$ for $\triangle_H$ can be written as $\beta = \beta_0+\beta_\infty$, where $\beta_\infty$ is supported on $\Sigma$ and invariant under the Reeb flow and $\beta_0$ is invariant under the sR flow. 

\begin{proposition}
\begin{enumerate}
\item If $\gamma$ is a QL for the flat torus $\R^2 / \sqrt{2\pi} \Z^2 $, then
$\beta=\gamma \otimes \vert dz \vert \otimes \delta _{p_z =0}$ is a QL
for $\triangle_H$, satisfying $\beta_\infty=0$ in the above decomposition.
\item Any probability measure $\beta$ on $\Sigma _1$ that is invariant under the Reeb flow is a QL for $\triangle_H$. For any such measure, we have $\beta_0=0$.

\item There exist quantum limits $\beta_0$ and $\beta_\infty$ for $\triangle_H$, with $\beta_0$ invariant under the sR flow and $\beta_\infty$ supported on $\Sigma$ and invariant under the Reeb flow, such that, for every $a\in[0,1]$, $\beta_a =a \beta_0 + (1-a)\beta_\infty$ is also a QL for $\triangle_H$.
\end{enumerate}
\end{proposition}

\begin{remark}
The quantum limits for flat square tori have been completely characterized by Jakobson in \cite{Jakobson-97}, and all of them are absolutely continuous.
\end{remark}

\begin{proof}
Let us prove the first point.
According to Proposition \ref{prop:spec_heisflat}, the pullback of any eigenfunction
$\phi$ of the flat torus $\R^2/\sqrt{2\pi} \Z^2$ to $M_H$  is an
eigenfunction of $\triangle_H$ that is independent of $z$. 
It follows that, for any QL $\gamma$ for the usual Laplacian on the torus, the measure $\frac{1}{2\pi }d\gamma_{x,y,p_x, p_y} \otimes |dz | \otimes \delta_{p_z=0}$ is a QL for $\triangle_H$. 

Let us prove the second point.
We now consider the QLs that are associated with the second kind of eigenvalues given by Proposition \ref{prop:spec_heisflat}. 
To any eigenvalue $\lambda_{\ell,m}= |m|(2\ell+1)$, we associate the eigenfunction $\Phi_{\ell,m}(x,y,z)=\phi_{\ell,m}(x,y)e^{imz}$, where $\phi_{\ell,m}$ is a section
of a complex line bundle on the torus. Recall that the Reeb flow is generated by
$\partial_z$, and thus all its orbits are periodic.
Considering eigenfunctions $\Phi_{0,m}(x,y,z)=\phi_{0,m}(x,y)e^{imz}$ as $m\rightarrow +\infty $, it follows from the proof of Lemma \ref{lem_spectrum} that their eigenspaces are $(\ker A_m)e^{imz}$ with
$A_m = \frac{\partial }{\partial x }+ i \frac{\partial }{\partial y } + mx $ locally.
The functions $\exp(-m \frac{x^2}{2} + \frac{m}{4} (x+iy)^2)$ belong locally to $\ker A_m$, and  are localized at $(0,0)$ as well as their Fourier transform. This implies that the associated QL is the lift to $\Sigma _1$ of $\delta _{0,0} \otimes |dz| $.
Using the fact that the spectrum of $B_m$ has uniform gaps, the same is true for the projection of this function onto $\ker A_m$. 
Note that the point $(0,0)$ plays no particular role.
Now, we use an argument due to Jakobson and Zelditch (see \cite{Jak-Zel-96}):
any convex combination of the uniform measures along Reeb orbits is a QL, because the sequences of eigenfunctions that are associated to two different circles are orthogonal at the limit. Since the latter combinations are dense in the set of invariant mesures, the second point follows.

Let us prove the third point.
The same orthogonality assumption allows one to construct more complicated examples. 
For instance, for $n\rightarrow+\infty$, we consider the eigenvalues 
$\lambda_n= (2p_0n+1)q_0 n$ for some fixed nonzero integers $p_0$ and $q_0$.
Setting $m=(2p_0n+1)q_0 n$ and $\ell =0$, we obtain a sequence
$\Phi_{0,m}$ whose QL is supported by  $\Sigma_1$. We denote by
$\beta_\infty$ the corresponding measure. If instead, we take 
$m= q_0n$ and $\ell =p_0 n,$ we get  another sequence of
eigenfunctions $\psi_n$. Taking a subsequence if necessary, we obtain
a QL $\beta_0$, supported in 
$C= \{ h_{X_H}^2+ h_{Y_H}^2 = \frac{2p_0}{q_0} h_{Z_H}^2 \}$. 
Indeed, setting $D=-i \partial _z$, $\psi_n$ satisfies $P\psi_n =0$ with 
$P=  -\triangle_H + ( \frac{2p_0}{q_0}D +1 ) D$ which is elliptic outside of $C$. 
It follows that $\phi_n$ and $\psi_n$ are asymptotically orthogonal (and
associated with the same eigenvalue). 
The convex combinations $\cos \alpha \, \phi_n+ \sin \alpha \, \psi_n$ with $\alpha $ fixed thus give the QL $\cos ^2  \alpha \, \beta_\infty+  \sin ^2 \alpha \, \beta_0$.
\end{proof}

\subsection{2D magnetic fields}
Following \cite{Mo-95}, we consider a two-dimensional closed Riemannian manifold $(X,g)$ endowed with a real two-form $B$ (the magnetic field) whose integral is an integer multiple of $2\pi$.
We associate to these data a complex Hermitian line bundle $L$ (the choice is not unique if the manifold is not simply connected)
 with a connection whose curvature is $B$, and a magnetic Schr\"odinger operator $H_B$ acting on the sections of $L$, defined
as the Friedrichs extension of $q(s)=\int_X \|\nabla s \|^2 dx_g$ on the Hilbert space of square integrable sections of $L$. 
The associated principal $\mathbb{S}^1$-bundle on $X$ is denoted by $M$, and is equipped with the measure $\mu=dx_g\otimes |d\theta |$.
We take as a distribution on $M$ the horizontal space of the connection and as a metric on it the pullback of the metric on $X$.
If $B$ does not vanish, the distribution $D$ is of contact type.
Let us consider the sR Laplacian $\triangle$ associated to these data on $M$.
Using a Fourier decomposition with respect to the $\mathbb{S}^1$-action on $L^2(M,\mu)$,  
 we write $L^2(M,\mu)=\oplus _{n\in \Z}L^2(X,L^{\otimes n})$ and we have $\triangle=\oplus _n H_{nB}$.
The associated Reeb vector field is  given by 
$$
Z=b \partial  _\theta -\vec{B},
$$
where $B=b\, dx_g$ and $\vec{B}$ is the horizontal lift of the Hamiltonian vector field of $B$ on $X$ associated with the symplectic form $B$. 
Note that these Reeb dynamics are not ergodic: they are integrable and $b$ is an integral of the motion.
This example is a generalization of the Heisenberg case considered in Section \ref{sec:heis}: in this case,
$X$ is the flat torus $\R^2_{x,y} /\sqrt{2\pi} \Z^2 $, the magnetic field is $B=dx\wedge dy$  and the Heisenberg manifold $M$
 is the principal bundle over $X$.
For an expanded version of this section, we refer the reader to \cite{yCdV-17}.

\subsection{Examples of ergodic Reeb flows on 3D manifolds}
In this section, we give examples of 3D contact structures for which the Reeb flow is ergodic.
 We use the well known fact that a Reeb flow can be realized, in some appropriate context, as a geodesic flow
 or as an Hamiltonian flow. We provide two constructions.

\paragraph{Geodesic flows.}
Let $(X,h)$ be a 2D compact Riemannian surface, and let $M=S^\star X$ be the unit cotangent bundle of $X$.
The 3D compact manifold $M$ is then naturally endowed with the contact form $\alpha$ defined as the restriction to $M$ of the Liouville 1-form $\Lambda= p\, dq$. Let $Z$ be the associated Reeb vector field.
Identifying the tangent and cotangent bundles of $X$ thanks to the Riemannian metric $h$, the set $M$ 
is viewed as the unit tangent bundle of $X$.
Using a metric $g$, for example the canonical metric, also called Sasaki metric,
 such that the restriction of the symplectic form to $D$ is the volume form of $g$, $Z$ is identified with
 the vector field on the unit tangent bundle of $X$ generating the geodesic flow on $S^\star X$. Therefore,
 with this identification, the Reeb flow is the geodesic flow on $M$.

Now, the geodesic flow is ergodic (and hence, the Reeb flow is ergodic with respect to the Popp measure) in the following cases:
\begin{itemize}
\item If the curvature of $X$ is negative, then the geodesic flow is ergodic. In that case, the Reeb flow is Anosov and the contact distribution is generated by the stable and unstable lines of the geodesic flow.
In the case where $X=\Gamma \setminus \H  $, with $\H$ the hyperbolic Poincar\'e disk, a rather precise study of this operator can be done using the decomposition into irreducible representations of the natural action of $SL_2(\R)$ onto $L^2 (\Gamma) \setminus SL_2(\R)$ (this example is studied in \cite{CHW-15}).
\item It is known that any compact orientable surface can be endowed with a Riemannian metric having an ergodic geodesic flow (see \cite{Do-88}). 
\end{itemize}
There are explicit examples with $M=\mathbb{P}^3(\R)$ seen as the unit tangent bundle of $(\mathbb{S}^2,g)$
 where the geodesic flow of $g$ is ergodic (see \cite{Do-88}).

We refer to \cite{Fo-Ha-13} for other examples of Anosov contact flows (which can be realized as Reeb flows) on 3D manifolds, obtained by surgery from unit cotangent bundles of hyperbolic surfaces.

\paragraph{Hamiltonian flows.}
We use the relationship with symplectic geometry (see, e.g., \cite{Geiges}).

Let $(W,\omega)$ be a symplectic manifold of dimension $4$, and let $M$ be a submanifold of $W$ of dimension $3$, such that there exists a vector field $v$ on a neighborhood of $M$ in $W$, satisfying $\mathcal{L}_v\omega=\omega$ (Liouville vector field), and transverse to $M$. Then the one-form $\alpha=\iota_v\omega$ is a global contact form on $M$, and we have $d\alpha=\omega$. Note that, if $\omega=-d\Lambda$ is exact, then the vector field $v$ defined by $\iota_v\omega=\Lambda$ is Liouville (in local symplectic coordinates $(q,p)$ on $W$, we have $v=p\,\partial_p$). If the manifold $M$ is moreover a level set of an Hamiltonian function $h$ on $W$, then the Reeb flow on $M$ (associated with $\alpha$) is a reparametrization, possibly with a nonconstant factor, of the Hamiltonian flow restricted to $M$.

If $D=\ker\alpha$ is moreover endowed with a Riemannian metric $g$,
then the contact form $\alpha_g$ constructed in Section \ref{sec_Popp} is collinear to $\alpha$, that is, $\alpha_g = h \alpha$
 for some smooth function $h$ (never vanishing). 
Let us then choose the metric $g$ such that $h=1$ (it suffices to take $g$ such that $(d\alpha)_{\vert D}$ coincides with the Riemannian volume induced by $g$ on $D$).

Then, the Reeb flow is ergodic on $(M,\nu)$ (where the Popp measure $\nu$ is defined as in Section \ref{sec_Popp}) if and only if the Hamiltonian flow is ergodic on $(M,|\alpha \wedge \omega |)$. This gives many possible examples of an ergodic Reeb flow.

 
\section{The local and microlocal Weyl laws}\label{sec:micWeyl}

\subsection{Definitions}

The definitions and results given in this subsection are general.
Let $(\phi_n)_{n\in\N^*}$ be an orthonormal basis of eigenfunctions of $\triangle_{sR}$.
Let $N(\lambda)=\# \{n\in\N^* \mid \lambda_n \leq \lambda \}$ be the spectral counting function.

\begin{definition}
For every bounded linear operator $A$ on $L^2(M,\mu)$, we define the Ces\`aro mean $E(A)\in \C$ by
\begin{equation*}
E(A) = \lim_{\lambda \rightarrow +\infty} \frac{1}{N(\lambda)}\sum_{\lambda_n \leq \lambda } \left\langle A \phi_n , \phi_n \right\rangle,
\end{equation*}
whenever this limit exists.
We define the variance $V(A)=V(A,(\phi_n)_{n\in\N^*}) \in [0,+\infty]$ by
\begin{equation*}
V(A) = \limsup _{\lambda \rightarrow +\infty } \frac{1}{N(\lambda)}\sum _{\lambda_n \leq \lambda} \left\vert\left\langle A \phi_n ,\phi_n \right\rangle \right\vert^2 .
\end{equation*}
\end{definition}
Note that $E(A)$  does not depend on the choice of the eigenbasis, because it is defined as the limit of suitable traces.
In contrast, $V(A)=V(A,(\phi_n)_{n\in\N^*})$ depends on the choice of the eigenbasis.
Actually, since $V(A)$ is thought as a variance, it would be more natural to replace $A$ with $A-E(A)$ in the scalar product at the right-hand side, but this definition is actually more convenient in what follows, and anyway we will mainly deal with operators of zero Ces\`aro mean.

The following definitions are obtained by restricting either to operators $A$ that are given by the multiplication by a function $f$ or to operators that are obtained by quantizing a symbol  $a:S^\star M\rightarrow \R$ of order $0$ into a pseudo-differential operator $\Op(a)\in\Psi^0(M)$ of order $0$.

\begin{definition}\label{defi:weyl-measure}
Given any continuous function $f:M\rightarrow \R $, we consider the operator
$A_f$ of multiplication by $f$, and we assume that the limit defining $E(A_f)$ exists. 
Then, the \textit{local Weyl measure} $w_\triangle $ is the probability measure on $M$ defined by
$$
\int _M f \, dw_\triangle = \lim _{\lambda \rightarrow +\infty }\frac{1}{N(\lambda)}\sum _{\lambda_n \leq \lambda } \int_M f |\phi_n |^2 \, d\mu  = E(A_f).
$$
In other words, $w_\triangle$ is the weak limit (in the sense of measures), as $\lambda \rightarrow +\infty$, of the probability measures on $M$ given by $\frac{1}{N(\lambda)}\sum _{\lambda_n \leq \lambda } |\phi_n |^2 \mu$.

The \textit{microlocal Weyl measure} $W_\triangle $ is the probability measure on $S^\star M$ defined by\footnote{Recall that positively homogeneous functions of order $0$ on $T^\star M$ are identified with functions on $ S^\star M$.}
$$
\int _{S^\star M} a \, dW_\triangle = \lim _{\lambda \rightarrow +\infty }\frac{1}{N(\lambda)}\sum _{\lambda_n \leq \lambda } 
\langle \Op (a)\phi_n , \phi_n \rangle = E(\Op(a)) ,
$$
for every symbol $a:S^\star M \rightarrow \R $ of order zero, whenever the limit exists
for all symbols $a$ of order $0$.
\end{definition}

When they exist, the local and microlocal Weyl measures do neither depend on the choice of the orthonormal basis, nor on the choice of the quantization $a \mapsto \Op(a)$. 
Moreover, we have 
$$
\pi_* W_\triangle=w_\triangle ,
$$
i.e., the local Weyl measure $w_\triangle$ is the pushforward of the microlocal Weyl measure $W_\triangle$ under the canonical projection $\pi:T^*M\rightarrow M$.
The fact that the limits defining the microlocal Weyl measures are indeed probabilities measures follows from the Garding inequality (which implies that, if $a \geq 0$, then 
$\liminf _{n\rightarrow +\infty }\langle \Op(a) \phi_n , \phi_n \rangle \geq 0 $) and from a theorem of L. Schwartz according to which positive distributions are Radon measures.

We state hereafter several general lemmas that will be useful in the sequel.

\begin{lemma}\label{remCS}
\begin{enumerate}
\item The set of $A\in\Psi^0(M)$ such that $V(A)=0$ is a linear subspace.
\item  Denoting by  $A^\star$ the adjoint of $A\in\Psi^0(M)$ in $L^2(M,\mu)$, we have 
\begin{equation}\label{ineqV}
V(A)\leq E(A^\star A),
\end{equation}
when the right-hand side is well defined.
\item Let $A$ and $B$ be bounded linear operators on $L^2(M,\mu)$. If $V(B)=0$ then $V(A+B)=V(A)$.
\end{enumerate}
\end{lemma}
\begin{proof}
Using the Cauchy-Schwarz inequality and the Young inequality we have $V(A+B)\leq 2\left( V(A)+V(B) \right)$, for all bounded linear operators $A$
 and $B$ on $L^2(M,\mu)$. The first claim follows.
The second also follows from the Cauchy-Schwarz inequality, since $\vert \left\langle A \phi_n ,\phi_n \right\rangle \vert^2 \leq \left\langle A^\star A \phi_n ,\phi_n \right\rangle$, for every integer $n$. Let us prove the third point. Using the generalized Young inequality, we have
$$
\langle (A+B)\phi_n,\phi_n \rangle ^2 \leq (1+ \varepsilon)\langle A\phi_n, \phi_n \rangle ^2+ (1+\frac{1}{\varepsilon}) \langle B\phi_n, \phi_n \rangle ^2 ,
$$
for every $\varepsilon  >0$, from which we infer that $V(A+B)-V(A) \leq \varepsilon V(A)+(1+\frac{1}{\varepsilon})V(B)$. Using that $V(B)=0$ and letting $\varepsilon$ go to $0$, we get that $V(A+B)\leq V(A)$. The converse inequality is obtained by writing $A=A+B-B$ and by following the same argument.
\end{proof}

\begin{lemma}\label{lemm:compact}
If $A$ is a compact operator on $L^2(M,\mu)$, then $E(A)=0$ and $V(A)=0$. 
\end{lemma}
\begin{proof}
If $A$ is compact, then $\Vert A \phi_n\Vert\rightarrow 0$ as $n\rightarrow +\infty$. The lemma follows, using the Cauchy-Schwarz inequality.
\end{proof}

This lemma can in particular be applied to pseudo-differential operators of negative order. This gives the following result.

\begin{proposition}
Given any pseudo-differential operator $A$ of order $0$, $E(A)$ and $V(A)$ depend only
on the principal symbol of $A$ (and on the choice of a Hilbert basis of eigenfunctions in the case of $V(A)$).
\end{proposition}

\begin{lemma}\label{lem_mesureWeyleven}
The microlocal Weyl measure $W_\triangle $ is even with respect to the canonical involution of $S^\star M$. 
\end{lemma}

\begin{proof}
A Weyl quantization is constructed by covering $M$ with a finite number of coordinate charts. Once this covering is fixed, by using a smooth partition of unity, it suffices to define the quantization of symbols that are supported in one of the chosen coordinate charts and there we choose the Weyl quantization. We denote it by $\OpW$ (see Appendix \ref{app:PDO}). 

For any orthonormal basis $(\phi_n)_{n\geq 1}$, for any integer $n$ we have
(for a symbol $a$ supported in a coordinate chart)
\begin{equation}\label{formula_conjugation}
\begin{split}
\langle \OpW(a) \phi_n,\phi _n \rangle & = (2\pi)^{-3}\int e^{i\langle
  x-y,\xi\rangle}a\left(\frac{x+y}{2},\xi\right) \phi_n(x)\overline{\phi_n(y)}\,
dx\,dy\,d\xi \\
&=(2\pi)^{-3}\int e^{i\langle
  x-y,\xi\rangle} a\left(\frac{x+y}{2},-\xi\right) \overline{\phi_n(x)}\phi_n(y)\,
dx\,dy\,d\xi\\
&= \langle \OpW(\tilde{a}) \overline{\phi_n},\overline{\phi _n}
\rangle ,
\end{split}
\end{equation}
where we have set $\tilde{a}(q,p)=a(q,-p)$.
Since $\triangle_{sR}$ commutes with complex conjugation, $(\overline{\phi_n})_{n\geq 1}$ is a Hilbert basis of eigenfunctions as well. Taking the limit $n\rightarrow +\infty$, it follows that $\int \tilde{a} \, dW_\triangle = \int a \, dW_\triangle$,
since the limit defining $W_\triangle$ does not depend on the Hilbert basis of
eigenfunctions. 
\end{proof}

The above definitions and results make sense in a general setting. We next identify the microlocal Weyl measure in the 3D contact case.


\subsection{The microlocal Weyl law in the 3D contact case}
\begin{theorem}\label{theo:weyl}
Let $A\in \Psi^0(M)$ be a pseudo-differential operator of order $0$ with principal symbol $a\in \mathcal{S}^0(M)$. In the 3D contact case, we have
\begin{equation*}
\sum_{\lambda_n \leq \lambda } \left\langle A \phi_n , \phi_n \right\rangle =
 \frac{P(M)}{64} \lambda^2 \int_M \big( a(q,\alpha_g (q))  +  a(q,-\alpha_g (q)) \big) \, d\nu + \mathrm{o}(\lambda^2) ,
\end{equation*}
as $\lambda\rightarrow+\infty$.

In particular, it follows that $N(\lambda) \sim \frac{P(M)}{32}\lambda^2$,
that the local and microlocal Weyl measures exist,
and that
\begin{equation}\label{localweylformula}
E(A) = \int _{S^\star M}a \, dW_\triangle = \frac{1}{2}  \int_M \big( a(q,\alpha_g (q)) + a(q,-\alpha_g (q)) \big) \, d\nu .
\end{equation}
\end{theorem}

\begin{remark}\label{rem9} 
Theorem \ref{theo:weyl} (proved in Section \ref{proof_microlocal}) says that the eigenfunctions ``concentrate'' on $\Sigma $. 
In our proof, we are able to infer the microlocal Weyl measure from the local one, because the distribution $D$ is of codimension one.
\end{remark}

\begin{remark}\label{rem:Weyl}
Equation \eqref{localweylformula} can be rewritten as
$$
dW_\triangle= \frac{1}{2} (d\hat{\nu}_1 + d\hat{\nu}_{-1}),
$$
where $\hat\nu_{\pm 1}$ is the lift of the Popp probability measure measure $\nu$ to $\Sigma_{\pm 1}$.
It also follows that 
$$
w_\triangle=\nu ,
$$
i.e., the local Weyl measure coincides with the Popp probability measure in the 3D contact case.
\end{remark}

Note that, due to \eqref{localweylformula}, $E(A)$ only depends on $a_{\vert\Sigma}$.
For $V(A)$, we have the following. 

\begin{corollary}\label{cor_weyl}
\begin{enumerate}
\item For every $A\in \Psi^0(M)$ whose principal symbol vanishes on $\Sigma$, we have  $V(A)=0$.
\item Let $A,B\in \Psi^0(M)$, with principal symbols $a,b\in\mathcal{S}^0(M)$ such that $a_{\vert \Sigma}=b_{\vert \Sigma}$. Then $V(A)=V(B)$.
\end{enumerate}
\end{corollary}

\begin{proof}[Proof of Corollary \ref{cor_weyl}]
The principal symbol of $A^\star A$ is $\vert a\vert^2$. If $a_{\vert\Sigma}=0$,
 then it follows from the microlocal Weyl formula \eqref{localweylformula} that $E(A^\star A)=0$,
 and hence $V(A)=0$. For the second point, writing $A=B+A-B$, we have $V(A-B)=0$ by the first point, and
 the conclusion follows by using Lemma \ref{remCS}. 
\end{proof}


\subsection{Proof of the microlocal Weyl law}\label{proof_microlocal}
In this section, we provide a general argument showing how one can derive the microlocal Weyl measure from the local one.
We first recall the local Weyl law.

\begin{theorem}\label{theo:locweyl}
For any 3D contact sR Laplacian, we have
$$
N(\lambda )\sim \frac{P(M)}{32} \lambda ^2 ,
$$
as $\lambda\rightarrow+\infty$, and $w_\triangle=\nu $ (Popp probability measure).
\end{theorem} 

The existence of a local Weyl measure in the 3D contact case has been implicitly established by many authors, without giving a name to it, or without giving an explicit expression of it (see for example \cite{Metivier1976}). In the forthcoming work \cite{CHT-II}, we give a very general proof (in a much more general context) inspired by the paper \cite{Ba-13}, based on the use of the so-called \textit{privileged coordinates} in sub-Riemannian geometry. In the 3D contact case, Theorem \ref{theo:locweyl} follows from results of \cite{Me-Sj-78}, where the authors establish the heat kernel asymptotics
$$
\int _M e(t,q,q)f(q)\, d\mu(q) \sim \frac{1}{4\pi ^2 t^2 }\left(\sum _{\ell=0}^{+\infty} \frac{1}{(2\ell+1)^2 }\right)\int _{\Sigma \cap  \{ \rho \leq 1 \}}
f(q) \, d\sigma (p,q)  ,
$$
as $t\rightarrow 0^+$, where $d\sigma $ is the symplectic measure on $\Sigma$, and where the heat kernel $e(t,q,q')$ is the Schwartz kernel of $e^{t\triangle_{sR}}$.
We then use the fact that $d\sigma =d\nu\otimes d\rho $ and the theorem follows by applying the Karamata tauberian theorem (see \cite{Ka-30}).

\medskip

Let us now establish the microlocal Weyl law.
Since our argument actually works in a more general setting than in the 3D contact case, we provide hereafter some results that are valid in a general sR context.

Let $(M,D,g)$ be a sR structure, where $M$ is a compact manifold of dimension $d$, $D$ is a subbundle of $TM$ and $g$ is a Riemannian metric on $D$. Let $\mu$ be a smooth density on $M$, and let $\triangle _{sR}$ be a sub-Riemannian Laplacian on $M$ (see Section \ref{sec_DeltaSR}).
In this general context, the characteristic manifold of $\triangle _{sR}$ is still defined by $\Sigma=D^\perp$. 
We assume that $\mathrm{Lie}(D)=TM$ (this is H\"ormander's assumption, which implies hypoellipticity),
 so that $\triangle _{sR}$ has a discrete real spectrum.

We have the following result on its spectral counting function $N(\lambda)=\# \{n \mid \lambda_n \leq \lambda \}$.

\begin{proposition}\label{prop_largeweylSR}
If the codimension of $D$ in $TM$ is positive, then 
$$
\lim_{\lambda \rightarrow +\infty} \frac{N(\lambda )}{\lambda ^{d/2}} =+\infty.
$$
\end{proposition}

\begin{proof}
Let $N$ be a subbundle of $TM$ such that $TM=D \oplus N$. Let $h$ be an arbitrary metric on $N$. For every $\varepsilon>0$, we consider the Riemannian metric $g_\varepsilon =g \oplus \varepsilon ^{-1}h$, and we consider the corresponding Riemannian Laplacian $\triangle_{g_\varepsilon, \mu}$. 
We have the two following facts.

First, denoting by $c$ the codimension of $D$ in $TM$, the spectral counting function of $\triangle_{g_\varepsilon,\mu}$ satisfies $N_\varepsilon (\lambda )  \sim C\varepsilon ^{-c/2}\lambda ^{d/2}$, for some constant $C>0$. Indeed, this follows from the Weyl law for the Riemannian Laplacian $\triangle_{g_\varepsilon,\mu}$, which is valid even though the measure $\mu$ is not the Riemannian volume, combined with the fact that $\mathrm{Vol}(M,\vert dq\vert_{g_\varepsilon}) =\varepsilon ^{-c/2}\mathrm{Vol}(M,\vert dq\vert_{g_1})$.

Second, with obvious notations, we infer from a minimax argument that $\lambda_n (\triangle_{g,\mu})\leq \lambda_n (\triangle_{g_\varepsilon,\mu})$, for every $n \in \N $ and every $\varepsilon >0$, because $ g^\star \leq  g_\varepsilon^\star =g^\star+\varepsilon h^\star$, while the $L^2$ norms are the same.

These facts imply that $N(\lambda) \geq N_\varepsilon(\lambda)$, and hence 
$$
\liminf_{\lambda \rightarrow +\infty } \frac{N(\lambda )}{ \lambda ^{d/2}} \geq  \lim_{\lambda \rightarrow +\infty }\frac{N_\varepsilon (\lambda )}{ \lambda ^{d/2}}=
\frac{C}{\varepsilon^{c/2}} ,
$$
for every $\varepsilon >0$. The result follows.
\end{proof}

\begin{remark} 
Note that, using similar  arguments, we have
$\lim_{ \varepsilon \rightarrow 0^+ } \lambda_n (\triangle_{g_\varepsilon,\mu})=\lambda_n (\triangle_{sR})$.
This fact implies that many Riemannian statements concerning $\lambda_n$ for $n$ fixed are as well valid for sR Laplacians. 
\end{remark}

We have the following consequence.

\begin{proposition} \label{prop:support}
If the microlocal Weyl measure $W_\triangle $ exists, then $\mathrm{supp}(W_\triangle)\subset S \Sigma $.
\end{proposition}

A more precise result will be given in \cite{CHT-II} using an appropriate approximation of the heat kernel near the diagonal.

\begin{proof}
We start with the following lemma.

\begin{lemma}\label{lemm:ell} 
Let $A\in\Psi^0(M)$ be a pseudo-differential operator whose principal symbol is equal to zero in a neighborhood of $\Sigma=D^\perp$. 
Then
$$
\sum _{\lambda_n \leq \lambda }| \langle A \phi_n ,\phi_n \rangle | = \mathrm{O}\big( \lambda^{d/2} \big),
$$
as $\lambda\rightarrow+\infty$.
\end{lemma}

\begin{remark}\label{rem:sqrt}
The exponent $\frac{d}{2}$ is the one that we would obtain in the classical elliptic case. Outside of $\Sigma$, the operator $\triangle_{sR}$ is elliptic, and it follows that 
the operator $\sqrt{-\triangle_{sR}}$ (which is defined using functional calculus) is, away of $\Sigma$, a pseudo-differential operator with principal symbol $\sqrt{\sigma_p(-\triangle_{sR})}$ (see \cite[Corollary 9]{HV-00}).  
\end{remark}

\begin{proof}[Proof of Lemma \ref{lemm:ell}.]
We define $P=\sqrt{-\triangle_{sR}}$, $s_n=\sqrt{\lambda_n}$ and $p=\sqrt{g^\star}$.
Let $A\in\Psi^0(M)$ be a pseudo-differential operator whose principal symbol is equal to zero in a neighborhood of $\Sigma=D^\perp$. Without loss of generality, we assume that $A$ is positive. We define $a_n =\langle A\phi_n ,\phi_n \rangle $, for every $n\in\N^*$. The proof follows closely the arguments of \cite{Ho-68} (see, in particular, the end of Section 3 in this paper) and \cite[Section 2]{DG-75}. 
Let $C_1\subset T^\star M $ be a closed cone whose interior contains $WF'(A)$ (see the definition in Appendix \ref{app:PDO}), disjoint from $\Sigma $, and let 
$\varepsilon >0 $ be such that the sR geodesic flow (which coincides with the flow of $\vec{p}$) satisfies
$\phi_t (WF'(A)) \subset C_1$ if $|t|\leq \varepsilon $.
Let $\rho \in \mathcal{S}(\R)$ (Schwartz space) be a nonnegative function such that
$\rho \geq c >0$ on $[-1, 0]$ and whose Fourier transform
$\hat{\rho}(t)=\int_\R \rho(s)e^{-its}\,ds$ is supported in the interval $(-\varepsilon, +\varepsilon)$.
We consider the distribution $Z_A$ on $\R$ defined by 
$$
Z_A(t)= \mathrm{Tr}\left(e^{-i tP} A \right)=\sum _{n=1}^{+\infty} a_ne^{-its_n} .
$$
Multiplying by $\hat\rho(t)e^{its}$ and integrating over $\R$, we infer that
\begin{equation}\label{eqZA}
\mathrm{Tr} \left( \frac{1}{2\pi} \int_{\R} \hat{\rho}(t)e^{its}e^{-itP}A\, dt\right)\,=\,\sum _{n=1}^{+\infty} a_n \rho (s -s_n)= \frac{1}{2\pi}\int _\R e^{its} Z_A(t) \hat{\rho}(t) dt ,
\end{equation}
for every $s\in\R$.

Let us prove that the right-hand side of \eqref{eqZA} is $\mathrm{O}\left(s^{d-1}\right)$, as $s\rightarrow+\infty$.
Let $E$ be an elliptic pseudo-differential operator of order $1$ such that $E-P $ is smoothing near $C_1$.
We claim that the mapping $t \mapsto R(t)= e^{-itP }A -e^{-itE}A$ is smooth with values in trace-class operators  for $|t|\leq \varepsilon$.
This will be enough to infer the estimate $\mathrm{O}\left(s^{d-1}\right)$, since this estimate is known for the elliptic operator $E$ (see \cite[Section 2]{DG-75}).
It follows from the Duhamel formula that
$$
e^{-itP }A -e^{-itE}A = i\int _0^t e^{i(s-t)P}(P-E)e^{-isE}A \, ds .
$$
The right-hand side is smoothing for $|t|\leq \varepsilon$, because
$WF\left(e^{-isE}Au\right) \subset C_1$ for any distribution $u$ on $M$, and $P-E$ is smoothing on $C_1$.
Hence it is trace-class. Besides, it clearly depends smoothly on $t$. 
We get then that $\mathrm{Tr} (e^{-itP }A -e^{-itE}A)\hat{\rho}(t)=Z_A(t)\hat\rho(t)-Z_E(t)\hat\rho(t)$ is smooth with compact support and thus has a fast decaying Fourier transform. We have thus proved that the right-hand side of \eqref{eqZA} is $\mathrm{O}\left(s^{d-1}\right)$, as $s\rightarrow+\infty$.

As a consequence, we have
$$
\sum  _{ \{ n\ \mid\ s-1 \leq s_n \leq s \}} a_n\leq \frac{1}{c} \sum _{n=1}^{+\infty} a_n  \rho (s - s_n)=\mathrm{O} \left(s^{d-1}\right) ,
$$
and hence $\sum _{ s_n \leq s } a_n = \mathrm{O}\left( s^d \right)$, as $s\rightarrow+\infty$. The lemma is proved.
\end{proof}

Now, it follows from Proposition \ref{prop_largeweylSR} and Lemma \ref{lemm:ell} that, if $A\in\Psi^0(M)$ is microlocally supported in $T^\star M\setminus\Sigma$, then 
$$
\frac{1}{N(\lambda)} \sum_{\lambda_n\leq \lambda} \vert \langle A\phi_n , \phi_n\rangle\vert \longrightarrow 0,  
$$
as $\lambda\rightarrow+\infty$. This proves Proposition \ref{prop:support}.
\end{proof}

\begin{corollary} \label{coro:codim1}
If the horizontal distribution $D$ is of codimension $1$ in $TM$, and if the local Weyl measure $w_\triangle$ exists, then the microlocal Weyl measure $W_\triangle$ exists and is equal to half of the pullback of $w_\triangle$ by the double covering $S \Sigma \rightarrow M $ which is the restriction of the canonical projection of $T^\star M $ onto $M$.
\end{corollary}

\begin{proof} 
It suffices to consider symbols that are even with respect to the involution $(q,p)\rightarrow (q,-p)$, because we already know that $dW_\triangle$ is even.
Using Proposition \ref{prop:support} and a density argument, we obtain that, if $a:S^\star M \rightarrow \R$ is continuous and vanishes on $\Sigma$, then $W_\triangle(a)$ is well defined and vanishes. Now, let $a$ be a general continuous function. Using the fact that $D$ is of codimension one, we write $a = a-\tilde{a}+\tilde{a}$, where $\tilde{a}$ is the microlocal lift to $\Sigma$ of the function $a$ on the base. By the above reasoning, $W_\triangle(a-\tilde{a})$ is well defined and vanishes. Now, by construction, $W_\triangle(\tilde{a})$ is well defined and is expressed with the local Weyl measure $w_\triangle$.   
\end{proof}

Let us now conclude the proof of Theorem \ref{theo:weyl}, in the 3D contact case.
It follows from Theorem \ref{theo:locweyl}  that the local Weyl measure exists and coincides with the Popp probability measure.
Then Theorem \ref{theo:weyl} follows from Corollary \ref{coro:codim1}.


\section{Birkhoff normal form}\label{sec:melrose}
In this section, we derive a normal form for the principal symbol of a 3D contact sub-Riemannian Laplacian, in the spirit of a result by Melrose in \cite[Section 2]{Me-84}. This normal form implies in particular that, microlocally near the characteristic cone, all 3D contact sub-Riemannian Laplacians (associated with different metrics and/or measures) are equivalent.

\medskip

Recall that the characteristic cone $\Sigma =(g^\star)^{-1}(0)$, given by \eqref{defSigma}, is a symplectic conic submanifold of
 $T^\star M\setminus \{0\}$ (the restriction $\omega_{\vert\Sigma}$ is symplectic),
parametrized by $(q,s)\mapsto (q,s\,\alpha_g(q))$ from
 $M\times(\R\setminus \{0\})$ to $T^\star M \setminus \{0\}.$ 
The function $\rho:\Sigma\rightarrow\R$ defined by
$\rho(s\,\alpha_g)=s$ is the Reeb Hamiltonian on $\Sigma$. According
to Lemma \ref{lemmaReeb} we also have $\rho= (h_Z)_{|\Sigma}$.
We also recall that 
\begin{equation*}
\begin{split}
\Sigma^+ &= \rho^{-1}((0,+\infty)) = \{ (q,s\, \alpha_g(q))\in T^\star M \ \mid\ s>0 \}  , \\
\Sigma^- & = \rho^{-1}((-\infty,0)) = \{ (q,s\, \alpha_g(q))\in T^\star M \ \mid\ s<0 \} ,
\end{split}
\end{equation*}
are positive conic submanifolds of $T^\star M$.
For every $q\in M$, we denote by $\Sigma_q^{\pm}$ the fibers in $\Sigma^\pm$ above $q$,
 that is, the half-lines generated in $\Sigma^{\pm}$ by $\alpha_g(q)$.
 
Given $k\in \N\cup \{+\infty\}$ and given a smooth function $f$ on $T^\star M$, the notation $f=\mathrm{O}_\Sigma(k)$
 means that $f$ vanishes along $\Sigma$ at order $k$ (at least). The word \textit{flat} is used when $k=+\infty$.

Given any open subset $U\subset M$, the notations $\Sigma_U$, $T^\star_U M$, \ldots, stand for the intersections of $\Sigma$, $T^\star M$, \ldots, with $T^\star U $.

\medskip

We establish in Section \ref{ss:classical} a classical Birkhoff normal form for $g^\star$, in three steps:
\begin{itemize}
\item In Sections \ref{sec_constructing} and \ref{sec_DW}, we give a normal form for the Hessian of $g^\star $ along $\Sigma$, thus obtaining a symplectic normal form modulo $\mathrm{O}_\Sigma(3)$.
\item In Section \ref{sec:semiglobal}, we derive a symplectic Birkhoff normal form 
$$
g^\star \sim  \sum_{j=1}^{+\infty} \rho_j I^j+  \mathrm{O}_\Sigma (\infty) ,
$$
where $I$ is a classical harmonic oscillator, Poisson commuting with all functions $\rho_j$, and where $\rho_1=\rho$ is the Reeb Hamiltonian.

These two first normal forms are valid globally near $\Sigma_U $, where $U\subset M$
is any open set on which the distribution $D$ is trivial, that is,
spanned by two vector fields $X$ and $Y$ globally defined on $U$. 
\item In Section \ref{sec:local}, we establish a symplectic normal form 
$$
g^\star \sim  \rho I + \mathrm{O}_\Sigma (\infty) ,
$$
which was initially found by Melrose. This normal form is global near $\Sigma_U$, where  $U\subset M$
is any open set that is \textit{Reeb trivial} according to the
following definition.
\end{itemize}

\begin{definition}\label{def:Reebtrivial}
An open set $U$ is said to be \textit{Reeb trivial} if:
\begin{enumerate}
\item the distribution $D$ is trivial in $U,$
\item for any smooth function $g$ defined in $U$, there exists a smooth solution $f$ defined in $U$ to the equation $Zf=g.$ 
\end{enumerate}
\end{definition}

In Section \ref{sec:QNF}, we quantize the above normal form, thus obtaining a quantum normal form that will be used to prove our main results.


\subsection{Classical normal form}\label{ss:classical}
The objective of this section is to design a homogeneous Birkhoff normal form of $g^\star$ along $\Sigma$. 

\paragraph{Symplectic normal bundle of $\Sigma$.}
In what follows, we focus on $\Sigma^+ = \rho^{-1}((0,+\infty))$ (the results for $\Sigma^-$ being similar, changing signs adequately).
 We denote by $T^\star M^+$ the open sub-cone of $T^\star M$ on which $h_Z>0$.

Since $ \Sigma^+$  is symplectic, for every $\sigma\in\Sigma^+$, we have the symplectic orthogonal decomposition
\begin{equation}\label{symplecticorthogonalsplitting}
T_\sigma(T^\star M^+) = T_\sigma \Sigma^+ \oplus \mathrm{orth}_\omega(T_\sigma \Sigma^+) ,
\end{equation}
where the notation $\mathrm{orth}_\omega$ stands for the symplectic $\omega$-orthogonal complement.
We define the symplectic normal bundle $N\Sigma^+$ of $\Sigma^+$ by
$$
N\Sigma^+ = \{ (\sigma,w)\ \mid\ \sigma\in\Sigma^+,\ w\in
N_\sigma\Sigma^+ \} \,\subset\, T(T^\star M^+),
$$
where
$N_\sigma\Sigma^+ = \mathrm{orth}_\omega(T_\sigma \Sigma^+) $.
The set $N\Sigma^+$ is a vector bundle over $\Sigma^+$, with fibers $N_\sigma\Sigma^+ \sim T_\sigma(T^\star M^+)/T_\sigma\Sigma^+$
 that are two-dimensional, and the manifold $\Sigma^+$ is canonically embedded in $N\Sigma^+$, by the zero section. 
We denote by $P$ the projection of $N\Sigma^+ $ onto $\Sigma ^+$. 

According to the tubular neighborhood theorem, $T^\star M^+$ is diffeomorphic to $N\Sigma^+$ in a neighborhood
 of $\Sigma^+$. More precisely, there exist a convex\footnote{A neighborhood $C_0$ of $\Sigma^+$
 in $N\Sigma^+$ is said to be convex if the intersection $C_0\cap N_q\Sigma^+$ with any fiber is convex.}
 neighborhood $C_0$ of $\Sigma^+$ in $N\Sigma^+$, a neighborhood $C$ of $\Sigma^+$ in $T^\star M^+$, and
 a diffeomorphism $\varphi:C_0\rightarrow C$ such that $\varphi(\sigma)=\sigma$ for every $\sigma\in\Sigma^+$.
 We will not use this diffeomorphism in what follows, but we can keep in mind that $N\Sigma^+ \simeq T^\star M^+$.

From now on, we assume that $D$ is trivial in an open subset $U\subset M$, and we will work only in $U$.
Considering a local $g$-orthonormal frame $(X,Y)$ of $D$ in $U$, we have, along $\Sigma^+_U$,
\begin{multline*}
T\Sigma^+=\ker dh_X\cap \ker dh_Y 
= \{ \xi\in T(T^\star M^+) \ \mid\ dh_X.\xi = dh_Y.\xi = 0 \}  \\
= \{ \xi\in T(T^\star M^+) \ \mid\ \omega(\vec h_X,\xi) = \omega(\vec h_Y,\xi) = 0 \} 
= \mathrm{orth}_\omega(\vec h_X,\vec h_Y)  ,
\end{multline*}
and thus
\begin{equation*}
N_\sigma\Sigma^+ = \mathrm{Span}\{ \vec h_X(\sigma), \vec h_Y(\sigma)\}  ,
\end{equation*}
along $\Sigma^+$.
This means that, with this local frame, the fibers are spanned by $\vec h_X$ and $\vec h_Y$.

Denoting by $\pi:T^\star M\rightarrow M$ the canonical projection, we
have, along $\Sigma^+$, $d\pi(\vec h_X) = X\circ\pi$ and $d\pi(\vec h_Y) = Y\circ\pi$, and
 thus $d\pi_\sigma(N_\sigma\Sigma^+)  = D_{\pi(\sigma)}$. In other words, for every $\sigma\in\Sigma^+$ the mapping
$$
\Theta(\sigma) = (d\pi_\sigma)_{\vert N_\sigma\Sigma^+}^{\vert D_{\pi(\sigma)}} : N_\sigma\Sigma^+\rightarrow D_{\pi(\sigma)}
$$
is an isomorphism, thus defining a bundle morphism $\Theta: N\Sigma^+\rightarrow D\circ\pi$ over $\Sigma$. Recalling that $D=\ker \alpha_g$
 and that $(d\alpha_g)_{\vert D}$ coincides with the oriented volume form induced by $g$, we transport this volume on fibers
 $N_\sigma\Sigma^+ $ by pullback under $\Theta$. We have $\Theta^\star d\alpha_g(\vec h_X,\vec h_Y) = d\alpha_g(X,Y)=1$.

We define $\omega_{N\Sigma^+}$ on the bundle $N\Sigma^+$ as follows: given any point $\sigma\in\Sigma^+$, we set 
$$
(\omega_{N\Sigma^+})_\sigma = (\omega_\sigma)_{\vert N_\sigma\Sigma^+} .
$$
Since $(\omega_{N\Sigma^+})_\sigma(\vec h_X(\sigma),\vec h_Y(\sigma)) = \omega_\sigma(\vec h_X(\sigma),\vec h_Y(\sigma))=\{h_X,h_Y\}(\sigma)=h_Z(\sigma)=\rho(\sigma)$, it follows that
$$
\omega_{N\Sigma^+} = \rho\, \Theta^\star (d\alpha_g)_{\vert D}.
$$
The symplectic normal bundle $N\Sigma^+$ inherits a linear symplectic structure on its fibers $N_\sigma\Sigma^+$ endowed with the symplectic form $(\omega_{N\Sigma^+})_\sigma$.
We define coordinates $(u,v)$ in $N_\sigma \Sigma^+_U$ with respect to the basis
$(\vec h_X(\sigma),\vec h_Y(\sigma))/\rho (\sigma )^{1/2}$, so that  $\omega_{N\Sigma^+}=du\wedge dv$.

In what follows, we denote by $\tilde\omega$ the  symplectic form on $N\Sigma^+_U$ defined as follows:
taking symplectic coordinates $(u,v)\in\R^2$ in the fibers $N\Sigma^+_\sigma$
 (that are symplectically orthogonal to $\Sigma$), we have $N\Sigma^+_U \simeq (\Sigma^+_U \times \R^2_{u,v})$, and we set
\begin{equation}\label{defomegatilde}
\tilde\omega = \omega_{\vert\Sigma^+} + du\wedge dv .
\end{equation}
Finally, we endow the symplectic conic manifold $(N\Sigma^+_U,\tilde\omega)$ with the conic structure defined by 
\begin{equation}\label{defconicstructure}
\tau \cdot (q,s\,\alpha_g(q) , u,v)=(q,\tau s\,\alpha_g(q) , \sqrt{\tau}u, \sqrt{\tau}v),
\end{equation}
for $\tau>0$, and the symplectic form $\tilde\omega$ is clearly homogeneous for this conic structure, in the sense that we have
 $\tau^\star (\omega_{\vert\Sigma^+} + du\wedge dv)=\tau (\omega_{\vert\Sigma^+} + du\wedge dv)$.

\paragraph{Birkhoff normal form.}

We take local coordinates $(\sigma,u,v)$ as above so that $T^\star
M^+_U \simeq N\Sigma_U^+\simeq (\Sigma^+_U \times \R^2_{u,v})$. 
On $\Sigma^+_U \times \R^2_{u,v}$ we consider the function
$(\sigma,u,v)\mapsto \rho(\sigma)$. With a slight abuse of notation, we continue to denote by $\rho$ this function.
  
For every fixed $\sigma\in\Sigma^+_U$, the function $(u,v)\mapsto
g^\star(\sigma,u,v)$ vanishes as well as its differential at
$(0,0)$. We consider its Hessian $\mathrm{Hess}(g^\star)$ and we
define the smooth function $I$ on $N_U\Sigma^+$ by
$$
I = \frac{1}{\rho} \mathrm{Hess}(g^\star).
$$
In the local symplectic coordinates $(\sigma,u,v)\in N\Sigma^+_U\simeq \Sigma^+_U \times \R^2_{u,v}$, we have 
\begin{equation*}
I(\sigma, u,v)=u^2+v^2.
\end{equation*}
Hereafter, $\Sigma^+_U \times \R^2_{u,v}$ is endowed with the symplectic form $\tilde\omega$ defined by \eqref{defomegatilde} and with the conic structure defined by \eqref{defconicstructure}.

\begin{theorem}\label{thm:BNF}
There exist a conic neighborhood $C_{0}$ of $\Sigma^{+}_{U}$ in $(T^\star U\setminus \{0\},\omega)$ and a homogeneous symplectomorphism $\chi:C_{0}\rightarrow \Sigma^+_U \times \R^2_{u,v}$ satisfying
$\chi (\sigma)=(\sigma,0)  $ for $\sigma \in \Sigma^+\cap C_{0}$,  such that 
\begin{equation*}
g^\star \circ \chi^{-1} =\rho I+\mathrm{O}_\Sigma(\infty).
\end{equation*}
\end{theorem}

In other words, this theorem says that, up to a canonical transform, $g^\star$ coincides
 with $\mathrm{Hess}(g^\star)+\mathrm{O}_\Sigma(\infty)$. This can be seen as a kind of symplectic Morse lemma along $\Sigma$.

This normal form has been obtained by Melrose in \cite[Section
2]{Me-84}, who even found a local normal form (with sketchy arguments however). Here, we establish only the formal part of the normal form, which is sufficient for our work. A proof of the full result (convergence of the normal form) is given in our work \cite{CHT-III} using a scattering method due to E. Nelson (see \cite{Nelson}).

The proof of Theorem \ref{thm:BNF} is quite long and is done in Sections \ref{sec_constructing}, \ref{sec_DW}, \ref{sec:semiglobal} and \ref{sec:local}, with some parts written in Appendix \ref{app:BNF}.
The main steps are the following:
\begin{itemize}
\item First of all, in Section \ref{sec_constructing}, we construct a 
 homogeneous diffeomorphism $\chi_1$ from a conic neighborhood $C_{0}$ of $\Sigma_U^+$ to $\Sigma ^+_U \times \R^2$
 such that 
\begin{equation*}
\begin{split}
\chi_1^* {\tilde{\omega}}&=\omega+\mathrm{O}_\Sigma(1) , \\
g^\star \circ \chi_1^{-1} & = \rho I +\mathrm{O}_\Sigma (3) .
\end{split}
\end{equation*}

\item Second, in Section \ref{sec_DW}, using a conic version of the Darboux-Weinstein lemma (stated and proved in Appendix \ref{app:DW}), we modify $\chi_1$ into a homogeneous diffeomorphism $\chi_2$ such that
\begin{equation*}
\begin{split}
\chi_2^* {\tilde{\omega}}&=\omega ,\\
g^\star\circ \chi_2^{-1} & = \rho I +\mathrm{O}_\Sigma(3).
\end{split}
\end{equation*}
In other words, we kill the remainder term $\mathrm{O}_\Sigma(1)$ in the pullback of $\tilde{\omega}$,
and thus we obtain a symplectomorphism $\chi_2$.

\item Finally, we improve the latter remainder to a flat remainder $\mathrm{O}_\Sigma(\infty)$, by solving cohomological equations 
in the symplectic conic manifold $((\Sigma^+ \times \R^2_{u,v}) \setminus \{0\},\tilde{\omega})$
(see Section \ref{sec:coeq}).
This is the most technical and lengthy part of the proof.
This is done in two steps:
\begin{enumerate}
\item In Section \ref{sec:semiglobal}, we first reduce $g^\star$ to the normal form
$\rho I + \sum _{k=2}^{+\infty} \rho_k I^k + \mathrm{O}_\Sigma (\infty )$, which is valid in  any open subset of $M$ on which $D$ is trivial. 
\item In Section \ref{sec:local}, we then reduce $g^\star$ to the normal form
$\rho I+\mathrm{O}_\Sigma (\infty) $, but this is valid only in any
open Reeb trivial subset $U$ of $M$ (i.e., on which $D$ is trivial and
on which one can solve equations $Zf=a$ globally on $U$, see Definition \ref{def:Reebtrivial}).
\end{enumerate}
\end{itemize}


\subsubsection{Construction of $\chi_1$}\label{sec_constructing}
The homogeneous diffeomorphism $\chi_1$ is constructed in an explicit way in the following result. 

\begin{proposition}\label{prop:chi1}
Let $(X,Y)$ be a smooth oriented $g$-orthonormal frame generating $D$ in the open subset $U\subset M$. The mapping 
$$
\begin{array}{rcl}
\chi_1 : T^\star U ^+ &\longrightarrow& \Sigma ^+_U \times \R^2_{u,v} \\
(q,p) & \longmapsto & \left( \big( q, h_Z(q,p) \alpha_g (q)\big) ,\left( \frac{h_X(q,p)}{\sqrt{h_Z(q,p)}},\frac{h_Y(q,p)}{\sqrt{h_Z(q,p)}} \right) \right)
\end{array}
$$
satisfies:
\begin{enumerate}
\item[(i)] $\chi_1 (\sigma)= (\sigma, 0) $ for every $\sigma \in \Sigma^+_U$;
\item[(ii)]  $\chi_1 $ is a homogeneous diffeomorphism
 of a conic neighborhood of $\Sigma ^+_U $
 in $T^\star M^+\setminus \{0\}$
onto a   conic neighborhood of $\Sigma ^+_U\times \{0\}$ in $\Sigma ^+_U \times \R^2$;
\item[(iii)] $\chi_1^\star \tilde{\omega} = \omega + \mathrm{O}_\Sigma (1)$;
\item[(iv)] $g^\star\circ\chi_1^{-1} = \rho I + \mathrm{O}_\Sigma (3)$.
\end{enumerate} 
\end{proposition}

\begin{proof}
The property (i) is obvious. 
Note also that, by construction, $\chi_1$ is homogeneous for the conic structure defined by \eqref{defconicstructure}.
The property (iv) is satisfied because
$(h_Z)_{\vert\Sigma^+}=\rho$ and 
$$
g^\star = h_X^2+h_Y^2 = h_Z \left( \left( \frac{h_X}{\sqrt{h_Z}} \right)^2 + \left( \frac{h_Y}{\sqrt{h_Z}}\right)^2 \right) .
$$
Since (iii) implies that the differential of $\chi_1$ is invertible along 
$\Sigma ^+_U $, (ii) will follow from (iii) and from the implicit function theorem. 
Now, (iii) follows from the symplectic orthogonal splitting \eqref{symplecticorthogonalsplitting} of $T_\sigma (T^\star M)$, the definition \eqref{defomegatilde} of $\tilde{\omega}$ and the choice of the coordinates $(u,v)$: indeed, noting that, since $(X,Y,Z)$ is a local frame of $TM$, we have $h_Z(q,p)>0$ for every $(q,p)\in\Sigma_U^+$ with $p\neq 0$, and defining on $\{h_Z>0\}$ the functions $u(q,p) = \frac{h_X(q,p)}{\sqrt{\vert h_Z(q,p)\vert}}$ and $v(q,p) = \frac{h_Y(q,p)}{\sqrt{\vert h_Z(q,p)\vert}}$, which are positively homogeneous of order $1/2$ with respect to $p$, we immediately get that
$({\omega_\psi})_{\vert\mathrm{orth}_\omega(T\Sigma)} = du(q,p)\wedge dv(q,p) $, and hence, for any $\psi=(q,p)\in\Sigma_U^+$ with $p\neq 0$, we have
$$
\omega_\psi = \alpha_g(q) \wedge dh_Z(q,p) - h_Z(q,p)\, d\alpha_g(q) + du(q,p)\wedge dv(q,p) .
$$
The proposition is proved.
\end{proof}

\begin{remark}
In the Heisenberg flat case, we recover the factorization of $g^\star=\sigma_P(-\triangle_{sR})$ done in Section \ref{sec:heis}, and in that case we have exactly $\chi_1^\star\tilde\omega = \omega$, without any remainder term.
\end{remark}


\subsubsection{Construction of $\chi_2$, using the Darboux-Weinstein lemma}\label{sec_DW}
In order to remove the remainder term $\mathrm{O}_\Sigma(1)$
 in $\chi_1^\star\tilde{\omega}$, we use a conic version of the Darboux-Weinstein lemma.
 This version is stated in a general version in Appendix \ref{app:DW}, Lemma \ref{lem_gen_Weinstein}.

It follows from that lemma, applied with $N=T^\star U^+$, $P=\Sigma^+_U$ and $k=1$, that there exists a homogeneous diffeomorphism $f$ defined in a conic neighborhood of $\Sigma^+_U$, such that $f^\star(\chi_1^\star\tilde{\omega})=\omega$, and such that $f$ is tangent to the identity along $\Sigma^+_U$, that is, such that $f=\mathrm{id}+\mathrm{O}_\Sigma(2)$.

We define $\chi_2=\chi_1\circ f$. Then $\chi_2$ is a local homogeneous
 diffeomorphism in some conic neighborhood of $\Sigma^+_U$, 
satisfying $\chi_2(\sigma)=(\sigma,0)$ for $\sigma\in\Sigma^+_U$, 
and such that $\chi_2^\star\tilde{\omega}=\omega$ (and thus $\chi_2$ is a symplectomorphism).

Since $f=\mathrm{id}+\mathrm{O}_\Sigma(2)$, using the fact that
 $u\, \mathrm{O}_\Sigma(2)=\mathrm{O}_\Sigma(3)$ and $v\, \mathrm{O}_\Sigma(2)=\mathrm{O}_\Sigma(3)$, we get that
$$
g^\star\circ\chi_2^{-1}  =  \rho I + \mathrm{O}_\Sigma(3) .
$$

At this step, we have thus obtained (up to a homogeneous canonical transform) the normal form with a remainder term $\mathrm{O}_\Sigma(3)$.
In order to improve this normal form at the infinite order,
 we are next going to solve a series of cohomological equations.

\subsubsection{Cohomological equations}\label{sec:coeq}
We consider the function $H=g^\star\circ \chi_2^{-1}$ defined on some open sub-cone of
$\Sigma^+_U \times \R^2$. 
According to Section \ref{sec_DW}, we have
\begin{equation*}
H=\rho I+\mathrm{O}_\Sigma(3) ,
\end{equation*}
and $H$ is homogeneous of order two with respect to the cone structure of $\Sigma^+_U \times \R^2$ defined by \eqref{defconicstructure}.

Our objective is to construct, near the half-line $\Sigma _U^+ \times  \{0\}$, a local symplectomorphism $\chi$ from $(\Sigma^+_U \times \R^2,{\tilde{\omega}})$ into itself such that 
$$
H\circ \chi = \rho I +\mathrm{O}_\Sigma(\infty) .
$$
The usual procedure, due to Birkhoff, 
and recalled, for pedagogical reasons, in Appendix \ref{app:BNF},
consists in constructing $\chi$ iteratively, by composing (symplectic) flows at time $1$ associated with appropriate Hamiltonian functions (also called Lie transforms), chosen by identifying the Taylor expansions at increasing orders, and by solving a series of (so-called) cohomological equations in the symplectic manifold $(\Sigma ^+_U \times \R^2,\tilde\omega)$. This is done at the formal level, and then the canonical transform $\chi$ is constructed by using the Borel theorem.

In the present setting, we have to adapt this general method and to define appropriate spaces of homogeneous functions and polynomials, sharing nice properties in terms of Poisson brackets. The procedure goes as follows, using what is written in Appendix \ref{app:BNF}.

\medskip

Let $C$ be a conic neighborhood of $\Sigma^+_U\times \{ 0\}$ (which will be taken sufficiently small in the sequel). For every integer $j$, we denote by $\mathcal{F}_j$ the set of functions $g$ that are smooth in $C$ and homogeneous of degree $j$ for the conic structure of $\Sigma^+ _U \times \R^2$ defined by \eqref{defconicstructure}, 
meaning that $g(q,\lambda s,\sqrt{\lambda} u,\sqrt{\lambda} v)= \lambda^j \cdot g(q,s,u,v)$, for all $\lambda >0$ and $(q,s,u,v)\in C$, in local coordinates with $\sigma=(q,s)$.
For every integer $k$, let $\mathcal{F}_{j,k}$ be the subspace of functions  of $\mathcal{F}_j$ that are homogeneous (in the classical sense) polynomials  of degree $k$ in $(u,v)$ with
coefficients which are homogeneous functions of degree $j-k/2$ in
$\Sigma ^+_U$. In the proof, we will occasionally use polar coordinates in $\R^2_{u,v}$ by setting $(u,v)=(r\cos \theta, r\sin \theta).$

For every integer $k$, we define the following subspaces of $\mathcal{F}_{j,k}$:
\begin{equation*}
\begin{split}
\mathcal{F}_{j,k}^0&=\left\{ a\in \mathcal{F}_{j,k} \ \mid\ \int_{0}^{2\pi} a(\sigma,r\cos \theta,r\sin \theta)\,
 d\theta =0,\quad \forall \sigma \in \Sigma^+_U,\quad\forall r>0 \right\},\\
\mathcal{F}_{j,k}^{\textrm{inv}}& =\left \{ (\sigma,u,v)\mapsto b(\sigma )(u^2+v^2)^{\frac{k}{2}}  \right\}.
\end{split}
\end{equation*}
The space $\mathcal{F}_{j,k}^0$ is the subset of functions of $\mathcal{F}_{j,k}$ of zero mean along circles.
We have clearly
\begin{equation}\label{sumT}
\mathcal{F}_{j,k} = \mathcal{F}_{j,k}^0\oplus \mathcal{F}_{j,k}^{\textrm{inv}},\qquad\textrm{and}\qquad \mathcal{F}_{j,k}^{\textrm{inv}} =
 \{0\}\ \textrm{if $k$ is odd}.
\end{equation}

In the sequel, we denote by $\mathcal{F}_{j,\geq k}$ (and accordingly, $\mathcal{F}_{j,\geq k}^0$ and $\mathcal{F}_{j,\geq k}^{\textrm{inv}}$)
 the set of functions of $\mathcal{F}_j$ whose Taylor expansion along $\Sigma^+_U \times \{0\}$
starts with terms of degree greater than $k$.

Note that $\rho I\in \mathcal{F}_{2,2}^{\textrm{inv}}$ and that $H\in \mathcal{F}_{2,\geq 2}$.

In what follows, we organize the procedure in two steps:
\begin{itemize}
\item In the first step, we get a normal form
$H\sim \rho I + \sum _{j=2}^{+\infty} \rho_j I^j $ with $\rho_j$ homogeneous of degree $2-j$ on $\Sigma ^+$.
This normal form is valid in any open subset $U$  of $M$ on which $D$ is trivial.
\item In the second step, we remove all terms  $\rho_j I^j$ (the ``invariant part'') in an open set that may be smaller.
\end{itemize}
This choice is due to the specific Poisson bracket properties that are satisfied in those spaces, and which are stated in the following lemma.

All Poisson brackets are now taken with respect to the symplectic form $\tilde{\omega }$ (or its restriction to one of the factors of $\Sigma^+_U \times \R^2$), and in order to keep readability we drop the index $\tilde{\omega }$ in the Poisson brackets.

\begin{lemma}\label{lem_cohom}
For all integers $j$ and $k$, we have
\begin{align}
&\{ \rho I,\mathcal{F}_{j,k}^{\mathrm{inv}}\}\subset \mathcal{F}_{j+1,k+2}^{\mathrm{inv}} , \label{cohom1_inclu}\\
& \{ \rho I,\mathcal{F}_{j,k}^0 \} = \mathcal{F}_{j+1,k}^0 \mod \mathcal{F}_{j+1,k+2}^0, \label{cohom2} \\
&\{ \mathcal{F}_{j,k}, \mathcal{F}_{j',k'}\} \subset \mathcal{F}_{j+j'-1,k+k'-2} \mod \mathcal{F}_{j+j'-1,k+k'} , \label{cohom3} \\
&\{ \mathcal{F}_{j,k}^{\mathrm{inv}}, \mathcal{F}_{j',k'}^{\mathrm{inv}}\} \subset \mathcal{F}_{j+j'-1,k+k'}^{\mathrm{inv}} . \label{cohom4}
\end{align}
Moreover, under the additional assumption that $U$ is Reeb trivial (according to Definition \ref{def:Reebtrivial}), \eqref{cohom1_inclu} becomes an equality, i.e.,
\begin{equation}\label{cohom1}
\{ \rho I,\mathcal{F}_{j,k}^{\mathrm{inv}}\} = \mathcal{F}_{j+1,k+2}^{\mathrm{inv}}.
\end{equation}
\end{lemma}

\begin{proof}
Recall that $\tilde\omega = \omega_{\vert\Sigma^+} + du\wedge dv$, that the coordinates $u$ and $v$ are symplectically conjugate, and that the coordinates $\sigma$ and $(u,v)$ are symplectically orthogonal. It follows that
\begin{equation}\label{fundform}
\{ a(\sigma )P(u,v), b(\sigma )Q(u,v) \} = \{ a, b\} PQ + ab  (\partial_u P\partial_v Q -\partial_v P \partial_u Q ) ,
\end{equation}
for all smooth functions $a$, $b$, $P,~Q$.
The inclusions \eqref{cohom1_inclu}, \eqref{cohom3} and \eqref{cohom4}, easily follow. 
To obtain (\ref{cohom2}), we observe that 
$\{  \rho I, a P  \}= \rho a \partial _\theta P \mod \mathcal{F}_{j+1,k+2}^0$ and that any homogeneous polynomial $P$ of degree $k$ and  of zero mean along circles is of the form $\partial _\theta Q$ with $Q$ of degree $k$ and  of zero mean along circles.

To obtain (\ref{cohom1}), we observe that $\{  \rho I, a I^k \}= \vec{\rho} a I^{k+1}$, so that
we have to solve the differential equation $\vec{\rho} a=b$. This is possible because $U$ is Reeb trivial, meaning that the equation $Z f=g$ admits a smooth solution $f$
for any smooth function $g$.
\end{proof}

\subsubsection{Invariant  normal form}\label{sec:semiglobal}
As explained previously, the objective of the first step is to construct a symplectomorphism $\chi $
reducing $H$ to
\[\tilde{H}= H\circ \chi =\rho I + \sum _{j=2}^{+\infty} \rho_j I^j + \mathrm{O}_\Sigma (\infty ) .\]
This uses the identity \eqref{cohom2}, which is valid globally.
More precisely, we have the following result.

\begin{proposition}\label{prop:invBNF}
 Under the assumptions of Proposition \ref{prop:chi1},
 there exist a conic neighborhood $C$ of $\Sigma ^+_U $ 
and a homogeneous symplectic  diffeomorphism $\chi:C\rightarrow\Sigma ^+_U \times \R^2$ 
such that
$$
g^\star \circ \chi^{-1}= \rho I + \sum _{j=2}^{+\infty} \rho_j I^j + \mathrm{O}_\Sigma (\infty ) ,
$$
with $\rho_j$ homogeneous of degree $2-j$ on $\Sigma ^+_U$.
\end{proposition}

\begin{proof}
For readers acquainted enough with the derivation of Birkhoff normal forms, this result follows from Lemma \ref{lem_cohom} and from the results of Appendix \ref{app:BNF} with
$\mathcal{G}_k = \mathcal{F}_{2,k}$, $\mathcal{S}_k=\mathcal{F}_{1,k}^0$, $\bar{\mathcal{G}}_k =\mathcal{F}_{2,k}^{\textrm{inv}}$, $k_0=3$, $p=2$.

For pedagogical reasons, and for readers who wish to follow the complete argument of proof, we have written a fully detailed proof in Appendix \ref{app:proof:prop:invBNF}.
\end{proof}

\subsubsection{Melrose's local normal form}\label{sec:local}
The objective of the second step is to construct a symplectomorphism allowing us to remove  all terms $\rho_j I^j$.

\begin{proposition}\label{prop:MelroseBNF}
Let $U$ be a Reeb trivial open subset of $M$ (according to Definition \ref{def:Reebtrivial}).
There exist a conic neighborhood $C'$ of  $(\Sigma^+_U \times \R^2,\tilde\omega)$ and a homogeneous symplectomorphism $\psi: C\rightarrow \Sigma^+_U \times \R^2$ such that $\sigma(C')\subset C$, and such that, in $C'$, $\psi $ is the identity
on $\Sigma^+_U \times \{0\}$
and
\begin{equation*}
H\circ \varphi\circ\psi = \tilde H \circ \psi = \rho I  + \mathrm{O}_\Sigma(\infty) .
\end{equation*}
\end{proposition}

\begin{proof}
We use Lemma \ref{lem_cohom} and the results of Appendix \ref{app:BNF} with
$\mathcal{G}_k = \mathcal{F}_{2,k}^{\textrm{inv}}$, $\mathcal{S}_k=\mathcal{F}_{1,k}^{\textrm{inv}}$, $\bar{\mathcal{G}}_k =0$, $k_0=4$, $p=0$.
For the convenience of the reader, a complete proof is given in Appendix \ref{app:proof:prop:MelroseBNF}.
\end{proof}

\subsection{Quantum normal form}\label{sec:QNF}
We are now going to quantize the Birkhoff normal form obtained in Theorem \ref{thm:BNF}.
One possibility could be to use Toeplitz operators associated with the symplectic
cones of the classical normal form (see \cite{BoutetGuillemin}).
But actually, it is simpler to avoid the use
of Toeplitz operators by using the flat Heisenberg manifold for which we have an explicit quantization
described in Section \ref{sec:heis}.

 We obtain the following quantum normal form, where 
we use the notion of a pseudo-differential operator that is flat along $\Sigma$ (see Definition \ref{def_flat} in Appendix \ref{app:flat}).

\begin{theorem}\label{thm_quantum_normal_form}
For every $q_0\in M$, there exists a (conic) microlocal neighborhood $\tilde{U}$ of $\Sigma_{q_0}$ in $T^\star M$ such that,
 considering all the following pseudo-differential operators as acting on functions microlocally supported \footnote{This means that their wave-front set is contained in $\tilde U$.} in $\tilde{U}$,
we have
\begin{equation}\label{quantum_normal_form}
-\triangle_{sR} = R\Omega + V_0+\mathrm{O}_\Sigma(\infty) ,
\end{equation}
where
\begin{itemize}
\item $V_0\in \Psi^0(M)$ is a self-adjoint pseudo-differential operator of order $0$,
\item $R\in \Psi^1(M)$ is a self-adjoint pseudo-differential operator of order $1$, with principal symbol 
\begin{equation}\label{symbR}
\sigma_P(R)=\vert h_Z\vert + \mathrm{O}_\Sigma(2),
\end{equation}
\item $\Omega\in \Psi^1(M)$ is a self-adjoint pseudo-differential operator of order $1$, with principal symbol 
\begin{equation}\label{symbOmega}
\sigma_P(\Omega)=I + \mathrm{O}_\Sigma(4),
\end{equation}

\item $[R,\Omega] = 0 \mod \Psi^{-\infty}(M)$ ,
\item $\exp(2i\pi\Omega)=\mathrm{id} \mod \Psi^{-\infty}(M)$.
\end{itemize}
\end{theorem}

\begin{remark}\label{remflat}
In the flat Heisenberg case, there are no remainder terms in \eqref{quantum_normal_form},
 \eqref{symbR} and \eqref{symbOmega}, and we recover the operators $R_H$ and $\Omega_H$
 defined in Section \ref{sec:heis}. The pseudo-differential operators $R$ and $\Omega$ can be
 seen as appropriate perturbations of $R_H$ and $\Omega_H$, designed such that the last two items
 of Theorem \ref{thm_quantum_normal_form} are satisfied.
\end{remark}

\begin{remark}
We stress that the last two items are valid only if we consider both sides as acting on
 functions that are microlocally supported in $\tilde{U}$. 
If one wants to drop this assumption then all the above operators have to be extended
 (almost arbitrarily) outside $\tilde{U}$, and then the equalities hold only modulo remainder terms in $\mathrm{O}_\Sigma(\infty)$. 

Note that the operators $R$ and $\Omega$ depend on the microlocal neighborhood $\tilde U$ under consideration.
 This neighborhood can then be understood as a chart in the manifold $T^\star M$, in which the quantum normal form is valid.
\end{remark}

In the sequel we will call \textit{normal} any (conic) microlocal neighborhood ${U}$
 in which the conclusions of Theorem \ref{thm_quantum_normal_form} hold true. We also speak of a \textit{normal chart} in $T^\star M$.

\begin{proof}[Proof of Theorem \ref{thm_quantum_normal_form}.]
Applying the Darboux theorem near $q_0$ to the contact form $\alpha _g$ on $M$, and near $0$ to the contact form $\alpha _H$ on $M_H$, 
we identify symplectically $\Sigma^+_g \times \R^2 $ (locally near $q_0$) to  $\Sigma^+_H \times \R^2 $ (locally near $0$).
Then, applying Theorem \ref{thm:BNF} two times  gives an homogeneous canonical transformation $\chi $ from a conical neighborhood of 
$\Sigma ^+_g $ onto a conical neighborhood of $\Sigma ^+_H$ so that
$$
g^\star \circ \chi^{-1}=g_H^\star +\mathrm{O}_\Sigma (\infty) .
$$
Let $U_\chi$ be an unitary Fourier Integral Operator associated with the canonical transformation $\chi$ (see \cite{Du-96,Ho-71}
 and Appendix \ref{app:PDO}).
Setting $-\tilde\triangle_{sR} = - U_\chi^\star \triangle_{H} U_\chi$, we have (generalized Egorov theorem)
$$
\sigma_P(-\tilde\triangle_{sR}) = g^\star_H  \circ \chi= g^\star+ h.
$$ 
where $h$ is a symbol of order $2$ which is $\mathrm{O}_\Sigma (\infty)$.
Actually, since the sub-principal symbol of $-\triangle_{H}$ vanishes, it follows from an argument due to Weinstein
(see Proposition \ref{prop_Weinstein} in Appendix \ref{app:Weinstein})
that we can choose  $U_\chi$ so that the sub-principal symbol of $-\tilde\triangle_{sR}$ vanishes as well.
It follows that we have  $ -\triangle_{sR} =-\tilde\triangle_{sR}+ V_0 - \Op(h)$,
with  $V_0\in \Psi^0(M) $ and $V_0$  self-adjoint.
Setting 
$$
R=U_\chi^\star R_H U_\chi, \qquad \Omega=U_\chi^\star \Omega_H U_\chi ,
$$
we have $\sigma_P(R)=\sigma_P(R_H)\circ\chi$ and
 $\sigma_P(\Omega)=\sigma_P(\Omega_H)\circ\chi$. The relations  \eqref{symbR} and \eqref{symbOmega} follow from the properties 
of the classical normal forms and the fact that $\rho =|h_Z| + \mathrm{O}_\Sigma (2)$ (both functions coincide on $\Sigma $ as well
as their Hamiltonian flows). 
The rest follows from the corresponding relations in the Heisenberg flat case.
\end{proof}

\begin{remark}\label{rem_symbolsROmega}
It also follows from the proof that 
\begin{equation}\label{symboleDelta}
g^\star = \sigma_P(-\triangle_{sR}) = \sigma_P(\Omega) \sigma_P(R) -h,
\end{equation}
where $h\in\mathcal{S}^2(M)$ with $h=\mathrm{O}_\Sigma (\infty)$.
\end{remark}



\section{Variance estimate and proof of Theorem \ref{thm1}}\label{sec_variance}
In this section, we are going to establish the following result (from which Theorem \ref{thm1} follows).
Let us choose some homogeneous symbols $\pi_+$ and $\pi_-$ of order $0$ such that
$\pi_+ $ (resp., $\pi_-$) vanishes near $\Sigma _-$ (resp., near $\Sigma _{+}$) and such that
$\pi_-=\pi_+ \circ \sigma $, where $\sigma $ is the canonical involution on $T^\star M$.

\begin{proposition}\label{prop:pretheo} 
We assume that the Reeb flow is ergodic on $(M,\nu)$.
Let $\Pi_\pm$ be pseudo-differential operators of order $0$ whose principal symbols are $\pi_\pm $.
For every pseudo-differential operator $A\in\Psi^0(M)$ whose principal symbol vanishes on $\Sigma ^-$, we have $V(A-\hat A_+)=0$, where 
\begin{equation}\label{defhataplus}
\hat A_+=\hat a_+\Pi_+, \qquad \hat a_+=\int_M a(q,\alpha_g (q)) \, d\nu.
\end{equation}
Similarly, for every $A\in\Psi^0(M)$ whose principal symbol vanishes on $\Sigma ^+$, we have $V(A-\hat A_-)=0$, with 
$\hat A_-=\hat a_-\Pi_-$.
\end{proposition}
\medskip

Admitting temporarily Proposition \ref{prop:pretheo}, let us prove Theorem \ref{thm1}.
As explained in the introduction, in order to establish the QE property, it suffices to prove that, for every
 pseudo-differential operator $A\in\Psi^0(M)$, if either the eigenfunctions $\phi_n,~n\in \N, $ are real-valued
or $a$ is even, then
\begin{equation*}
V(A-\bar{a}\,\mathrm{id})=0,
\end{equation*}
where we have set 
$$
\bar{a}=\frac{1}{2}\int_M  \big( a(q,\alpha_g (q))+ a(q,-\alpha_g (q)) \big) \, d\nu.
$$
In order to prove that fact, we write $A=A_+ + A_-$,
 with the principal symbol of $A_+$ (resp., of $A_-$) vanishing on $\Sigma^-$ (resp., on $\Sigma^+$). Using the above results, we have $V(A_+ -\hat a_+\Pi_+)=0$ and $V(A_- -\hat a_-\Pi_-)=0$. Since 
$$
V(A-\hat A_+ -\hat A_-) \leq 2\left( V(A_+ -\hat A_+) + V(A_- -\hat A_-) \right) ,
$$
we infer that $V(A-\hat a_+ \Pi_+ -\hat a_- \Pi_-)=0$. Besides, noting that $\pi_++  \pi_- = 
1 \,+O_{\Sigma}(1)$ and that $\bar a = \frac{1}{2} (\hat a_++\hat a_-)$, we have
\begin{equation*}
 \hat a_+ \pi_+ +\hat a_- \pi_- 
 = \frac{\hat a_+  + \hat a_-}{2}(\pi_++  \pi_-)+  \frac{\hat a_+-\hat a_-}{2} (\pi_+ -  \pi_-) 
= \bar{a}  + \frac{\hat a_+-\hat a_-}{2}(\pi_+ -  \pi_-) + O_{\Sigma}(1),
\end{equation*}
Indeed:
\begin{itemize}
\item If the eigenfunctions are real-valued, using  \eqref{formula_conjugation} and the
fact that   $\pi_+ -  \pi_-$ is odd,  it  follows that $\langle (\Pi_+ -  \Pi_-)\phi_n , \phi_n \rangle \rightarrow 0  $ as $n\rightarrow +\infty$,  and hence 
 $V(\Pi_+ -  \Pi_-)=0$.
Using Lemma \ref{remCS}, we conclude that $V(A-\bar{a}\,\mathrm{id})=0$.
\item If $a$ is even, then ${\hat a_+-\hat a_-}=0$.
\end{itemize}
Theorem \ref{thm1} is proved.

\medskip

The proof of Proposition \ref{prop:pretheo} is done in Section \ref{sec_proof_prop:pretheo}. We  first establish in Sections \ref{sec:aver} and \ref{main:lemma} two useful preliminary lemmas.


\subsection{Averaging in a normal chart}\label{sec:aver}
Given $A\in \Psi^0(M)$, according to Corollary \ref{cor_weyl}, $V(A)$ depends only on the restriction $a_{\vert\Sigma}$ (where $a=\sigma_p(A)\in\mathcal{S}^0(M)$), in the sense that, if the principal symbols of two pseudo-differential operators $A_1$ and $A_2$  of order $0$ agree on $\Sigma$, then $V(A_1-A_2)=0$. This property gives us the possibility to modify $A$ without changing $a_{\vert\Sigma}$, and we can use this latitude to impose the additional condition $[A,\Omega]=0 \mod \Psi^{-\infty}(M)$.

\begin{lemma}\label{lem:aver}
Let $A\in \Psi^0(M)$ be microlocally supported in a normal chart $U$ and let $\Omega$ be given by Theorem \ref{thm_quantum_normal_form} (in the microlocal neighborhood $U$). Assuming $U$ small enough, the operator defined by
$$
B= \frac{1}{2\pi} \int_0^{2\pi} \exp(is\Omega)A\exp(-is\Omega)\, ds,
$$
is in $\Psi^0(M)$, is microlocally supported in $U$, and satisfies
\begin{equation*}
\begin{split}
\sigma_P(B)&=\sigma_p(A)+\mathrm{O}_\Sigma(1), \\
[B,\Omega]&= 0 \mod \Psi^{-\infty}(M).
\end{split}
\end{equation*}
\end{lemma}

\begin{proof}
The proof follows an argument introduced by Weinstein in \cite{We-77} (see also \cite{yCdV-79}). 
For every $s\in[0,2\pi]$, we set $B_s = \exp(is\Omega) A\exp(-is\Omega)$. 
By the Egorov theorem, we have $B_s\in\Psi^0(M)$ and $\sigma_P(B_s) = a\circ {\exp(s \vec{w})}$, where $w=\sigma_P(\Omega)$ and ${\exp(s \vec{w})}$ is the flow generated by the Hamiltonian vector field $\vec{w}$ associated with the Hamiltonian function $w$.
Here, the microlocal neighborhood $U$ in which this construction is performed must be chosen small enough, so that it is invariant under the flow ${\exp(s \vec{w})}$, for $s\in[0,2\pi]$. This is possible because, using \eqref{symbOmega}, we have $w=\mathrm{O}_\Sigma(2)$, and therefore ${{\exp(s \vec{w})}}_{\vert\Sigma}=\mathrm{id}$. Moreover, we infer that $\sigma_P(B_s)_{\vert\Sigma}=a_{\vert\Sigma}$.

Setting $B = \frac{1}{2\pi}\int_0^{2\pi} B_s\, ds\in\Psi^0(M)$, we have $\sigma_P(B)_{\vert\Sigma} = a_{\vert\Sigma}$.
By Theorem \ref{thm_quantum_normal_form}, we have $\exp(2i\pi\Omega)=\mathrm{id} \mod \Psi^{-\infty}(M)$, and thus $B_{2\pi}=B_0 \mod \Psi^{-\infty}(M)$.
Now, since $\frac{d}{ds}B_s = i[\Omega,B_s]$, integrating over $[0,2\pi]$ yields $[B,\Omega]=0\mod \Psi^{-\infty}(M)$.
\end{proof}


\subsection{The main lemma playing the role of an infinitesimal Egorov theorem} \label{main:lemma}

Recall that $(\Sigma, \omega_{\vert\Sigma})$ is a symplectic manifold. We denote by $\{\ ,\ \}_{\omega_{\vert\Sigma}}$ the corresponding Poisson bracket on that manifold.
Hereafter, we use the Hamiltonian function $\rho={h_Z}_{\vert\Sigma}$ on $\Sigma$, as defined in Section \ref{sec_reeb}.

The following lemma may be seen as a substitute for the invariance properties (infinitesimal Egorov theorem) with respect to the geodesic flow, that are used in the proof of the classical Shnirelman theorem.

\begin{lemma}\label{lemcrochetnul} 
Let $Q\in\Psi^0(M)$ be such that $\sigma_P(Q)_{\vert\Sigma} = \{ a_{\vert\Sigma} , \rho \}_{\omega_{\vert\Sigma}}$ for some $a\in \mathcal{S}^0(M)$. Then $V(Q)=0$.
\end{lemma}

\begin{proof}
By using a partition of unity, without loss of generality, we assume that the support of $a$ is contained in a normal chart $U$ near $\Sigma _+$. Setting $A=\Op(a)$, using \eqref{symbR} and Lemma \ref{lemmaReeb}, we have
\begin{multline*}
\sigma_P([A,R])_{\vert\Sigma} = \frac{1}{i} \left( \{ a,  h_Z+\mathrm{O}_\Sigma(2) \}_\omega \right)_{\vert\Sigma}
= \frac{1}{i} \left(\{ a, h_Z \}_\omega\right)_{\vert\Sigma}
= \frac{1}{i} \left( da . \vec h_Z \right)_{\vert\Sigma} \\
= \frac{1}{i}  da_{\vert\Sigma} . \vec \rho 
= \frac{1}{i} \{ a_{\vert\Sigma} , \rho \}_{\omega_{\vert\Sigma}}
= \frac{1}{i} \sigma_P(Q)_{\vert\Sigma} ,
\end{multline*}
and therefore $Q=i[A,R]+C_0 \mod  \Psi^{-1}$ with $\sigma _p (C_0)=0$ on $\Sigma $.

By Lemma \ref{lem:aver} that there exists $B\in\Psi^0(M)$, microlocally supported in $U$, such that
 $\sigma_P(B)_{\vert\Sigma}=\sigma_P(A)_{\vert\Sigma}$ and $[B,\Omega]= 0 \mod \Psi^{-\infty}(M)$.
Therefore $Q=i[B,R]+C_1 \mod \Psi^{-1}$ with $\sigma _p (C_1)=0$ on $\Sigma $.

By Corollary \ref{cor_weyl}, $V(Q)$ depends only of $\sigma_P(Q)_{\vert\Sigma}$, and moreover, since any pseudo-differential operator of negative order is compact, we get, by Lemmas \ref{remCS} and \ref{lemm:compact}, that $V(Q)=V([B,R])$.
As a consequence, in order to prove the lemma, it suffices to prove that $V([B,R])=0$.

To this aim, let us estimate each term $\left\langle [B,R]\phi_n,\phi_n\right\rangle$, for every $n\in\N^*$. We have
\begin{equation*}
\begin{split}
\left\langle [B,R]\phi_n,\phi_n\right\rangle
&= \left\langle BR\phi_n,\phi_n\right\rangle - \left\langle RB\phi_n,\phi_n\right\rangle \\
&= \frac{1}{\lambda_n}\left\langle BR\phi_n,-\triangle_{sR}\phi_n\right\rangle - 
\frac{1}{\lambda_n}\left\langle
 RB(-\triangle_{sR})\phi_n,\phi_n\right\rangle ,
\end{split}
\end{equation*}
because $-\triangle_{sR}\phi_n=\lambda_n\phi_n$.
 Now, using \eqref{quantum_normal_form}, we have $-\triangle_{sR}=R\Omega+V_0+ C$, with $V_0\in\Psi^0(M)$ and $C=\mathrm{O}_\Sigma(\infty)$. 
Note that $C\in\Psi^2(M)$ and that $C$ is self-adjoint (but we cannot say that $C\in\Psi^1(M)$ because of the remainder term $h$ in \eqref{symboleDelta}).
It follows that
$$
\left\langle [B,R]\phi_n,\phi_n\right\rangle
= I_n+J_n+K_n ,
$$
with
\begin{equation*}
\begin{split}
I_n &= \frac{1}{\lambda_n}\left\langle BR\phi_n,R\Omega\phi_n\right\rangle - \frac{1}{\lambda_n}\left\langle RBR\Omega\phi_n,\phi_n\right\rangle ,\\
J_n &= \frac{1}{\lambda_n}\left\langle BR\phi_n,V_0\phi_n\right\rangle - \frac{1}{\lambda_n}\left\langle RBV_0\phi_n,\phi_n\right\rangle , \\
K_n &= \frac{1}{\lambda_n}\left\langle BR\phi_n,C\phi_n\right\rangle - \frac{1}{\lambda_n}\left\langle RBC\phi_n,\phi_n\right\rangle .
\end{split}
\end{equation*}
Since $R$ and $\Omega$ are self-adjoint, we have
\begin{equation*}
I_n = \frac{1}{\lambda_n}\left\langle \Omega RBR\phi_n,\phi_n\right\rangle -
 \frac{1}{\lambda_n}\left\langle RBR\Omega\phi_n,\phi_n\right\rangle 
= \frac{1}{\lambda_n}\left\langle [ \Omega, RBR]\phi_n,\phi_n\right\rangle .
\end{equation*}
Since $[R,\Omega]= 0 \mod\Psi^{-\infty}(M)$ and $[B,\Omega]= 0 \mod\Psi^{-\infty}(M)$, we infer that $[ \Omega, RBR] = 0 \mod\Psi^{-\infty}(M) $, and hence $I_n = \mathrm{o}(1)$ as $n\rightarrow+\infty$.

Let us now focus on the second term. Since $V_0$ is self-adjoint, this term can be written as
$$
J_n = \frac{1}{\lambda_n} \left\langle V_0[B,R]\phi_n,\phi_n\right\rangle +
 \frac{1}{\lambda_n} \left\langle [V_0,RB]\phi_n,\phi_n\right\rangle .
$$
The two pseudo-differential operators $V_0[B,R]$ and $[V_0,RB]$ are of order $0$ and therefore are bounded. Since $\lambda_n\rightarrow +\infty$ as $n\rightarrow+\infty$, it follows that $J_n = \mathrm{o}(1)$ as $n\rightarrow+\infty$.

Finally, the third term can be written as
$$
K_n = \frac{1}{\lambda_n}\left\langle (C[B,R] + [C, R B] )\phi_n,\phi_n\right\rangle = 
 \frac{1}{\lambda_n}\left\langle D\phi_n,\phi_n\right\rangle,
$$
with $D = C [B,R] + [C, R B]$. Clearly, we have $D\in\Psi^2(M)$ and $D=\mathrm{O}_\Sigma(\infty)$.
Therefore, by Lemma \ref{lemm:facto} (see Appendix \ref{app:flat}), there exists $D_1\in\Psi^0(M)$, with $D_1=\mathrm{O}_\Sigma(\infty)$, such that $D = -D_1\triangle_{sR}$ modulo a smoothing operator.
It follows from this factorization that
$$
K_n = \left\langle D_1\phi_n,\phi_n\right\rangle + \mathrm{o}(1) ,
$$
as $n\rightarrow+\infty$, with $\sigma_P(D_1)_{\vert\Sigma}=0$.

We conclude from the study of these three terms that
$$
\left\langle [B,R]\phi_n,\phi_n\right\rangle
= \left\langle D_1\phi_n,\phi_n\right\rangle + \mathrm{o}(1),
$$
as $n\rightarrow+\infty$, and hence $V( [ B , R ] ) = V(D_1)$.
Since $D_1\in\Psi^0(M)$ has a principal symbol vanishing along $\Sigma$, it follows from Corollary \ref{cor_weyl} that $V(D_1)=0$. The conclusion follows.
\end{proof}


\subsection{Proof of Proposition \ref{prop:pretheo}}\label{sec_proof_prop:pretheo}
Let $A\in\Psi^0(M)$ whose principal symbol $a$ vanishes on $\Sigma ^-$. The objective is to prove that $V(A-\hat A_+)=0$. 

Let $(a_t)_{t\in\R}$ be a family of elements of $\mathcal{S}^0(M)$, depending smoothly on $t$, such that 
$$
{a_t}_{\vert\Sigma} = a_{\vert\Sigma}\circ \exp(t\vec\rho).
$$ 
We set $A_t=\Op(a_t)$, for every $t\in\R$, and we set $\bar A_T=\frac{1}{T}\int_0^T A_t\, dt$, for every $T>0$.
The principal symbol $\bar a_T\in\mathcal{S}^0(M)$ of $\bar A_T$ is $\bar a_T = \frac{1}{T}\int_0^T a_t\, dt$.

In order prove that $V(A-\hat A_+)=0$, we proceed in two steps:
\begin{enumerate}
\item Prove that $V(A-A_t)=0$ for every time $t$, and hence that $V(A-\bar{A}_T )=0$ (this step does not require any ergodicity assumption);
\item Using the ergodicity of the Reeb flow and the Von Neumann mean ergodic theorem, prove that $\displaystyle\lim_{T\rightarrow +\infty} V(\bar{A}_T -\hat A_+)=0$. 
\end{enumerate}

\paragraph{First step: $V(A-\bar A_T)=0$.} 

\begin{lemma}\label{lem14}
For every $t\in\R$, we have $V\left(\frac{d}{dt}A_t \right)=0$.
\end{lemma}

\begin{proof}
By definition of $A_t$, we have $\sigma_P\left( \frac{d}{dt} A_t \right)_{\vert\Sigma} = \{a_{\vert\Sigma} \circ\exp(t\vec\rho),\rho\}_{\omega_{\vert\Sigma}}$, and then the result follows from Lemma \ref{lemcrochetnul}.
\end{proof}

As a corollary, we have the following proposition.

\begin{proposition}\label{prop_eqAAt}
We have $V(A-A_t) = 0$ for every time $t$, and $V(A-\bar A_T)=0$, for every $T>0$.
\end{proposition}

\begin{proof} 
We start from 
$$
\langle (A-A_t)\phi_n ,\phi_n \rangle =-  \int_0^t \left\langle \frac{dA_s}{ds} \phi_n ,\phi_n \right\rangle \, ds ,
$$
for every time $t$, and hence, by the Cauchy-Schwarz inequality,
$$
\langle (A-A_t)\phi_n ,\phi_n \rangle ^2 \leq  t\int_0^t \left\vert \left\langle   \frac{dA_s}{ds} \phi_n ,\phi_n \right\rangle \right\vert^2 ds .
$$
Summing with respect to $n$, 
we get that $V(A-A_t)\leq t\int_0^t V\left(\frac{d}{ds}A_s\right)\, ds$. By Lemma \ref{lem14}, we infer that $V(A-A_t)=0$, for every $t\in\R$.
  
Let us now prove that $V(A-\bar A_T)=0$, with $\bar A_T=\frac{1}{T}\int_0^T A_t\, dt$. We have, by the Fubini theorem,
$$
\left\langle (A-\bar A_T)\phi_n ,\phi_n \right\rangle
= \frac{1}{T} \int_0^T \left\langle (A-A_t)\phi_n , \phi_n \right\rangle dt ,
$$
for every integer $n$. Using again the Jensen inequality, and summing with respect to $n$, we get that $V(A-\bar A_T) \leq \frac{1}{T} \int_0^T V(A-A_t) \, dt$.
It follows that $V(A-\bar A_T) = 0$ .
\end{proof}

\paragraph{Second step: $V(\bar A_T-\hat A_+)\rightarrow 0$ as $T\rightarrow +\infty$.}
Here, we are going to use the ergodicity assumption on the Reeb flow. 

Using \eqref{ineqV}, we have $V(\bar A_T-\hat A_+)\leq E((\bar A_T-\hat A_+)^\star (\bar A_T -\hat A_+))$,
 and it follows from Theorem \ref{theo:weyl} (microlocal Weyl law) that
$$
E((\bar A_T-\hat A_+)^\star (\bar A_T -\hat A_+))= \frac{1}{2} \int_{\Sigma _1} |\bar a_T -\hat a_+|^2 \, d\hat\nu_1 ,
$$
with $(\bar a_T)_{\vert\Sigma}=\frac{1}{T}\int_0^T a_{\vert\Sigma}\circ \exp(t\vec \rho )\, dt$ and $\hat a_+$ defined by \eqref{defhataplus}.
Since the flow $\exp(t\vec \rho )$ (which is the lift to $\Sigma _1$ of the flow of $Z$, by Lemma \ref{lemmaReeb}) is ergodic on $(\Sigma_1,\hat\nu_1)$, it follows from the Von Neumann mean ergodic theorem (see, e.g., \cite{Petersen}) that $\bar a_T$ converges to $\hat a_+$ in $L^2(\Sigma_1,\hat\nu_1)$ as $T\rightarrow +\infty$.
Therefore $V(\bar A_T-\hat A_+)$  converges to $0$ as $T\rightarrow+\infty$.

\medskip

Using the inequality $V(A-\hat A_+)\leq 2 ( V(A-\bar A_T) + V(\bar A_T-\hat A_+) )$, and the results of the two steps above, we conclude the proof of Proposition \ref{prop:pretheo}.


\section{Proof of Theorem \ref{thm2}}\label{sec_proof_thm2}

Let us prove the first part of Theorem \ref{thm2}.
We set $\mu_n(a) = \left\langle \Op(a)\phi_n, \phi_n\right\rangle$, for every $n\in\N^*$ and every $a\in\mathcal{S}^0(M)$. 

The cosphere bundle $S^\star M$ is identified, by taking intersections with half lines, with the compactification $\widehat{U^\star M}$
 of the unit cotangent bundle $U^\star M=\{ g^\star =1 \}$, as follows: we add to each cylinder $U^\star_q M$, homeomorphic
 to $\mathbb{S}^2 \times \R$, its two extremities that we identify with $S\Sigma _q \sim (\Sigma _1 \cup \Sigma _{-1})_q$,
 obtaining
$$
\widehat{U^\star M} = U^\star M \cup \Sigma _1 \cup \Sigma _{-1} .
$$
Let $\beta$ be a QL. By definition, $\beta$ is a probability measure on
 $\widehat{U^\star M}$, and there exists a sequence of integers $(n_j)_{j\in\N^*}$ such that $\mu_{n_j}$ converges weakly to $\beta$. 
The measure $\beta $ is then decomposed in a unique way as the sum $\beta = \beta_0 + \beta_\infty$,
 with $\beta_0$ supported on $U^\star M$ and $\beta_\infty$ supported on $S\Sigma\sim \Sigma_1 \cup \Sigma _{-1}$:
 $\beta_0$ is the restriction of $\beta$ to compactly supported functions on ${U^\star M}$.

Let us first prove that $\beta_0$ is invariant under the sR geodesic flow.
Let $A\in\Psi^0(M)$, with principal symbol $a$, be microlocally supported away of $\Sigma.$ 
Since $\phi_{n_j}$ is an eigenfunction of $-\triangle_{sR}$ associated with the real eigenvalue $\lambda_{n_j}$, we have
\begin{equation}\label{annulnj}
\left\langle \left[\sqrt{-\triangle_{sR}},A\right]  \phi_{n_j} , \phi_{n_j} \right\rangle =0.
\end{equation}
On the microlocal support of $A$, $\sqrt{-\triangle_{sR}}$ is a pseudo-differential operator of order $1$ with principal
 symbol $\sqrt{g^\star}$ (see Remark \ref{rem:sqrt} and \cite{HV-00}). It follows that $\left[\sqrt{-\triangle_{sR}},A\right]$
 is a pseudo-differential operator of order $0$ with principal symbol $i \{ \sqrt{g^\star },a \}_\omega$.
Passing to the limit in \eqref{annulnj}, using the definition of $\beta_0$ and the fact that $a$ is supported away of $\Sigma$,
 we get $\int_{U^\star M} \{ \sqrt{g^\star} ,a \}_\omega \, d\beta_0 =0$.
Since $\{ \sqrt{g^\star} ,a \}_\omega = \vec{G}.a $ where $\vec{G}$ is the geodesic flow (which coincides with the geodesic flow generated by $g^\star$, on $\{g^\star=1\}$), the invariance of $\beta_0$ under the sR geodesic flow is inferred from the following general lemma.

\begin{lemma}\label{l:inv}
Let $N$ be a manifold, equipped with a measure $\delta$, and let $X$ be a complete vector field on $N$.
 If $\int _N (X .\phi) \, d\delta =0$ for every $\phi \in C_0^\infty (N,\R)$, then the measure $\delta$ is invariant under the flow of $X$. 
\end{lemma}

Let us now prove that $\beta_\infty$ is invariant under the lift $\exp(t\vec\rho)$ to $S\Sigma$ of the Reeb flow
 (defined in Section \ref{sec_reeb}). Using a partition of unity, we can work
in ${U}\subset \Sigma_\sigma \times \R^2_{u,v}$.
  Let $b_0$ be an arbitrary smooth function on the manifold $\Sigma_1$ with support in ${U}\cap \Sigma_1$.
By homogeneity, using Lemma \ref{lem:aver}, there exists a pseudo-differential operator $B\in\Psi^0(M)$, microlocally supported in
 ${U}$, of principal symbol $b$, such that $b_{\vert\Sigma}=b_0$ and $[B,\Omega]=0\mod\Psi^{-\infty}(M)$.
Using the estimates obtained in Section \ref{main:lemma}, and in particular in the proof of Lemma \ref{lemcrochetnul}, we have
$$
\langle [B,R]  \phi_{n_j} , \phi_{n_j} \rangle -  \langle  D_1 \phi_{n_j} ,
 \phi_{n_j} \rangle \underset{j\rightarrow +\infty}{\longrightarrow} 0 ,
$$
for some $D_1\in\Psi^0(M)$ that is flat along $\Sigma$, satisfying
\begin{equation}\label{symbD1}
i\sigma_P(D_1) \sigma_P(-\triangle_{sR}) = \sigma_P(C) \{ b, \sigma_P(R) \}_\omega + \{ \sigma_P(C), b\sigma_P(R) \}_\omega,
\end{equation}
with $C\in\Psi^2(M)$ such that $\sigma_P(C)=\mathrm{O}_\Sigma(\infty)$.
Passing to the limit, we obtain 
\begin{equation}\label{eqbetalim}
\beta \left( \{b,\sigma_P(R)\}_\omega + i \sigma_P(D_1) \right) = 0.
\end{equation}

By \eqref{symbR}, we have $\sigma_P(R)=h_Z+\mathrm{O}_\Sigma(2)$ along $\Sigma^+$, and hence, reasoning similarly as at the beginning of the proof of Lemma \ref{lemcrochetnul}, we have
\begin{equation}\label{crochetSigmaplus}
(\{ b, \sigma_P(R) \}_\omega)_{\vert\Sigma^+} = \{ b_{\vert\Sigma}, \rho  \}_{\omega_{\vert\Sigma}} = \{ b_0, \rho  \}_{\omega_{\vert\Sigma}} .
\end{equation}
Since $\sigma_P(D_1)_{\vert\Sigma}=0$, we have $\beta_\infty(\sigma_P(D_1))=0$, and since $\beta=\beta_0+\beta_\infty$, we infer from \eqref{eqbetalim}, \eqref{symbD1} and \eqref{crochetSigmaplus} that
\begin{equation*}
\beta_\infty\left( \{ b_0, \rho  \}_{\omega_{\vert\Sigma}} \right) + \beta_0\left( L(b)  \right) = 0 ,
\end{equation*}
where the operator
$$
L(b) = \left( 1+\frac{\sigma_P(C)}{\sigma_P(-\triangle_{sR})}\right) \{b,\sigma_P(R)\}_\omega + \frac{1}{\sigma_P(-\triangle_{sR})}  \{ \sigma_P(C), b\sigma_P(R) \}_\omega 
$$ 
is defined on $U^\star M$. 
We have thus proved that
\begin{equation}\label{eqlimmes}
\int _{\Sigma _1 } \{ b_{\vert\Sigma}, \rho  \}_{\omega_{\vert\Sigma}} \, d\beta_\infty + \int _{U^\star M} L(b) \, d\beta_0 = 0 ,
\end{equation}
for any symbol $b\in\mathcal{S}^0(M)$ which is microlocally supported in $\tilde{U}$ and such that $\{b,\sigma_P(\Omega)\}_\omega=0$.
Let us now prove the following lemma.

\begin{lemma}\label{lem_Lb}
We have $\int _{U^\star M} L(b) \, d\beta_0 = 0$, for every symbol $b\in\mathcal{S}^0(M)$ which is microlocally supported in $\tilde{U}$ and 
such that $\{b,\sigma_P(\Omega)\}_\omega=0$.
\end{lemma}

\begin{proof}[Proof of Lemma \ref{lem_Lb}.] 
Recall that $\sigma_P(R) = \rho$ and that $\sigma_P(\Omega) = I$. 

The condition $\{b,\sigma_P(\Omega)\}=0$ gives $\{b,I \}=0$.
Now, we choose an arbitrary smooth function $f$ satisfying $f(0)=0$ and $f(t)=1$ for $t\geq 1$, and for every integer $k$ we define $b^k$ by 
$$
b^k = b f(kI).$$
We obtain a sequence $(b^k)_{k\in\N^*}$ of symbols such that $\{b^k,I \}=0$, 
for every $k\in\N^*$ (because $\{f(kI), I\}=0$). 
By definition of $f$, we have $b^k_{\vert\Sigma}=0$ and thus 
$\int _{\Sigma _1 } \{ b^k_{\vert\Sigma}, \rho  \}_{\omega_{\vert\Sigma}} \, d\beta_\infty=0$. 
Then, we infer from \eqref{eqlimmes} that 
\begin{equation}\label{eqLk}
\int _{U^\star M} L(b^k) \, d\beta_0 = 0,
\end{equation}
for every $k\in\N^*$.

Let us now prove that $L(b^k)$ converges pointwise to $L(b)$ in $U^\star M$ and is uniformly bounded.
Since the coordinates $\sigma=(q,s)$ and $(u,v)$ are symplectically orthogonal, we have
$$
\{ b^k,\rho \}_\omega
= \{ b f(kI ) , \rho  \}
= f(kI) \{ b , \rho  \},
$$
and hence $\{ b^k,\rho \}_\omega$ converges pointwise to $\{ b,\rho \}$, and is uniformly bounded. Besides, we have
\begin{equation*}
\begin{split}
\{ \sigma_P(C), b^k\rho  \} & =\, \{ \sigma_P(C) , f(kI) \rho \, b \}_{\tilde\omega} \\
&=\,  f(kI) \{ \sigma_P(C) , \rho \, b \} +  \{ \sigma_P(C) , f(kI) \} \, \rho \, b \\
&=\,  f(kI) \{ \sigma_P(C) , \rho \, b \} + kf'(kI) \{ \sigma_P(C) , I \} \, \rho \, b .
\end{split}
\end{equation*}
In the latter line, the first term raises no problem and converges pointwise with dominated convergence. 
The second term needs more care. Since $C$ is flat on $\Sigma$, we have in particular 
that $\{ \sigma_P(C) , I \} \leq cI^3$ for some constant $c>0$. 
By definition of $f$, if $I\geq 1/k$ then
$f'(kI)=0$, and if $I\leq 1/k$ then
$$
kf'(kI) \{ \sigma_P(C) , I \}
\leq c \Vert f'\Vert_\infty k I^3 \leq \frac{c \Vert f'\Vert_\infty I}{k} ,
$$
and therefore, the second term converges pointwise to $0$ with a dominated convergence.
We conclude that $L(b^k)$ converges pointwise to $L(b)$ with a dominated ($L^1$) convergence.

Using \eqref{eqLk}, the lemma follows by applying the Lebesgue dominated convergence theorem.
\end{proof}

Using \eqref{eqlimmes}, it follows from Lemma \ref{lem_Lb} that $\beta_\infty\left( \{ b_0, \rho  \}_{\omega_{\vert\Sigma}} \right) = 0$
 for every classical symbol $b_0$ of order $0$ on $\Sigma_1$.
Reasoning similarly on $\Sigma_{-1}$, we have thus proved that $\int _{S\Sigma}\{ b_0, \rho \}_{\omega_{\vert\Sigma}} \, d\beta_\infty =0$
 for any classical symbol $b_0$ of order $0$ on the manifold $S\Sigma$. The invariance property follows from Lemma \ref{l:inv} as before.

\medskip

Let us now prove the second part of Theorem \ref{thm2}.

Let $\mathcal{D}^0_\Sigma$ be a countable dense subset of the set of $a\in\mathcal{S}^0(M)$ such that $a_{\vert\Sigma}=0$. By Corollary \ref{cor_weyl}, we have $V(\Op(a))=0$, for every $a\in \mathcal{D}^0_\Sigma$. Using the lemma of Koopman and Von Neumann (already mentioned in Section \ref{sec1}), and using a diagonal argument, we infer that there exists a sequence $(n_j)_{j\in\N^*}$ of integers of density one such that $\mu_{n_j}(a)\rightarrow 0$ as $j\rightarrow +\infty$, for every $a\in \mathcal{D}^0_\Sigma$.
It follows that, for every quantum limit $\beta$ associated with the family $(\phi_{n_j})_{j\in \N^*}$, we have $\beta(a)=0$, for every $a\in \mathcal{D}^0_\Sigma$. By density, we infer that $\beta(a)=0$, for every $a\in \mathcal{S}^0(M)$ such that $a_{\vert\Sigma}=0$. We have thus proved that the support of $\beta$ is contained in $S\Sigma$.


\section{Complex-valued eigenbasis, and the non-orientable case} \label{sec:non-or}

\paragraph{Complex-valued eigenbasis, with $D$ oriented.}
We can extend Theorem \ref{thm1} to the case of a complex-valued eigenbasis, to the price of requiring that the principal symbol $a$ of $A$ satisfies the evenness condition
\begin{equation}\label{apair}
a(q,\alpha_g(q)) = a(q,-\alpha_g(q)),
\end{equation}
for every $q\in M$. The proof is the same.

\paragraph{$D$ not orientable.}
Let us assume that the subbundle $D$ is not orientable. Then there exists a double covering $\tilde{M}$ of $M$ with an involution $J$, so that we can lift all data to $\tilde{M}$, and then the subbundle $\tilde{D}$ of $T\tilde{M}$ is orientable.
The Reeb vector field $\tilde Z$ on $\tilde{M}$ is odd with respect to the involution.

\begin{definition}
The Reeb dynamics are ergodic if every measurable subset of $\tilde{M}$ which is invariant under $\tilde{Z}$ and invariant under $J$ is of measure $0$ or $1$.
\end{definition}

\begin{theorem}\label{thm8.1}
We assume that the Reeb dynamics are ergodic. Then we have $QE$ for any eigenbasis of $\triangle_{sR}$.
\end{theorem}

The proof is an adaptation of the orientable case. Note that ${\Sigma }\setminus \{0\}$ is connected.
We can remove the assumption that the eigenfunctions are real-valued: indeed, any eigenfunction on $M$, real-valued or complex-valued, is lifted to $\tilde M$ to an even function.
Moreover, denoting by $V_M$ (resp., by $V_{\tilde M}$) the variance on $M$ (resp., on $\tilde M$), we have $V_M(A) = V_{\tilde M}(\tilde A)$, with $\tilde A$ even with respect to the involution $J$. In particular, the principal symbol of $\tilde A$ satisfies \eqref{apair} along $\Sigma$.



\appendix


\section{Appendix: some tools of microlocal analysis}

\subsection{Pseudo-differential operators and Fourier integral operators}\label{app:PDO}
In this section, we recall some definitions and facts used in the paper concerning pseudo-differential operators (PDOs)
and Fourier integral operators (FIOs). Proofs can be found in the original papers \cite{Ho-71} and \cite{DH-72}, and in the more geometrical book \cite{Du-96} (see also \cite{Zworski}).

\subsubsection{PDOs}
In what follows, $M$ is a smooth compact manifold of dimension $d$ equipped with a smooth non-vanishing density $\mu$.
The algebra $\Psi (M)$ of classical pseudo-differential operators on $M$ is graded according to the chain of inclusions $\Psi^{-\infty} (M)\cdots \subset \Psi^m (M) \subset \Psi^{m+1}(M) \subset \cdots$, where $m$ is called the order.
What is important and useful is the notion of principal symbol $\sigma _p$  and of sub-principal symbol $\sigma_{\textrm{sub}}$ of a PDO $A\in \Psi^m(M)$.
There is a bijective map
$$
(\sigma_p, \sigma_{\textrm{sub}}):\Psi^m (M)/\Psi^{m-2}(M) \longrightarrow \mathcal{S}^m (M) \oplus \mathcal{S}^{m-1}(M) ,
$$
where $\mathcal{S}^m(M) $ is the space of smooth homogeneous functions of order $m$ defined on the cone $T^\star M \setminus \{0\}$.
A quantization is a continuous linear mapping
$$
\Op:\mathcal{S}^0(M) \rightarrow \Psi^0(M)
$$
satisfying $\sigma _p (\Op(a))= a$. 
An example of quantization is obtained by using partitions of unity and the so-called Weyl quantization given in local coordinates by:
$$
\OpW(a) f (q) = (2\pi)^{-d}\int_{\R^d_{q'}\times \R^d_p} e^{i \langle q-q' , p \rangle }a\left( \frac{q+q'}{2},p \right) f(q') \, dq' \, dp  .
$$
Note that $\sigma_{\textrm{sub}}$ is usually defined for operators acting on half-densities: here we make the identification
$f \leftrightarrow fd\mu ^{1/2} $  between functions and half-densities, taking into account that the manifolds are equipped with densities.
This was a nice  original discovery of Leray (see \cite{GKL-64}).
The sub-principal symbol is characterized by the action of pseudo-differential operators on oscillatory functions as follows:
if $u(q) = b(q)\exp(i\tau S(q))$ with $b$, $S$ smooth and real-valued and, if $A\in \Psi^m (M)$, then
$$
\int _M A (u)  \bar{u} \, d\mu = \tau^m  \int _M \left( \sigma _p (A) (q,S'(q))+\tau^{-1}\sigma _{\textrm{sub}}(q,S'(q))\right)
\vert u(q)\vert^2\, d\mu (q)  + \mathrm{O}(\tau^{m-2}) .
$$
Moreover, we have the following properties:
\begin{itemize}
\item $\sigma_p(AB)= \sigma_p(A) \sigma_p(B)$, for any $A\in \Psi^m(M)$ and $B\in \Psi^l(M)$.
\item If $A\in \Psi^m(M)$ and $B\in \Psi^l(M)$, then
$[A,B]\in\Psi^{m+l-1}(M)$
and $\sigma_p([A,B])=\{ \sigma_p(A), \sigma_p(B)\}/i$, where the Poisson bracket is taken with respect to the canonical symplectic structure of $T^\star M$.

\item If $X$ is a vector field on $M$ and $X^\star $ is its formal adjoint in $L^2(M,\mu)$, then $X^\star X $ is a PDO of order $2$ such that $\sigma _p (X^\star X)=h_X^2  $ and $\sigma_{\textrm{sub}}(X^\star X)=0$.

\item PDOs act on Sobolev spaces: if $A\in  \Psi^m(M)$, then $A$ maps continuously the
space $H^s(M)$ to the space $H^{s-m}(M)$. It follows that two quantizations of a given symbol $a$ of order $0$ differ by a compact operator. 
\item To each distribution $T$  on $M$ is associated its wave-front set $WF(T)$,
 which is a closed sub-cone of  $T^\star M \setminus \{0\}$, whose projection onto $M$ is the singular support of $T$. 
More precisely, in local coordinates, we have
$(q,p)\notin WF(T)$ if and only if there exists $\chi \in C_0^\infty(M)$ with $\chi (q)\neq 0$ such that the Fourier transform of $\chi T$ is rapidly decaying in some conical neighborhood of $p$.
For every operator $A:C^{\infty} (M)\rightarrow \mathcal{D}'(N) $, of Schwartz kernel $k_A \in \mathcal{D}'(M \times N)$, we define
$$
WF'(A)=\{ (q,p;q',-p')\in T^\star M \times T^\star N \ \mid\ (q,q',p,p')\in WF(k_A) \} .
$$
The maps $WF$ and $WF'$ have nice set theoretical properties such as $WF (Au)\subset WF' (A) \circ WF(u)$, for every $u\in C^{\infty} (M)$.
\end{itemize}

\subsubsection{FIOs} 
Let $\chi : V \rightarrow W$ be a symplectic diffeomorphism from an open cone
$V \subset T^\star M$ to an open cone
$W \subset T^\star N$, where $M$ and $N$ are manifolds having the same dimension,
 respectively endowed with smooth non-vanishing measures $\mu$ and $\nu$. To this diffeomorphism is associated
a family of linear  operators $U:L^2(M,\mu)\rightarrow L^2(N , \nu)$, sometimes called ``quantizations of $\chi$'',
 with the following properties:
\begin{itemize}
\item $\{ (z,\chi (z)) \ \mid\ z\in V \} \, \cap\, WF' (U^\star U - \mathrm{id})=\emptyset $.
We say that $U$ is \textit{microlocally unitary} near the graph of $\chi $.
\item If $A\in \Psi^m (N)$, then $B=U^\star A U \in \Psi^m (M )$ and
the principal symbols satisfy on $V$ the relationship
$\sigma_p(B)\circ \chi =\sigma _p (A)$. This is sometimes called the \textit{Egorov theorem}.
\item If $\sigma _{\textrm{sub}}(A)=0$, then the same property holds for $B$ (see also Appendix \ref{app:Weinstein}).
\end{itemize}

\subsubsection{Pseudo-differential operators flat on $ \Sigma $}\label{app:flat}
Let us define the notion of flatness used in Section \ref{main:lemma}.

\begin{definition}\label{def_flat} Let $\Sigma $ be a closed sub-cone of $T^\star M \setminus \{0\}$.
Given any $k\in \N\cap \{\infty\}$ and any smooth function $f$ on $T^\star M$, the notation $f=\mathrm{O}_\Sigma(k)$ means that $f$ vanishes on $\Sigma$ at least at the order $k$. The word flat is used when $k=+\infty$.

A pseudo-differential operator $A$ (on any order) on $M$ is said to be flat on $\Sigma$ if, for any $(q_0,p_0)\in \Sigma$ and, for  any local canonical coordinate system around $q_0$, the full Weyl symbol of $A$ is flat at $(q_0,p_0)$. In this case, we write $A\in \mathrm{O}_\Sigma(\infty)$.
\end{definition}

\begin{lemma} 
The set of pseudo-differential operators that are flat on $\Sigma$ is a bilateral ideal in the algebra of classical pseudo-differential operators on $M$.
\end{lemma}

\begin{proof}
This follows from the fact that the full Weyl symbol of a product of pseudo-differential operators at the
point $(q_0,p_0)$  is 
given modulo flat terms by sums of products of partial derivatives of the symbols of both operators
at the same point $(q_0,p_0)$.
\end{proof}

The following lemma is required in Section \ref{main:lemma}.

\begin{lemma} \label{lemm:facto}
If $C \in \Psi^m(M)$ is flat on $\Sigma$, then there exists $C_1\in \Psi^{m-2}(M)$, which is flat on $\Sigma$, such that $C=C_1  \triangle _{sR}$ modulo a smoothing operator that is flat on $\Sigma$. 
\end{lemma}

\begin{proof} 
By using a partition of unity, $C$ can be written as a sum of operators compactly supported in an atlas of $M$. Hence it suffices to prove the statement in $T^\star U$, where $U$ is a chart of $M$.
Let $c'$ be the quotient of the principal symbol $c$  of $C$ by $g^\star=\sigma_P(-\triangle _{sR})$.
Since $c$ is flat on $\Sigma $ and $g^\star$ vanishes along $\Sigma $ exactly
at order $2$, $c'$ is a smooth symbol of order $m-2$ that is flat on $\Sigma$.
Then the operator $C'=\OpW(c')$ is flat on $\Sigma$  and $C= C'\triangle _{sR} +R_1 $
where  $R_1\in \Psi^{m-1}(M)$ is also flat on $\Sigma$.  Then we
iterate the construction on $R_1$. 
\end{proof}

\subsection{The Weinstein argument}\label{app:Weinstein}
We provide here an argument of Weinstein given in \cite{We-74}, leading to the following result.

\begin{proposition}\label{prop_Weinstein}
Let $X$ be a smooth manifold.
Let $\Delta$ be a pseudo-differential operator defined in some cone $C\subset T^\star X$.
We denote by $p$ the principal symbol of $\Delta$, and we assume that the sub-principal symbol of $\Delta $ vanishes.
Let  $\chi : C \rightarrow C'\subset T^\star Y $ be a canonical transformation, where $Y$ is another smooth manifold. 
Then, there exists a microlocally unitary Fourier Integral Operator $U_\chi$, associated with $\chi $, such that $U_\chi\Delta U_\chi^\star =B$, 
where $B$ is a pseudo-differential operator in $C'$ whose principal symbol
 is $p\circ \chi^{-1}$ (general Egorov theorem) and whose sub-principal symbol vanishes.
\end{proposition}

\begin{proof}
The proof uses in a strong way the symbolic calculus of Fourier Integral Operators, for which we refer to the book \cite{Du-96} or to the paper \cite{We-74}. Let us sketch the argument.
We choose the Fourier Integral Operator $U_\chi$ associated with the canonical transformation $\chi$ such that its principal symbol is constant of modulus $1$ and $U_\chi$ is microlocally unitary, i.e., $ U_\chi^\star U_\chi =\mathrm{id}$ in the cone $C$. This is possible as follows:
we choose a first $U_0$ with only the prescription of the principal symbol, then $ U_0^\star U_0 =\mathrm{id} + A$ where $A $ is a self-adjoint pseudo-differential operator in $\Psi^{-1}(C)$.
If $D=(\mathrm{id} +A)^{-1/2}$ in $C$, then we take $U_\chi=U_0 D$.

Denoting by $K(x,y)$ the Schwartz kernel of $U_\chi$, if $B=U_\chi\Delta U_\chi^\star $, the relation $BU_\chi -U_\chi \Delta \sim 0$ is written as
$$
\left( \Delta_x \otimes \mathrm{id}_Y - \mathrm{id}_X \otimes B_y \right)K \sim 0 .
$$
The distribution $K$ is a Lagrangian distribution associated with a submanifold of $C\times C'$ which is the graph of $\chi $.
If we assume that the principal symbol of $U_\chi$ is a constant of modulus $1$, then the sub-principal symbol of the right-hand side, which is $0$, is the sum of Lie derivatives of the principal symbol of $K$, which vanish due to the choice of $U_\chi$ and of the product of the sub-principal symbol of $\mathrm{id}_X \otimes B_y$  by the non-vanishing symbol of $U_\chi$. This implies that the latter vanishes.
This is the argument of Weinstein. We have only to take care of the fact that tensor products are pseudo-differential operators only in some cones of the product of the cotangent spaces
where $\xi$ and $\eta$ are of comparable sizes.  
\end{proof}

\section{Appendix: Darboux-Weinstein lemma}\label{app:DW}
We have the following easy generalization of the well-known Darboux-Weinstein lemma (see \cite{We-71}).

\begin{lemma}\label{lem_gen_Weinstein}
Let $N$ be a manifold endowed with two symplectic forms $\omega_1$ and $\omega_2$, and let $P$ be a compact submanifold of $N$ along which $\omega_1=\omega_2+\mathrm{O}_P(k)$, for some $k\in\N^*\cup\{+\infty\}$. Then there exist open neighborhoods $U$ and $V$ of $P$ in $N$ and a diffeomorphism $f:U\rightarrow V$ such that $f=\mathrm{id}_N+\mathrm{O}_P(k+1)$ and $f^\star\omega_2=\omega_1$. Moreover, if $N$ has a conic structure, then the diffeomorphism $f$ can be chosen to be homogeneous with respect to that conic structure.
\end{lemma}

The most usual statement of that lemma is when $k=1$, and then the usual conclusion is that $f=\mathrm{id}_N+O_P(1)$; actually, already in that case we have the better conclusion that $f=\mathrm{id}_N+O_P(2)$ (as is well known, and as it is proved for instance in \cite[Lemma 43.11 p. 462]{KrieglMichor}), i.e., $df(q)=\mathrm{id}$ for every $q\in N$, or in other words, $f$ is tangent to the identity.

\begin{proof}
We follow the standard argument (see, e.g., \cite{McDuffSalamon}). We define the closed two-form $\omega(t)=\omega_1+t(\omega_2-\omega_1)$, for every $t\in[0,1]$. Let $U$ be a neighborhood of $P$ in which $\omega(t)$ is nondegenerate for every $t$. By the relative Poincar\'e lemma, since $\omega_1$ and $\omega_2$ agree along $P$, shrinking $U$ if necessary, there exists a one-form $\eta$ on $U$ such that $\omega_1-\omega_2=d\eta$, with $\eta_x=0$ for every $x\in P$. Since $\omega_1=\omega_2+\mathrm{O}_P(k)$, we can actually choose $\eta$ such that $\eta= \mathrm{O}_P(k+1)$. Indeed, as it is well known in the relative Poincar\'e lemma, we can choose $\eta = Q(\omega_1-\omega_2)$, where $Q$ is defined by $Q\omega=\int_0^1F(t)^\star\iota_{Y(t)}\omega\, dt$, where $Y(t)$, at the point $y=F(t,x)$, is the vector tangent to the curve $F(s,x)$ at $s=t$, and $(F(t))_{0\leq t\leq 1}$ is a smooth homotopy from the local projection onto $P$ (in a tubular neighborhood of $P$) to the identity, fixing $P$.

The diffeomorphism $f$ is then constructed by the Moser trick. The time-dependent vector field $X(\cdot)$ defined for every $t$ by $\iota_{X(t)}\omega(t)=\eta$ generates the time-dependent flow $f(\cdot)$ (satisfying $\dot f(t)=X(t)\circ f(t)$, $f(0)=\mathrm{id}_N$),
 and we have
$$
\frac{d}{dt}f(t)^\star\omega(t) = f(t)^\star \mathcal{L}_{X(t)}\omega(t) + f(t)^\star\dot\omega(t) = f(t)^\star d (\iota_{X(t)}\omega(t)-\eta) = 0,
$$
whence $\omega_1=f(1)^\star\omega_2$. We set $f=f(1)$.

Since $\iota_{X(t)}\omega(t)=\eta=\mathrm{O}_P(k+1)$, it follows that $X(t)=\mathrm{O}_P(k+1)$ and hence $f(t)=\mathrm{id}_N+\int_0^t X(s)\circ f(s)\, ds=\mathrm{id}_N+\mathrm{O}_P(k+1)$ for every $t\in[0,1]$. The lemma is proved.

Let us now prove that, if $N$ is conic, with a conic structure $x\mapsto \lambda\cdot x$, for $\lambda>0$ and $x\in N$, then $f$ is homogeneous. The two-form $\omega=\omega_1-\omega_2$ is then conic, meaning that $\omega_{\lambda\cdot x}(\lambda\cdot v_1,\lambda\cdot v_2) = \lambda\omega_x(v_1,v_2)$, for all $\lambda>0$, $x\in N$ and $v_1,v_2\in T_xN$. It is easy to see that the homotopy operator $Q$ considered above can be chosen to be homogeneous. Then $\eta=Q\omega$ is homogeneous as well. It easily follows that the time-dependent vector field $X(\cdot)$ of the Moser trick is homogeneous (meaning that $X(t,\lambda\cdot x)=\lambda\cdot X(t,x)$) and hence that its flow is homogeneous. The conclusion follows.
\end{proof}


\section{Appendix on Birkhoff normal forms}

\subsection{The general procedure}\label{app:BNF}
In this section, we recall how to derive Birkhoff normal forms in a general setting.

We consider the  submanifold $X=X\times\{0\}$ of a symplectic manifold $Y= X\times \R^d$ with a Poisson bracket $\{\cdot,\cdot\}$.

\subsubsection{Taylor expansions along $X$}\label{app:taylorexp}
We will always consider germs of objects defined in some neighborhood of $X$ in $Y$.
We consider a (germ of) real-valued smooth function $S$ on $Y$ such that $S=\mathrm{O}_X (2)$.
Then the flow of the corresponding Hamiltonian vector field
$\vec{S}$ on $Y$ is defined on an interval containing $[0,1]$ in some neighborhood of $X$ in $Y$.
The (germ of) symplectomorphism $\exp(\vec S)$ (flow of $\vec{S}$ at time $1$) satisfies $\exp(\vec S) (x)=x$ for every $x\in X$.
Given any germ of function $f$, denoting by $\mathcal{T}(f)$ the Taylor expansion of $f$ along $X$, we have $\mathcal{T}(f)=\sum _{k=0}^{+\infty} f_k $ where $f_k$ is a polynomial on $\R^d$ whose coefficients are smooth functions on $X$.
Setting $\mathrm{ad}\, S.f=\{S,f\}$, we have 
$$
\mathcal{T}(f\circ\exp(\vec S))=\exp(\mathrm{ad}\, S) \mathcal{T}(f).
$$
This is a well defined power series because $\mathrm{ad}^n S.f =\mathrm{O}_X (n)$.
Another important property is that, if $S_k =\mathrm{O}_X(k)$ for every $k\geq 2$, then 
$\exp(\mathrm{ad}\, S_k)=\mathrm{id}+\mathrm{O}_X(k-1)$ and the composition
$$
\exp(\mathrm{ad}\, S_n)\circ\exp(\mathrm{ad}\, S_{n-1})\circ\cdots\circ\exp(\mathrm{ad}\, S_2)
$$
converges formally, as $n\rightarrow +\infty $, to the Taylor expansion of a symplectic diffeomorphism $\chi_\infty$
satisfying $\chi_\infty=\mathrm{id}+\mathrm{O}_X(1)$.

Indeed, this follows from the well-known Borel theorem: given any nonzero integer $n$, given a Fr\'echet space $E$ and a (multi-index) sequence $(a_\alpha)_{\alpha\in\N^n}$ in $E$, there exists a smooth  function $f:\R^n\rightarrow E$ whose infinite Taylor expansion at $0$
 is $\sum_{\alpha\in\N^n} a_\alpha x^\alpha$. Moreover, if $E$ has a conic structure, then $f$ can be chosen to be homogeneous for that conic structure.
 
The map $\chi_\infty$ is a local homogeneous diffeomorphism because it is tangent to the identity.
Note that this diffeomorphism, constructed by the Borel theorem, is not a priori symplectic, but
it satisfies $\chi_\infty^\star\omega=\omega+\mathrm{O}_X(\infty)$, and it is then possible to modify it by composing it with a homogeneous diffeomorphism tangent to the identity, in order to finally get a homogeneous symplectomorphism: this is done thanks to the Darboux-Weinstein lemma (Lemma \ref{lem_gen_Weinstein} in Appendix \ref{app:DW}, applied with $k=\infty$).

\subsubsection{Cohomological equations and normal forms.}
Let  $\mathcal{G}_k$ (resp., let $\mathcal{S }_l$) be a subspace of the space of homogeneous polynomials of degree $k$ (resp., of degree $l$) on $\R^d$
with smooth coefficients on $X$, for which
there exists $p\geq 0 $ such that 
$$
\{ \mathcal{S }_l,\mathcal{G}_k  \} \subset \mathcal{G}_{l+k-p } .
$$
We assume that $k_0+p-2 \geq 2$ and we consider the action of 
$\exp(\mathrm{ad}\, S_{k+p-2})$ on a formal series $H=H_2 +\sum _{k=k_0}^{+\infty} H_k $ with  $H_k \in \mathcal{G}_k$.
We have 
$$
\exp(\mathrm{ad}\, S_{m+p-2} ) H = H_2 +   \sum _{k=k_0 }^{m-1}H_k +  \{ S_{m+p-2} , H_2 \}+ H_{m} + R ,
$$
where the terms in the remainder $R$ are of degree $ k+n(m -2) $ (action of $n$ brackets on a term in  $\mathcal{G}_k$),
which is greater than $m$ provided that either ($k\geq 3$ and  $n \geq 1 $) or ($k=2$ and $n\geq 2$).
Moreover, there is only a finite number of terms of each degree. 
This implies that, for $k_0 \geq \max (4-p, 3)$, the series is formally convergent and
we can change the series starting from the term  of degree $k$  by adding a bracket of the form
$ \{ S_{k+p-2} , H_2 \} $. The cohomological equation is then 
written as
$$
\{ S_{k+p-2} , H_2 \} + H_k = \bar{H}_k ,
$$
where $\bar{H}_k$ is chosen in some suitable subspace $\mathcal{N}_k $  of $\mathcal{G}_k  $.
Finally, by considering the formal (infinite) composition of symplectic diffeomorphisms associated 
with the sequence $(S_{k+p-2})_{k\geq k_0}$, we get the normal form 
$$
H \sim H_2 +\sum _{k=k_0}^{+\infty} \bar{H}_k +\mathrm{O}_X (\infty ) .
$$

\subsubsection{Homogeneous canonical transformation.}
Let us assume that $X$ is an homogeneous symplectic cone with $\tau^\star \omega =\tau\cdot \omega $,
that $d=2r$ with the canonical symplectic structure, 
and that $\tau\cdot (x,n)=(\tau x,\sqrt{\tau }n)$.
If $S$ is homogeneous of degree $1$, then the associated canonical transformation is homogeneous of degree $0$, and
commutes with $\tau$. We can then apply the previous reductions on some homogeneous functions $H$ of fixed degree.

\subsection{Proof of Proposition \ref{prop:invBNF}}\label{app:proof:prop:invBNF}
Using that $H=g^\star\circ\chi_2^{-1}=\rho I+\mathrm{O}_\Sigma(3)$, and since $\mathcal{F}_{2,3}^{\mathrm{inv}}=\{0\}$ (see \eqref{sumT}), we start with the fact that
\begin{equation}\label{expdebut}
\begin{split}
H &= H_2  + \mathrm{O}_\Sigma(\infty) \mod \mathcal{F}_{2,\geq 3} \\
&= H_2 + \mathrm{O}_\Sigma(\infty) \mod \left( \mathcal{F}_{2,\geq 3}^0 \oplus \mathcal{F}_{2,\geq 4}^{\mathrm{inv}} \right) ,
\end{split}
\end{equation}
where we have set $H_2=\rho I$.

Our objective is to construct a homogeneous symplectomorphism $\varphi$ allowing us to remove from \eqref{expdebut} all terms in $\mathcal{F}_2^0$, in order to obtain the normal form 
$$
H\circ \varphi = H_2 + \mathrm{O}_\Sigma(\infty) \mod \mathcal{F}_{2,\geq 4}^{\mathrm{inv}},
$$
by using Lie transforms generated by appropriate elements of $\mathcal{F}_1^0$. 

\medskip

We proceed by (strong) recurrence on $k$, starting at $k=3$, by constructing, at each step, a local homogeneous symplectomorphism $\varphi_k$ satisfying $\varphi_k=\mathrm{id}+\mathrm{O}_\Sigma(2)$ and such that
$$
H\circ \varphi_3\circ\cdots \circ \varphi_k = H_2 + \mathrm{O}_\Sigma(\infty)
 \mod \left( \mathcal{F}_{2,\geq k+1}^0 \oplus \mathcal{F}_{2,\geq 4}^{\mathrm{inv}} \right) ,
$$
and we search each $\varphi_k$ in the form $\varphi_k = \varphi_k(1) = \exp (\vec F^0_k)$, where $\varphi_k(t) = \exp(t \vec F^0_k)$ is the flow at time $t$ generated by an adequate Hamiltonian function $F^0_k\in \mathcal{F}_{1,k}^0$ (note that $\varphi_k$ is then symplectic, as desired).

Before going to the recurrence, let us note that, for every $k\geq 3$, there exists a (small enough) conic neighborhood $C_k$ of $\Sigma_U^+$ such that the flow $\varphi_k$ is well defined on $[0,1]\times C_k$; moreover, since we are going to compose these symplectomorphisms, we choose $C_k$ such that $\varphi_{k+1}$ maps $C_{k+1}$ to $C_k$, for every $k\geq 3$.
Indeed, since $F^0_k = \mathrm{O}_\Sigma(k)$, we have 
$$
\varphi_k(t) = \mathrm{id} + \int_0^t \vec F^0_k\circ \varphi_k(s) \, ds = \mathrm{id} + \mathrm{O}_\Sigma(k-1)
$$
uniformly with respect to $t$ on compact intervals, and the claim follows, as well as the expansion $\varphi_k=\mathrm{id}+\mathrm{O}_\Sigma(2)$.

\medskip

Let us now make the construction by recurrence.

For $k=3$, we want to prove that there exists $F^0_3\in \mathcal{F}_{1,3}^0$ such that, setting $\varphi_3 = \exp (\vec F^0_3)$, we have $H\circ\varphi_3 = H_2 \mod \left( \mathcal{F}_{2,\geq 4}^0 \oplus \mathcal{F}_{2,\geq 4}^{\mathrm{inv}} \right)$.
Using that $\frac{d}{dt}(H\circ\varphi_3(t))=\{F^0_3,H\}_{\tilde\omega}\circ\varphi
_3(t)$, and since the flow is well defined on $[0,1]\times C_3$, we have
\begin{equation}\label{exp2varphi3}
H\circ\varphi_3(t) = H + \{ F_3^0, H\}_{\tilde\omega} + \mathrm{O}\left( \frac{t^2}{2} \{ F_3^0,\{ F_3^0, H \}_{\tilde\omega} \}_{\tilde\omega} \circ\varphi_3(t) \right),
\end{equation}
on $[0,1]\times C_3$. 
Using \eqref{expdebut}, we have $H = H_2 + H_3^0 \mod \left( \mathcal{F}_{2,\geq 4}^0 \oplus \mathcal{F}_{2,\geq 4}^{\mathrm{inv}} \right)$, with $H_3^0\in \mathcal{F}_{2,4}^0$.
Hence, taking $t=1$ in \eqref{exp2varphi3}, using \eqref{cohom2} and \eqref{cohom3}, we infer that
$$
H\circ\varphi_3 = H_2 + H_3^0 + \{ F_3^0, H_2 \}_{\tilde\omega} + \mathrm{O}_\Sigma(\infty) \mod \left( \mathcal{F}_{2,\geq 4}^0 \oplus \mathcal{F}_{2,\geq 4}^{\mathrm{inv}} \right).
$$
Therefore, we have to solve the cohomological equation
$$
\{ H_2, F_3^0 \}_{\tilde\omega} = H_3^0 + \mathrm{O}_\Sigma(\infty) \mod \left( \mathcal{F}_{2,\geq 4}^0 \oplus \mathcal{F}_{2,\geq 4}^{\mathrm{inv}} \right) ,
$$
which has a solution $F^0_3\in \mathcal{F}_{2,3}^0$ by using \eqref{cohom2}.

\medskip

Let us assume that we have constructed $\varphi_3,\ldots,\varphi_{k-1}$ such that
$$
H\circ \varphi_3\circ\cdots \circ\varphi_{k-1} = H_2 + \mathrm{O}_\Sigma(\infty) \mod \left( \mathcal{F}_{2,\geq k}^0 \oplus \mathcal{F}_{2,\geq 4}^{\mathrm{inv}} \right) .
$$
We want to prove that there exists $F^0_k\in \mathcal{F}_{2,k}^0$ such that,
 setting $\varphi_k = \exp (\vec F^0_k)$, we have
$$
H\circ \varphi_3\circ\cdots \circ\varphi_k = H_2 + \mathrm{O}_\Sigma(\infty) \mod \left( \mathcal{F}_{2,\geq k+1}^0
 \oplus \mathcal{F}_{2,\geq 4}^{\mathrm{inv}} \right) .
$$
Using that
$$
\frac{d}{dt}(H\circ\varphi_3\circ\cdots \circ\varphi_k(t))=
\{F^0_k,H\circ \varphi_3\circ\cdots \circ\varphi_{k-1} \}_{\tilde\omega}\circ\varphi_k(t) ,
$$
and since the flow $\varphi_k$ is well defined on $[0,1]\times C_k$, we have
\begin{multline}\label{exp2varphik}
H\circ\varphi_3\circ\cdots \circ\varphi_k(t) = H\circ\varphi_3\circ\cdots \circ\varphi_{k-1} +
 \{F^0_k,H\circ \varphi_3\circ\cdots \circ\varphi_{k-1} \}_{\tilde\omega} \\
+ \mathrm{O} \left(  \frac{t^2}{2} \{ F_k^0, \{ F_k^0,  H\circ \varphi_3\circ\cdots \circ\varphi_{k-1} \}_{\tilde\omega} \}_{\tilde\omega}  \right) ,
\end{multline}
on $[0,1]\times C_k$.
Using the recurrence assumption, we have 
$$
H\circ \varphi_3\circ\cdots \circ\varphi_{k-1} = H_2 + \tilde H_k^0 + \mathrm{O}_\Sigma(\infty)
 \mod \left( \mathcal{F}_{2,\geq k+1}^0 \oplus \mathcal{F}_{2,\geq 4}^{\mathrm{inv}} \right) ,
$$
for some $\tilde H_k^0\in \mathcal{F}_{2,k}^0$. Hence, taking $t=1$ in \eqref{exp2varphik}, using \eqref{cohom2} and \eqref{cohom3}, we infer that
$$
H\circ \varphi_3\circ\cdots \circ\varphi_k = H_2 + \tilde H_k^0 + \{ F^0_k, H_2ö€Žà¤ Š\}_{\tilde\omega} + \mathrm{O}_\Sigma(\infty) \mod
 \left( \mathcal{F}_{2,\geq k+1}^0 \oplus \mathcal{F}_{2,\geq 4}^{\mathrm{inv}} \right).
$$
Therefore, we have to solve the cohomological equation
$$
\{ H_2, F_k^0 \}_{\tilde\omega} = \tilde H_k^0 + \mathrm{O}_\Sigma(\infty) \mod \left( \mathcal{F}_{2,\geq k+1}^0
 \oplus \mathcal{F}_{2,\geq 4}^{\mathrm{inv}} \right) ,
$$
which has a solution $F^0_k\in \mathcal{F}_{2,k}^0$ by using \eqref{cohom2}.

The recurrence is established.

\medskip

By definition, $\varphi_k$ is the flow at time $1$ generated by the Hamiltonian function $F_k^0\in \mathcal{F}_1$. By definition of
 $\mathcal{F}_1$, we have $F^0_k(\lambda\cdot(\sigma,u,v)) = \lambda F^0_k(\sigma,u,v)$, that is, $F^0_k$ is homogeneous for the conic structure defined by \eqref{defconicstructure}. Then $\varphi_k$ is indeed homogeneous, as a consequence of the following general lemma, that we recall for completeness.

\begin{lemma}\label{lem_homogeneous}
Let $(N,\omega)$ be a conic symplectic manifold, with a conic structure $x\mapsto \lambda\cdot x$, for $\lambda>0$ and $x\in N$.
Let $H$ be a smooth Hamiltonian function on $N$, which is homogeneous, meaning that $H(\lambda\cdot x)=\lambda H(x)$ for every $\lambda>0$
 and every $x\in N$. Then the associated Hamiltonian vector field $\vec H$ is homogeneous, in the sense that
 $\vec H(\lambda\cdot x)=\lambda\cdot\vec H(x)$, and as a consequence, the generated flow $\exp(t\vec H)$ is homogeneous as well.
\end{lemma}

\begin{proof}[Proof of Lemma \ref{lem_homogeneous}.]
By definition, the symplectic form $\omega$ is conic, in the sense that $\omega_{\lambda\cdot x}(\lambda\cdot v_1,\lambda\cdot v_2) =
 \lambda\omega_x(v_1,v_2)$, for all $\lambda>0$, $x\in N$ and $v_1,v_2\in T_xN$. The Hamiltonian vector field $\vec H$ is defined at
 any point $x\in N$ by $\omega_x(\vec H(x),v) = dH(x).v$, for every $v\in T_x N$. Since $H$ is homogeneous, by differentiation we get
 that $dH(\lambda\cdot x).(\lambda\cdot v)=\lambda dH(x).v$, for every $v\in T_xN$, and therefore,
\begin{equation*}
\omega_{\lambda\cdot x}(\vec H(\lambda\cdot x),\lambda\cdot v) = dH(\lambda\cdot x).(\lambda\cdot v) = \lambda dH(x).v =
 \lambdaö€Žà¤ Š\omega_x(\vec H(x),v) = \omega_{\lambda\cdot x}(\lambda\cdot\vec H(x),\lambda\cdot v),
\end{equation*}
from which it follows that $\vec H(\lambda\cdot x)=\lambda\cdot\vec H(x)$.
\end{proof}

Let us now finish the proof.

We consider the infinite composition $\varphi_3\circ\cdots\circ\varphi_k\cdots $ which is convergent in the sense of formal series along $\Sigma _U ^+$.
By the Borel theorem recalled in Appendix \ref{app:taylorexp}, there
 exists a smooth homogeneous mapping $\varphi$ that is the Borel summation of that formal composition, i.e., such that $\varphi =
 \varphi_3\circ\cdots\circ\varphi_k\cdots + \mathrm{O}_\Sigma(\infty)$.
Clearly, $\varphi$ is a local homogeneous diffeomorphism (because it is tangent to the identity), and we have $\varphi = \mathrm{id} + \mathrm{O}_\Sigma(2)$ and $H\circ \varphi = H_2 + \mathrm{O}_\Sigma(\infty) \mod \mathcal{F}_{2,\geq 4}^{\mathrm{inv}}$.
Note that $\varphi$ may not be a symplectomorphism, however, by construction we have $\varphi^\star \tilde \omega = \tilde\omega +
 \mathrm{O}_\Sigma(\infty)$.
It is however possible to modify the homogeneous diffeomorphism $\varphi$, by composing it with a homogeneous diffeomorphism tangent
 to identity at infinite order, so as to obtain exactly $\varphi^\star \tilde \omega = \tilde\omega$ (and thus, $\varphi$ is a homogeneous
 symplectomorphism). This is done thanks to the Darboux-Weinstein lemma, given in Appendix \ref{app:DW}, applied with $k=+\infty$.

\subsection{Proof of Proposition \ref{prop:MelroseBNF}}\label{app:proof:prop:MelroseBNF}
Following the previous section, our objective is now to construct a symplectomorphism allowing us to remove all terms in $\mathcal{F}_2^{\mathrm{inv}}$.
Although the proof is similar to the one done in the previous section, there are several differences and subtleties which make it preferable to write the whole proof in details.

We proceed by (strong) recurrence on $k$, starting at $k=1$, by constructing, at each step, a local homogeneous symplectomorphism $\psi_{2k}$ satisfying $\psi_{2k}=\mathrm{id}+\mathrm{O}_\Sigma(2)$ and such that
$$
\tilde H\circ \psi_2\circ\cdots \psi_{2k} = H_2 + \mathrm{O}_\Sigma(\infty) \mod \mathcal{F}_{2,\geq 2k+2}^{\mathrm{inv}} ,
$$
and we search each $\psi_{2k}$ in the form $\psi_{2k} = \psi_{2k}(1) = \exp (\vec F^{\mathrm{inv}}_{2k})$, where $\psi_{2k}(t) = \exp(t \vec F^{\mathrm{inv}}_{2k})$ is the flow at time $t$ generated by an adequate Hamiltonian function $F^{\mathrm{inv}}_{2k}\in \mathcal{F}_{1,2k}^{\mathrm{inv}}$ (note that $\psi_{2k}$ will indeed be homogeneous by Lemma \ref{lem_homogeneous}).

Before going to the recurrence, let us note that, as in the previous section, for every $k\geq 1$, there exists a (small enough) conic neighborhood $C'_{2k}$ of $\Sigma_U^+$ such that the flow $\psi_{2k}$ is well defined on $[0,1]\times C'_{2k}$; moreover, since we
 are going to compose these symplectomorphisms, we choose $C'_{2k}$ such that $\psi_{2k+2}$ maps $C'_{2k+2}$ to $C'_{2k}$, for every $k\geq 1$.
Indeed, since $F^{\mathrm{inv}}_{2k} = \mathrm{O}_\Sigma(2k)$, we have 
$$
\psi_{2k}(t) = \mathrm{id} +
 \int_0^t \vec F^{\mathrm{inv}}_{2k}\circ \psi_{2k}(s) \, ds = \mathrm{id} + \mathrm{O}_\Sigma(2k-1)=\mathrm{id}+\mathrm{O}_\Sigma(1)
$$ 
uniformly with respect to $t$ on compact intervals, and the claim follows. This argument is however not sufficient in order to establish that $\psi_{2k}=\mathrm{id}+\mathrm{O}_\Sigma(2)$, because we start with $k=1$. To prove this property, we use temporarily the notation $x=(\sigma,u,v)=(x_1,\ldots,x_6)$, and we consider the six functions $\pi_i(x)=x_i$ in $\R^6$, $i=1,\ldots,6$. 
Writing $F_{2k}^{\mathrm{inv}}(q,s,u,v) = a_{2k}(q,s)(u^2+v^2)^k$, and using \eqref{fundform}, we infer that
\begin{equation*}
\frac{d}{dt} \pi_i\circ \psi_{2k}(t) = \{ F^{\mathrm{inv}}_{2k}, \pi_i \}_{\tilde\omega}\circ \psi_{2k}(t)
= \left\{ \begin{array}{ll}
\{a_{2k},\pi_i\}_{\omega_W}(\pi_5^2+\pi_6^2)^k \circ\psi_{2k}(t) = \mathrm{O}_\Sigma(2) & \textrm{if}\ i\leq 4, \\
0 & \textrm{if}\ i=5,6,
\end{array}\right.
\end{equation*}
and therefore $\pi_i\circ\psi_{2k} = \pi_i + \mathrm{O}_\Sigma(2)$, for $i=1,\ldots,6$. We conclude that $\psi_{2k}=\mathrm{id}+\mathrm{O}_\Sigma(2)$.

\medskip

Let us now make the construction by recurrence.

For $k=1$, we want to prove that there exists $F^{\mathrm{inv}}_2\in \mathcal{F}_{1,2}^{\mathrm{inv}}$ such that, setting
 $\psi_2 = \exp (\vec F^{\mathrm{inv}}_2)$, we have $\tilde H\circ\psi_2 = H_2 \mod \mathcal{F}_{2,\geq 6}^{\mathrm{inv}}$.
Using that $\frac{d}{dt}(\tilde H\circ\psi_2(t))=\{F^{\mathrm{inv}}_2,\tilde H \}_{\tilde\omega}\circ\psi_2(t)$, and since the 
flow is well defined on $[0,1]\times C'_2$, we have
\begin{equation}\label{exp2psi3}
\tilde H\circ\psi_2(t) = \tilde H + \{ F_2^{\mathrm{inv}}, \tilde H\}_{\tilde\omega} + \mathrm{O}\left( \frac{t^2}{2}
 \{ F_2^{\mathrm{inv}},\{ F_2^{\mathrm{inv}}, \tilde H \}_{\tilde\omega} \}_{\tilde\omega} \circ\psi_2(t) \right),
\end{equation}
on $[0,1]\times C'_2$. 
We have $\tilde H = H_2 + \tilde H_4^{\mathrm{inv}} + \mathrm{O}_\Sigma(\infty) \mod \mathcal{F}_{2,\geq 6}^{\mathrm{inv}}$,
 with $\tilde H_4^{\mathrm{inv}}\in \mathcal{F}_{2,4}^{\mathrm{inv}}$.
Hence, taking $t=1$ in \eqref{exp2psi3}, using \eqref{cohom1} and \eqref{cohom4}, we infer that
$$
\tilde H\circ\psi_2 = H_2 + \tilde H_4^{\mathrm{inv}} + \{ F_2^{\mathrm{inv}} , H_2 \}_{\tilde\omega} + 
\mathrm{O}_\Sigma(\infty) \mod \mathcal{F}_{2,\geq 6}^{\mathrm{inv}} .
$$
Therefore, we have to solve the cohomological equation
$$
\{ H_2, F_2^{\mathrm{inv}} \}_{\tilde\omega} = \tilde H_4^{\mathrm{inv}} + \mathrm{O}_\Sigma(\infty) \mod \mathcal{F}_{2,\geq 6}^{\mathrm{inv}}  ,
$$
which has a solution $F^{\mathrm{inv}}_4\in \mathcal{F}_{2,4}^{\mathrm{inv}}$ by using \eqref{cohom1}.

It is important to note that, thanks to \eqref{cohom4}, the whole procedure done in this second step takes place
 in $\mathcal{F}_{2,\geq 2}^{\mathrm{inv}}$ (if terms in $\mathcal{F}_{2,\geq 2}^0$ were to appear again then our two-steps procedure would fail).
 This kind of triangular stability is crucial.

\medskip

Let us assume that we have constructed $\psi_2,\ldots,\psi_{2k-2}$ such that 
$$
\tilde H\circ \psi_2\circ\cdots \circ\psi_{2k-2} =  H_2  \mod \mathcal{F}_{2,\geq 2k}^{\mathrm{inv}} .
$$
We want to prove that there exists $F^{\mathrm{inv}}_{2k}\in \mathcal{F}_{2,2k}^{\mathrm{inv}}$ such that, setting $\varphi_{2k} = \exp (\vec F^{\mathrm{inv}}_{2k})$, we have
$$
\tilde H\circ \psi_2\circ\cdots \circ\psi_{2k} = H_2 + \mathrm{O}_\Sigma(\infty) \mod \mathcal{F}_{2,\geq 2k+2}^{\mathrm{inv}}  .
$$
Using that 
$$
\frac{d}{dt}(\tilde H\circ\psi_2\circ\cdots \circ\psi_{2k}(t))=\{F^{\mathrm{inv}}_{2k},\tilde H\circ \psi_2\circ\cdots  \circ\psi_{2k-2} \}_{\tilde\omega}\circ\psi_{2k}(t) ,
$$
and since the flow $\psi_{2k}$ is well defined on $[0,1]\times C'_{2k}$, we have
\begin{multline}\label{exp2psik}
\tilde H\circ\psi_2\circ\cdots \circ\psi_{2k}(t) = \tilde H\circ\psi_2\circ\cdots \circ\psi_{2k-2} + \{F^{\mathrm{inv}}_{2k},\tilde H\circ
 \psi_2\circ\cdots \circ\psi_{2k-2} \}_{\tilde\omega} \\
+ \mathrm{O} \left(  \frac{t^2}{2} \{ F_{2k}^{\mathrm{inv}}, \{ F_{2k}^{\mathrm{inv}}, \tilde H\circ \psi_2\circ\cdots
 \circ\psi_{2k-2} \}_{\tilde\omega} \}_{\tilde\omega}  \right) ,
\end{multline}
on $[0,1]\times C'_{2k}$.
Using the recurrence assumption, we have 
$$
\tilde H\circ \psi_2\circ\cdots \circ\psi_{2k-2} = H_2 + \tilde{\tilde{H}}_{2k}^{\mathrm{inv}} + \mathrm{O}_\Sigma(\infty) \mod
 \mathcal{F}_{2,\geq 2k+2}^{\mathrm{inv}}  ,
$$
for some $\tilde{\tilde{H}}_{2k}^{\mathrm{inv}}\in \mathcal{F}_{2,2k}^{\mathrm{inv}}$. Hence, taking $t=1$ in \eqref{exp2psik}, using 
\eqref{cohom1} and \eqref{cohom4}, we infer that
$$
\tilde H\circ \psi_2\circ\cdots \circ\psi_{2k} = H_2 + \tilde{\tilde{H}}_{2k}^{\mathrm{inv}} + \{ F^{\mathrm{inv}}_{2k}, H_2ö€Žà¤ Š\}_{\tilde\omega}
 + \mathrm{O}_\Sigma(\infty) \mod \mathcal{F}_{2,\geq 2k+2}^{\mathrm{inv}} .
$$
Therefore, we have to solve the cohomological equation
$$
\{ H_2, F_{2k}^{\mathrm{inv}} \}_{\tilde\omega} = \tilde{\tilde{H}}_{2k}^{\mathrm{inv}} + \mathrm{O}_\Sigma(\infty)
 \mod \mathcal{F}_{2,\geq 2k+2}^{\mathrm{inv}}  ,
$$
which has a solution $F_{2k}^{\mathrm{inv}}\in \mathcal{F}_{2,2k}^{\mathrm{inv}}$ by using \eqref{cohom1}.

The recurrence is established.

\medskip

Considering, as in the first step, the formal infinite composition $\psi_2\circ\cdots\circ\psi_{2k}\cdots$, by the Borel theorem, there exists a smooth homogeneous mapping $\psi$ such that $\psi = \psi_2\circ\cdots\circ\psi_{2k}\cdots + \mathrm{O}_\Sigma(\infty)$.
By construction we have $\psi^\star \tilde \omega = \tilde\omega + \mathrm{O}_\Sigma(\infty)$, and using again Lemma \ref{lem_gen_Weinstein} (Appendix \ref{app:DW}), we modify slightly $\psi$, by composing it with a homogeneous diffeomorphism tangent to identity at infinite order, so that $\psi^\star \tilde \omega = \tilde\omega$.

\bigskip

\noindent{\bf Acknowledgment.}
The authors were partially supported by the grant ANR-15-CE40-0018 of the ANR (project SRGI).
The second author was partially supported by the grant ANR-13-BS01-0007-01 of the ANR (project GERASIC).

The authors warmly thank the referees for their useful remarks and questions that have greatly helped to improve the paper.


\end{document}